\documentclass[11pt]{article}
\usepackage[colorlinks]{hyperref}
\usepackage{hyperref}
\usepackage{color}
\usepackage{graphicx}
\usepackage{graphics}
\usepackage{makeidx}
\usepackage{showidx}
\usepackage{latexsym}
\usepackage{amssymb}
\usepackage{verbatim}
\usepackage{amsmath}
\usepackage{amsthm}
\usepackage{amsfonts}
\usepackage{subcaption} 
\usepackage[abs]{overpic}
\usepackage{xcolor,varwidth}
\usepackage{tikz-cd} 
\usepackage{blindtext}
\usepackage{pgfplots}
\pgfplotsset{compat = newest}

\usepackage{geometry}
\newtheorem{theorem}{Theorem}[section]
\newtheorem{proposition}[theorem]{Proposition}
\newtheorem{remark}[theorem]{Remark}
\newtheorem{claim}[theorem]{Claim}
\newtheorem{question}[theorem]{Question}
\newtheorem{definition}[theorem]{Definition}
\newtheorem{corollary}[theorem]{Corollary}
\newtheorem{lemma}[theorem]{Lemma}
\newtheorem{convention}[theorem]{Convention}

\title{Regularity and persistence in non-Weinstein Liouville geometry\\
 via hyperbolic dynamics}
\author{Surena Hozoori}
\newcommand{\Addresses}{{
  \bigskip
  \footnotesize

Surena Hozoori, \textsc{Department of Mathematics, University of Rochester.}\par\nopagebreak
  \textit{E-mail address}: \texttt{shozoori@ur.rochester.edu}
  
  }}
    \date{}
\begin{document}
\maketitle

\noindent
\begin{abstract}
We explore the construction of non-Weinstein Liouville geometric objects based on Anosov 3-flows, intoduced by Mitsumatsu \cite{mitsumatsu}, in the generalized framework of {\em Liouville Interpolation Systems} and non-singular partially hyperbolic flows. We study the subtle phenomena inherited from the regularity and persistence theory of hyperbolic dynamics in the resulting Liouville structures, and prove dynamical and geometric rigidity results in this context. Among other things, we show that Mitsumatsu's examples characterize 4-dimensional non-Weinstein Liouville geometry with 3-dimensional $C^1$-persistent transverse skeleton. We also draw applications to the regularity theory of the weak dominated bundles for non-singular partially hyperbolic 3-flows.
\end{abstract}

{
  \hypersetup{linkcolor=black}
  \tableofcontents
}

\section{Introduction}

An important strand of questions in symplectic geometry revolves around distinguishing the symplectic condition from other coinciding geometric conditions. The celebrated non-squeezing theorem of Gromov, for instance, distinguishes between the symplectic condition for diffeomorphisms and the existence of an invariant volume form \cite{gromov,symptop}. On the other hand, symplectic and complex geometries appear together in many contexts, and it was initially unknown how the world of the two geometries might in fact differ. The natural compatibility condition between the two geometries defines the class of {\em Stein} structures. It is now well-established that Stein geometry can be formulated  in terms of an equivalent symplectic topological description \cite{ce}. Such topological description requires the symplectic form to be exact, as well as the existence of a gradient-like {\em Liouville flow}. More specifically, one would focus on symplectic manifolds of the type $(W,d\alpha)$, where $\alpha$ is a 1-form with non-degenerate derivative $d\alpha$, called the {\em Liouville form}, and we further assume the unique vector field $Y$ defined by $\iota_Y d\alpha=\alpha$, called the {\em Liouville vector field}, is gradient-like with respect to some Morse function on $W$. The study of such 1-forms is called {\em Liouville geometry} and the gradient-like condition on the Liouville flow is referred to as the {\em Weinstein condition}. Under such circumstances, sometimes called the {\em Weinstein geometry}, we can exploit Morse theory to reach a topological description of such objects using {\em symplectic} handle decompositions \cite{weinstein}. One straightforward consequence of applying Morse theory in this context is the fact that the underlying manifold admits a CW-complex with at most half-dimensional cells. In particular in dimension 4, the Weinstein condition requires the underlying manifold to have the topological type of at most 2. The correspondence between the Stein and Weinstein then implies the same for Stein manifolds \cite{ce}.

It was unknown for a while whether examples of {\em non-Weinstein Liouville geometry} is possible. Note that one can always deform a Liouville vector field to be non gradient-like, and therefore, the non-Weinstein condition refers to Liouville forms which can not be homotoped to be Weinstein. The first examples of non-Weinstein geometry were found by McDuff \cite{mcduff} in 1991 and then, Geiges \cite{geiges2,geiges} and Mitsumatsu \cite{mitsumatsu} extended on the ideas of McDuff. In dimension 4, the idea of these constructions is most generalized in the result of Mitsumatsu in 1994, where he showed that given a 3-manifold $M$ equipped with an {\em Anosov flow} $X^t$, one can construct a 4-dimensional {\em Liouville domain} of the type $(W:=[-1,1]\times M,\alpha)$. Such Liouville domains are necessarily non-Weinstein, since the underlying manifold has the topological type of a closed 3-manifold. Other examples can been constructed \cite{mcduff,bowden} by attaching symplectic handles to the examples of Mitsumatsu. In all of these examples, the non-Weinstein condition is established thanks to the obstruction on the topological type of the underlying manifold. 

Furthermore, new attention has been paid to Liouville dynamics recently, motivated by the prominent role they play in the theory of {\em convex hypersurfaces} in contact geometry \cite{honda,ep,bhh}. More specifically, Liouville dynamics naturally appears as the hypersurface dynamics in the contact geometry of 1-dimension higher. The theory of convex (hyper)surfaces in contact geometry was initiated by Giroux in the early 90s \cite{convex}, and particularly in dimension 3, is proven to be extremely useful in the study of the topological aspects of contact structures (see \cite{massot,contop} for an introduction to the topic). Giroux shows that $C^\infty$-generically, contact geometry in the neighborhood of an embedded surface can be described in terms of a topological set of data on the surface. The Liouville dynamics in this context is simply a volume expanding surface dynamics and after a $C^\infty$-perturbation, such dynamics can be shown to be {\em Morse-Smale}. Then, thanks to a notion of {\em convexity} in this context, an efficient topological description of the contact form in a neighborhood of such embedded surface can be derived (in terms of {\em dividing curves} on the surface). Such topological description, introduced by Giroux, has revolutionized our understanding of contact topology in dimension 3 and in particular, has contributed significantly to many classification results in low dimensions. In higher dimensions however, the situation is much more complicated, partly due to the absence of a general dynamical theory of Liouville flows as one goes beyond dimension 2. Nevertheless, new progress has been made in this direction in recent years, as one would eventually like to apply the ideas of Giroux to high dimensional contact geometry. We might still hope to have a topological description of the contact geometry of a generic embedded hypersurface in a high dimensional contact manifolds. The relevant result in arbitrary dimensions is established by Honda-Huang in 2019 \cite{honda}, where it is shown that $C^0$-generically, an embedded hypersurface inherits a Weinstein
Louville dynamics, i.e. the induced hypersurface Liouville dynamics is gradient-like. However, the fact that such genericity is only $C^0$, leaves ample room for other {\em exotic} Liouville dynamics in high dimensions \cite{honda,huang,jul,mori1,mori2}. This is where non-Weinstein Liouville geometry resides and is left mostly unexplored. As pointed out in \cite{honda}, there has been no systematic study of non-Weinstein Liouville geometry and our tangible knowledge of the matter, to a great degree, relies on the variations and modifications of the Mitsumatsu's construction. Recent developments have started to change our understanding of the Liouville dynamics beyond the Weinstein case \cite{bc,eoy,breen} and this paper aims to contribute one step further in this direction.

On a separate note, it is shown \cite{hoz3} that the construction of Mitsumatsu can be exploited to achieve a characterization of Anosov 3-flows purely in terms of Liouville geometry. Anosov flows are the prototypical examples of chaotic dynamical systems introduced by Anosov \cite{anosov,anosov0}, for which deep results on the persistence and regularity of the invariant manifolds are established \cite{anosov,anosov0,hps,ps}.  These are flows with respect to which, the underlying manifold admits a continuous flow-invariant splitting 
of the type $TM\simeq E^s \oplus E^u \oplus \langle X \rangle$, where $X$ is the generator of the flow and, $E^u$ and $E^s$ are expanding and contracting subspaces of $TM$, called the {\em strong unstable and stable bundles}, respectively. Strong connections to the geometry and topology of the underlying manifold has been explored, in particular in dimension 3, thanks to the use of ({\em taut}) foliation theory in the study of these flow \cite{hyp,bartintro}. More specifically, the {\em weak invariant} foliations tangent to $E^{wu}:=E^u\oplus \langle X \rangle$ and $E^{ws}:=E^s\oplus \langle X \rangle$ are often exploited to develop such geometric theory in dimension 3. However, such foliations are usually of low regularity and therefore, the methods applied in this study are typically very topological in nature.

On the other hand, a contact geometric theory of Anosov 3-flows has been under development in recent years \cite{hoz3,hoz4,sal1,sal2,massoni,clmm,mit2}, which relies on an important observation of Mitsumatsu \cite{mitsumatsu} and Eliashberg-Thurston \cite{et} in the mid 1990s. That is, the generating vector field for any Anosov 3-flow lies in the intersection of a transverse pair of negative and positive contact structures, also called a {\em bi-contact structure}. It turns out that such {\em bi-contact condition} has a dynamical interpretation, as it characterizes a notion of hyperbolicty for flows weaker than being Anosov. Namely, this characterizes the class of {\em projectively Anosov flows}. 
These are flows generated by the non-vanishing vector fields like  $X$ and admitting a continuous $X^t$-invariant {\em dominated splitting} $TM/\langle X \rangle \simeq E\oplus F$. Here, by $F$ {\em dominating} $E$, we mean the action of $X^t$ on $F$ has greater norm compared to $E$.
The observation of Mitsumatsu and Eliashberg-Thurston provides a bridge between the worlds of hyperbolic dynamics and contact geometry, by giving a purely contact geometric characterization of projectively Anosov flows. One can naturally ask whether in this context, Anosov flows have a counterpart in the contact geometric formulation. The affirmative answer to this question was provided in 2020 \cite{hoz3}, based on Mitsumatsu's construction of Liouville domains.
One can show that the Liouville condition for the 1-form defined, on $[-1,1]\times M$, as the linear interpolation of the two underlying contact forms, can carry information about the expansion data of the {\em supported} (projectively Anosov) vector field in the intersection. More precisely, we are interested in the Liouville domains of the form
$$\alpha=(1-s)\alpha_-+(1+s)\alpha_+ \ \ \ \ \ \ \ \ \ on\ \ \  [-1,1]_s\times M,$$
where $\alpha_-$ and $\alpha_+$ are negative and positive contact forms, respectively, on the 3-manifold $M$ with kernels transversely intersecting along the (projectively Anosov) vector field $\langle X\rangle$. For a 3-dimensional interpretation, one can think of this as an interpolation of contact forms $\alpha_-$ and $\alpha_+$ such that the kernel travels through $E$ the {\em dominated} bundle of $X$ (see Section~\ref{3.3} for more thorough discussion). Hence, Mitsumatsu's construction is central to the Liouville geometric theory of Anosov flows in dimension 3, which has facilitated the application of new symplectic geometric ideas to questions in the realm of Anosov dynamics. In particular, modifications of this construction has been applied to study a non-compact version of such geometric model as Liouville {\em manifolds} on $\mathbb{R}\times M$, by considering exponential interpolations \cite{mnw,massoni}, i.e. Liouville manifolds of the form
$$\alpha=e^{-s}\alpha_-+e^s\alpha_+ \ \ \ \ \ \ \ \ \ \ \ \ \ \ \ \ \ \ \ \ \  on\ \ \  \mathbb{R}_s\times M,$$ 
where $\alpha_-$ and $\alpha_+$ are as before.
This non-compact theory has been exploited recently in order to define a host of new symplectic geometric invariants of Anosov 3-flows based on the invariants of the resulting non-Weinstein Liouville manifolds \cite{clmm}.

The goals of this paper are three-fold; (1) We want to unify the previous 3-dimensional variations of Mitsumatsu's construction in the generalized framework of {\em Liouville Interpolation Systems} (also denoted by LIS s here), which geometrically is meant to encapsulate the weakest notion of interpolation between contact forms as seen in the construction, and rectify the apparent differences between different models. In particular, we will compare the two most commonly used, i.e. linear vs. exponential, models. Moreover, we often generalize our arguments to non-singular partially hyperbolic 3-flows, a generalization of Anosov 3-flows for which the Mitsumatsu's construction still works; (2) Study the Liouville geometry of the resulting objects and see the subtle regularity and persistent theoretic aspects of hyperbolic dynamics is inherited in the Liouville geometric theory of these flows. This helps us establish dynamical and geometric rigidity of the Liouville structure, showing that the Liouville geometry in this case is strongly determined by the supported flow, i.e. the {\em skeleton dynamics}; (3) Extend our analysis of the LIS to Liouville manifolds with persistent 3-dimensional skeleton and show that under assumptions, non-Weinstein Liouville geometry is characterized by the examples of Mitsumatsu. Beside these, our generalized setting allows us to draw corollaries about the regularity of the weak dominated bundle, generalizing the classical results on $C^1$-regularity of the weak invariant bundles for Anosov 3-flows \cite{hps,reg}.

\vskip0.5cm

\noindent\fbox{%
    \parbox{\textwidth}{%
\textbf{Assumptions:} In this paper, unless stated otherwise, $M$ is a closed oriented connected 3-manifold and $X$ is vector field on $M$. Here, the flow generated by a $C^\infty$ vector field $X$ is denoted by $X^t$. Also, we assume all (projectively) Anosov flows to be {\em oriented}, i.e. have oriented invariant bundles. This condition is always satisfied, possibly after going to a double cover of $M$.
    }%
}
\vskip0.5cm

\begin{remark}
Refinements of the results in this paper can be given in terms of the regularity of the flows as well as the Liouville geometry. However, for the sake of more straight forward statements, we here assume the flows to be $C^\infty$, a common assumption thanks to the {\em $C^1$-structural stability} of Anosov flows. Subsequently, we often take our Liouville forms to be $C^\infty$, unless stated otherwise, exploiting the fact that the Liouville geometry of these flows, unlike their underlying invariant foliations, can be described in the same regularity as of the flows.
\end{remark}

All variations of the Mitsumatsu's examples rely on an interpolation between a pair of negative and positive contact forms $(\alpha_-,\alpha_+)$, whose kernels intersect transversely along an Anosov vector field $X$. In this paper, we want to show that in this situation, the Liouville geometry is strongly determined by such supported vector field and is independent of all the auxiliary data related to interpolation construction. 
Therefore, we want to exploit this kernel interpolation idea in a very general sense of it.
We consider Liouville forms of the type
$$\begin{cases}
\alpha:=\lambda_-\alpha_-+\lambda_+\alpha_+ \ \ \ \ \ \ \ \text{on}\ \ I_s\times M \ \ \ \text{for some interval} \ \ I_s\subset \mathbb{R}\\
\alpha_-,\alpha_+: \text{negative and positive contact forms with transverse kernels on}\ \ M\\
\lambda_-,\lambda_+:I_s \times M\rightarrow \mathbb{R}_{>0} \text{ are positive functions}
\end{cases},$$
where considering the fact that $\ker{[\alpha]}=\ker{[\alpha_++(\frac{\lambda_+}{\lambda_-})\alpha_-]}$, the interpolation of the contact forms boils down to enforcing the geometric condition
$$\partial_s\cdot \frac{\lambda_+}{\lambda_-}\neq 0\ \ \Longleftrightarrow \ \ \alpha\wedge \mathcal{L}_{\partial_s}\alpha \neq 0 \ \ \Longleftrightarrow \ \  Y\pitchfork \partial_s,$$
where $Y$ is the Liouville vector field of $\alpha$. Respecting the orientations, we assume $\partial_s\cdot \frac{\lambda_+(.,x)}{\lambda_-(.,x)}> 0$ for any $x\in M$, i.e. $\frac{\lambda_+(.,x)}{\lambda_-(.,x)}$ is a local oriented reparametrization of the interval $I_s$. This geometric description is in fact enough for most purposes, as most of the action happens in near the skeleton. In this paper, we consider the two categories of {\em Liouville domains} in the compact setting (assuming positive contactness of the boundary) and {\em Liouville manifolds} in the non-compact one (assuming the completeness of the Liouville flows). There are other boundary conditions which are well motivated in the literature, e.g. the class of ideal Liouville domains \cite{ideal} or variations of singular symplectic geometry \cite{sing,poiss}, and we expect the culprit of our analysis to remain valid in a broader sense.

 When $I_s=[N_-,N_+]$ is a compact interval, it is natural to assume the positive contactness of $\alpha$ on $\partial(I_s\times M)$, which defines what we refer to as a Liouville domain. Note that this requires $\alpha|_{s=N_-}$ and $\alpha|_{s=N_+}$ are negative and positive contact forms on $M$, respectively. This in particular includes the linear construction originally used by Mitsumatsu and further in \cite{mitsumatsu,hoz3,hoz4}.

For the non-compact setting corresponding to $I_s=\mathbb{R}$, it is natural for us to consider the class of Liouville manifolds, i.e. assume the completeness of the resulting Liouville flow. Such category of objects satisfy a convenient deformation theory, thanks to an application of the {\em Moser technique}. However, it is not always arithmetically easy to determine whether a flow is complete. We therefore enforce the condition of $s\mapsto \ln{\frac{\lambda_+(s,x)}{\lambda_-(s,x)}}$ being a positive reparametrization of $\mathbb{R}$, for any $x\in M$, which implies completeness of the Liouville flow (as it will be proved in Section~\ref{4}). This non-compact model in particular includes the exponential model used in \cite{mnw,clmm,massoni}.

We record this information (in compact or non-compact case) as $(\alpha_-,\alpha_+)_{(\lambda_-,\lambda_+)}$ and call it a (compact or non-compact) {\em Liouville interpolation system (LIS)}. Considering such boundary conditions, we denote the space of such objects by
$$\begin{cases}
\mathcal{LIS}_c(M): \ \ \ \text{space of compact Liouville interpolation systems on }M\\
\mathcal{LIS}(M): \ \ \ \text{space of (non-compact) Liouville interpolation systems on }M
\end{cases}.$$


This contains both the linear and exponential models previously considered in the literature and in fact, we allow the interpolation functions to depend on $x\in M$. We will later see that the selected interpolation regime does not affect the Liouville geometry.
We primarily work in the non-compact in order to enjoy the deformation theory needed to study of the symmetries of the Liouville structure. Therefore, unless stated otherwise, LIS s are non-compact in the following. We will later relate the compact and non-compact settings via strict Liouville emebeddings (see Section~\ref{5}).

Furthermore, it is important for us that the construction of Mitsumatsu in fact works in a class of dynamics larger than Anosov 3-flows. These are {\em non-singular partially hyperbolic flows} which are characterized by a continuous flow invariant splitting $TM\simeq E\oplus E^u$, where $E^u$ is a 1-dimensional expanding line bundle (strong unstable bundle), and $E$ is a 2-plane field containing $X$, which is {\em dominated} by $E^u$ (center bundle). The Liouville domain can still be constructed for such flows, thanks to the expansion of $E^u$.

We start our study by revisiting the previously known constructions, in order to recontextualize and generalize them to the broad interpolation framework of LIS s, and the category of non-singular partially hyperbolic flows. Similar characterization for Anosov flows is possible of course, if one notes that $X$ is Anosov, if and only if, both $X$ and $-X$ are non-singular partially hyperbolic vector fields. This generalizes the results of \cite{hoz3,massoni} for the linear and exponential {\em Liouville pairs}.

\begin{theorem}\label{introgeneralchar}
The followings are equivalent:

(1) $X^t$ is a non-singular partially hyperbolic 3-flow with the splitting $TM \simeq E\oplus  E^u$;

(2) $X^t$ admits some supporting LIS $(\alpha_-,\alpha_+)_{(\lambda_-,\lambda_+)}\in \mathcal{LIS}(M)$;

(3) $X^t$ admits some supporting compact LIS $(\alpha_-,\alpha_+)_{(\lambda_-,\lambda_+)}\in \mathcal{LIS}_c(M)$
\end{theorem}

The key observation in the above theorem is the following. The bi-contact condition implies the existence of a dominated splitting $TM/\langle X\rangle \simeq E\oplus F$ and the Liouville condition at the skeleton gives the absolute expansion of $F$ required by the partial hyperbolicity of $X$, i.e. the existence of a lift of the dominated splitting to an invariant splitting of the form $TM\simeq E\oplus E^u$. Results of this paper show that the Liouville geometry of $(\alpha_-,\alpha_+)_{(\lambda_-,\lambda_+)}\in \mathcal{LIS}(M)$ is strongly determined by the supported (positive reparametrization class of) flows.

Here, we remark that the considering the class of non-singular partially hyperbolic 3-flows as a generlization of Anosov flows is quite natural from a Liouvile geometric point of view. However, we will see that subtle failures in the regularity and persistence theory of hyperbolic dynamics in this larger class of flows can have consequences in the non-Anosov examples. This distinction is important as the {\em Liouville homotopic invariants} of our constructed objects are invariant under homotopy through the class of non-singular partially hyperbolic flows, and it is not well understood how the space of such flows of is related (in terms of the homotopy type) to the space of Anosov flows (see Question~\ref{qinv}). As sketched in \cite{et}, non-Anosov examples of non-singular partially hyperbolic 3-flows can be constructed using the {\em DA (Derived from Anosov) deformation}, an operation of blowing up a periodic orbit of an Anosov flow to a fully repelling orbit, while still preserving a dominated splitting. We will revisit this construction in Section~\ref{3.5} to show that such deformation can be constructed via a {\em bi-contact homotopy}.


Early on in our study, we will notice that there exists a strictly exact Lagrangian foliation $\mathcal{F}^{wn}$, which we call the {\em weak normal foliation}, tangent to the plane field $E^{wn}:=\langle \partial_s,X\rangle$ which is invariant under the Liouville flow and plays a significant role in the theory. This foliation is fixed when we fix the flow and therefore, the isotopies and maps preserving such exact Lagangian foliation can be employed to study the symmetries of LIS s. Later on (in Section~\ref{7}), we will see that the existence of such exact Lagrangian foliation is not a coincidence from our LIS construction, but in fact is inherited from the persistence features of the Liouville skeleton. 

For the Liouville form induced from an arbitrary LIS, we can explicitly compute the Liouville dynamics. In particular, we show that the skeleton is a section of the projection $\pi:I_s\times M \rightarrow M$ and can study its repellence from the geometric data. 

\begin{theorem}\label{introdynrig}
(Dynamical rigidity) Assume $X^t$ be a non-singular partially hyperbolic flow on $M$ admitting the dominated splitting $TM\simeq E\oplus E^{u}$, the weak dominated bundle $E$ is $C^k$ for some $0<k$, both weak invariant bundles are $C^l$ for some $0<l\leq k$, $(\alpha_-,\alpha_+)_{(\lambda_-,\lambda_+)}\in\mathcal{LIS}(M)$ is a supporting LIS and $Y$ is the corresponding Liouville vector field on $\mathbb{R}\times M$. Then,

(1) $Skel(Y)$ is a $C^k$ section of $\pi:\mathbb{R}\times M \rightarrow M$ given as
$$Skel(Y)=\bigg\{(\Lambda_s(x),x)\in \mathbb{R}\times M \ \ \text{ where } \ \ \Lambda_s:M\rightarrow \mathbb{R} \ \ \text{ is determined by }$$
$$\ker{\big[ \lambda_-(\Lambda_s(x),x)\alpha_-+\lambda_+(\Lambda_s(x),x)\alpha_+\big]}=E \bigg\}.$$ More specifically, $Skel(Y)$ is exactly as regular as $E$. When $X$ is Anosov, the skeleton is exactly as regular as the weak stable bundle $E^{ws}$ and in particular, it is always a $C^{1+}$ section of $\pi$;
 
 (2) $\pi_*(Y|_{Skel(Y)})\subset TM$ is a synchronization of $X$;
 
 (3) $Y$ preserves a $C^\infty$ exact Lagrangian foliation $\mathcal{F}^{wn}$ containing the flow lines of $Y$ expands a differentiable transverse measure at $Skel(Y)$. Furthermore, $Y$ is normally hyperbolic at $Skel(Y)$, if and only if, $X$ is Anosov;
 
(4) there exists a $C^l$ 1-dimensional foliation $\mathcal{F}^n$ inside $\mathcal{F}^{wn}$ which is transverse to and invariant under the flow of $Y$. In particular when $X$ is Anosov, $\mathcal{F}^n$ is $C^{1+}$ and coincides with the strong repelling foliation whose existence is implied by the normal hyperbolicity given in (3).\\


\end{theorem}

A remarkable fact here is that by (1) in the above theorem, the skeleton is exactly as regular as the weak dominated bundle of the supported flow. Therefore, as we will later see (in Section~\ref{8}), one can study the Liouville dynamics near such skeleton in order to investigate the regularity of weak bundles. More specifically, the dynamics near the skeleton implies its persistence, up to a certain regularity. This is more easily seen in the Anosov cases, where by {\em normal hyperbolicity} (using (3) in the above), $C^1$-persistence is straight forward. However, we will carefully extend this to lower regularities, using (1) in the above theorem, i.e. realizing the skeleton as a graph and applying the more classical tool of the {\em $C^r$-section theorem}. In particular, a lower bound on the order of Hölder regularity can be derived from the expansion data.

\begin{corollary}\label{introperscor}
(Persistence of skeleton) $Skel(Y)$ is $C^1$-persistent under $C^2$-deformations of a supported Anosov vector field $X$. More generally,  $Skel(Y)$ is $C^k$-persistent for some $k>0$ under $C^2$-deformations through arbitrary non-singular partially hyperbolic flows.
\end{corollary}

Theorem~\ref{introdynrig} provides a very explicit description of the Liouville flow, in particular, via (2) and (4) above, i.e. the skeleton dynamics is simply a synchronization of the supported flow and the existence of the strong normal foliations determines the normal dynamics. Among other things, this implies that the Liouville dynamics is unique up to $C^l$-conjugacy, where $l$ is regularity of the weak bundles and in particular, $k>1$ in the Anosov case. In fact ,we can explicitly show this by proving the conjugacy of the Liouville flow to the linearization at its skeleton (Theorem~\ref{introlin}). We will later improve this to a smooth conjugacy using the Moser technique (Corollary~\ref{introdynuniq2}). 

The standard application of the Moser technique to Liouville manifolds (see Lemma~\ref{moserclassic}) implies that after applying an isotopy of $\mathbb{R}\times M$, we can assume the underlying symplectic structure to be fixed under the homotopies through Liouville manifolds. Therefore, any rigidity of the Liouville dynamics should have a geometric counterpart on its dual, the Liouville form.
It turns out that a novel use of the Moser technique helps us recover the Liouville form strictly (and not just up to homotopy) from the supported flow. This promotes the construction of Mitsumatsu to a 1-to-1 correspondence, via extracting the skeleton dynamics, between the appropriate equivalence classes of non-singular partially hyperbolic flows and Liouville manifolds.

\begin{theorem}\label{intro1to1} (Geometric rigidity)
There is a 1-to-1 correspondence between positive reparametrization classes of non-singular partially hyperbolic vector fields, up to $C^\infty$-conjugacy, and Liouville forms induced from some LIS on in $\mathcal{LIS}(M)$, up to strict Liouville equivalence, i.e.

$$\bigg\{\substack{\text{Positive reparametrization class of} \\  \text{non-singular partially hyperbolic flows} \\  \text{up to conjugacy}} \bigg\} \mbox{\Large$\overset{\text{1-to-1}}{\longleftrightarrow}$}
\bigg\{\substack{\text{Liouville forms induced from some LIS on $\mathbb{R}\times M$} \\  \text{up to strict Liouville equivalence}} \bigg\}.$$
\end{theorem}

This means that Liouville geometry only depends on the supported flow and is independent of all the other choices we made in our construction. This also sheds light on how we can naturally think of the compact LIS s in this context. In particular, this implies that the distinction between the previously introduced linear and exponential models boils down to compactness and not Liouville geometry. In the following $L(\alpha_-,\alpha_+)_{(\lambda_-,\lambda_+)}$ is the Liouville form induced from the LIS $(\alpha_-,\alpha_+)_{(\lambda_-,\lambda_+)}$.

\begin{corollary}\label{introlinemb}
Fixing a non-singular partially hyperbolic flow (up to positive reparametrization), for any supporting compact LIS $(\alpha_-,\alpha_+)_{(\lambda_-,\lambda_+)}$ and any supporting (non-compact) LIS $(\bar{\alpha}_-,\bar{\alpha}_+)_{(\bar{\lambda}_-,\bar{\lambda}_+)}$, there exists a strict Liouville embedding
$$i:([N_-,N_+]\times M,L(\alpha_-,\alpha_+)_{(\lambda_-,\lambda_+)})\rightarrow (\mathbb{R}\times M, L(\bar{\alpha}_-,\bar{\alpha}_+)_{(\bar{\lambda}_-,\bar{\lambda}_+)}),$$
an embedding $i:[N_-,N_+]\times M \rightarrow \mathbb{R}\times M$ satisfying
$$i^*L(\bar{\alpha}_-,\bar{\alpha}_+)_{(\bar{\lambda}_-,\bar{\lambda}_+)}=L(\alpha_-,\alpha_+)_{(\lambda_-,\lambda_+)}.$$
This is in particular true for any linear and exponential Liouville pairs supporting the same (positive reparametrization class of) flows.
\end{corollary}


This also refines the regularity of the uniqueness claim for the Liouville dynamics, compared to our conclusion from Theorem~\ref{introdynrig}.

\begin{corollary}\label{introdynuniq2}
Let $X^t$ be a non-singular partially hyperbolic 3-flow. The Liouville vector field induced from a $C^\infty$ supporting LIS is unique up to $C^\infty$-conjugacy.
\end{corollary}

As pointed out by Massoni \cite{massoni}, an important step in exploiting these Liouville manifolds to derive invariants of the underlying flows is to show that the space of these objects, i.e. LIS s, forms a fibration over the space of the supported flow. In other words, a topological version of the geometric rigidity theorem above is needed if one wants to study the homotopy classes of flows and their invariants.
We follow similar ideas as in \cite{massoni} in order to derive our fibration result in this context. In the following, $S \mathcal{PHF}(M)$ is the space of non-singular partially hyperbolic flows on $M$, up to positive reparametrization, and the map defining the fibration sends a LIS to its skeleton dynamics.

\begin{theorem}\label{introfib}(Fibration)
The map
$$\begin{cases}
\mathcal{LIS}(M) \rightarrow S \mathcal{PHF}(M)  \\
(\alpha_-,\alpha_+)_{(\lambda_-,\lambda_+)}  \mapsto [\pi(F(\alpha_-,\alpha_+)_{(\lambda_-,\lambda_+)} |_{\Lambda_s})]
\end{cases}$$
is a Serre fibration of the space of Liouville interpolating systems over the space of non-singular partially hyperbolic flows up to positive reparametrization.
\end{theorem}

Thanks to our explicit construction, filtrations of the above fibration are available if one want to equip the flow any useful geometric information, like the {\em synchronization} (equivalently, the information of an expanding norm on the strong unstable bundle, up to constant scaling) or the skeleton graph (see Remark~\ref{fibfilter}).
The following is the main takeaway of our fibration result.

\begin{corollary}
The Liouville homotopic invariants of LIS s are invariants of the supported flow, up to homotopy through non-singular partially hyperbolic flows.
\end{corollary}

It is important for us to establish Theorem~\ref{introfib} in low regularity as well. More precisely, we explicitly construct our fibration at a regularity depending on the weak invariant bundles. In regularity lower than $C^1$ however, the Moser technique fails and we rely on constructing explicit strict Liouville isotopies (via elementary isotopies introduced in Section~\ref{4}). As a result, we achieve an explicit formulation of the Liouville dynamics via its linearization at the skeleton. 

\begin{theorem}\label{introlin}
(Linearization) Let $\alpha=L(\alpha_-,\alpha_+)_{(\lambda_-,\lambda_+)}$ be the Liouville form induced from an LIS supporting a non-singular partially hyperbolic $X$ with $C^k$ weak invariant bundle (with $k\geq 1$ when $X$ Anosov). Then, there is a $C^k$ strict Liouville equivalence $\psi:\mathbb{R}\times M \rightarrow \mathbb{R}\times M$ between $(\mathbb{R}\times M,\alpha)$ and its linearization. In particular, $\psi_*(Y)$ is the linearization of $Y$, the Liouville vector field of $\alpha$, at its skeleton.
\end{theorem}

An application of controlling the regularity in the linearization theorem is the fact that lower bounds on the regularity of {\em strong normal foliations} can be achieved. In particular, we achieve $C^1$-regularity of these foliations in the presence of normal hyperbolicity at the skeleton, i.e. the Anosov case (see (4) in Theorem~\ref{introdynrig}).

\begin{corollary}
The strong normal Lagrangian bundle $E^n$ of any LIS at its skeleton is tangent to a 1-dimensional $C^k$ foliation $\mathcal{F}^n$ of $\mathbb{R}\times M$, contained in the weak normal Lagrangian foliation $\mathcal{F}^{wn}$, whenever the weak bundles of the supported non-singular partially hyperbolic flow are $C^k$. In particular, $\mathcal{F}^n$ is $C^1$, whenever $X$ is Anosov.
\end{corollary}

Inspired by the persistence of the skeleton in the LIS construction, i.e. Corollary~\ref{introperscor}, we study the persistence features of 3 dimensional skeletons in general Liouville manifolds (or domains). In the category of vector fields, $C^1$-persistence of an invariant submanifold and its normal hyperbolicty are known to be equivalent as a result of an important chapter in the history of hyperbolic dynamics. More precisely, Hirsch-Pugh-Shub \cite{hps} establishes $C^1$-persistence of normally hyperbolic invariant submanifolds in 1970, revisiting the classical {\em graph transformation} methods of Hadamard and Perrone. The celebrated result of Mañe \cite{mane} in 1978 then proves the converse of the persistent theorem by showing that any $C^1$-persistent invariant submanifold is normally hyperbolic, completing the geometric characterization of $C^1$-persistence in terms of normal hyperbolicity. One should note that $C^1$-persistence captures when an invariant set survives, {\em as a manifold}, under deformations. Understanding $C^0$-persistence, i.e. persistence of invariant sets as topological sets, is known to be considerably more subtle problem \cite{floer1,floer2,c0}. In the same spirit, while Theorem~\ref{introdynrig} introduces Liouville interpolation as a sufficient condition for normally hyperbolicity and persistent 3-dimensional skeletons, we would like to characterize when such conditions are present. Our analysis of normal hyperbolicity at a 3-dimensional Liouville skeleton results in the following structure theorem. This boils down to a complete characterization, if we further assume transversality of the skeleton and the kernel of the Liouville form. More precisely, adding such transversality assumption, the underlying Liouville manifold can be shown to be strictly Liouville equivalent to the Mitsumatsu's examples. To the best of our knowledge, this provides the first classification result of any kind in non-Weinstein Liouville geometry.

We notice in part (b) of the following that as usual, Anosovity and $C^1$-regularity appear together.
However, the interplay with Liouville geometry is required to enhance the standard regularity theoretic arguments and achieve the geometric model of a LIS. 

\begin{theorem}\label{intropers}
Suppose $(W^4,\alpha)$ is Liouville manifold with an oriented $C^1$-persistent 3-dimensional skeleton $\Lambda$ and $\alpha$ is nowhere vanishing. Let $E^n$ be the invariant repelling normal bundle at $\Lambda$. Then,

(a) $Y|_\Lambda$ is an Axiom A flow, where $\Lambda_T=\{ \ker{\alpha}=T\Lambda \}$ is collection of a finite number of repelling periodic orbits of $Y$, and for some tubular neighborhood of $\Lambda_T$, $\Lambda/N(\Lambda_T)$ is a hyperbolic plug whose non-empty core coincides with $\Lambda_L=\{ E^n\subset \ker{\alpha}\}$.

(b)  If we furthermore have $T\Lambda \pitchfork \ker{\alpha}$, then $Y|_\Lambda$ is a synchronized Anosov vector field and $(W^4,\alpha)$ is $C^1$-strictly Liouville equivalent to a a Liouville form induced from a LIS supporting $Y|_\Lambda$. 
\end{theorem}

This should be seen as a result distinguishing the Liouville condition for flows from other coinciding geometric conditions. More precisely, given any arbitrary flow $X$ on $M$, we can always extend $X$ it to a thickening $(-\epsilon,\epsilon)\times M$ by adding a sufficiently repelling normal direction with a $C^1$-persistent skeleton dynamically equivalent to $(M,X)$ (after a positive reparametrization of $X$, we can then set the divergence of the extended flow to be as desired). Theorem~\ref{intropers} however states that
in the realm of Liouville geometry, this is only possible if $X$ satisfies certain necessary conditions. For instance, such $(M,X)$ would  necessarily admit a hyperbolic invariant set $\Lambda_L\subseteq M$. We are in general interested in pheneomena distinguishing a Liouville flow from a general volume expanding flow (see Question~\ref{qgeneral} and \ref{qlioupers}).

As a consequence, combined with the geometric rigidity established in Theorem~\ref{1to1}, we have achieved a characterization of Anosov dynamics as the skeleton dynamics for 3-dimensional transverse $C^1$-persistent Liouville skeletons. Various extensions of the above structure theorem can be further investigated, where one can hope to characterize $C^0$-persistent skeletons under assumptions, possibly in terms of partial hyperbolicity (see Question~\ref{qtrans}, ~\ref{qsing} and~\ref{qc0}).

\begin{corollary}
There exists the following 1-to-1 correspondence:
$$\bigg\{\substack{\text{Positive reparametrization classes of} \\  \text{Anosov flows} \\  \text{up to conjugacy}} \bigg\} \mbox{\Large$\overset{\text{1-to-1}}{\longleftrightarrow}$}
\bigg\{\substack{\text{Liouville forms on $\mathbb{R}\times M$ with $C^1$-persistent} \\  \text{3-dimensional skeleton $\Lambda$ with $\ker{\alpha}\pitchfork T\Lambda$} \\  \text{up to strict Liouville equivalence}} \bigg\}.$$
\end{corollary}

The study of persistence in the non-Anosov case requires extra care since in the absence of normal hyperbolicity's rate condition, the regularity of the weak dominated bundle might fail to be $C^1$ and therefore, low regularity should be dealt with. The example of Eliashberg-Thurston \cite{et} shows that low regularity (and worse, lack of unique integrability) indeed happens. But in the following, we show that for non-singular partially hyperbolic flows which are {\em geometrically close} to being Anosov, the existence of the strong normal bundle can still be derived. Dropping these conditions further, we can ask about the existence of {\em exotic Liouville pairs}, i.e. 3-dimensional embedded transverse Liouville skeleton, for which the Liouville dynamics nearby is not conjugate to its linearization at the skeleton (see Question~\ref{qgenrep} and \ref{qexotic}).

\begin{theorem}\label{nearanosov}
Let $\Lambda \subset (W,\alpha)$ be the 3 dimensional $C^1$ embedded Liouville skeleton with $\ker{\alpha}\pitchfork T\Lambda$. Also assume that at $\Lambda$, the Liouville vector field $Y$ expands $TW/T\Lambda$ with the rate $r_n>\frac{1}{2}$. Then, there exists an invariant bundle $E^n\subset \ker{\alpha}$ such that $E^n\pitchfork T\Lambda$. 
In this case, the Liouville flow is conjugate to its linearization at $\Lambda$.
\end{theorem}

The dynamical rigidity result of Theorem~\ref{introdynrig} provides a Liouville geometric description of the weak dominated bundle of a non-singular partially hyperbolic flow as the graph representing the Liouville skeleton. Therefore, one can revisit the regularity theory of these invariant bundles in terms of the normal expansion at such repelling graphs (this includes both weak stable and unstable plane bundles in the Anosov case). The upshot is that using this Liouville geometric interpretation, one can then apply the more classical theory of {\em graph transformations} -whose roots famously goes back to the work of Hadamard and Perron-, recover a new proof for the weak invariant bundles of an Anosov 3-flow being $C^{1+}$ (a classical fact from the regularity theory of Anosov flows \cite{hps,reg,reg2}), and generalize it to new lower bounds for the regularity of the weak dominated bundles for non-singular partially hyperbolic 3-flows. Our method furthermore shows that these regularity estimates behave well under deformations. 

\begin{quote}
\centering
{\em "Every five years or so, if not more often, someone ‘discovers’ the theorem of Hadamard and Perron proving it either by Hadamard’s method or Perron’s. I myself have been guilty of this." Anosov 1967 \cite{anosov}}
\end{quote}

This is our Liouville geometric take on the quote of Anosov.

\begin{theorem}\label{regex}
Let ${\Pi}^k(M)$ be the space of $C^k$ plane fields and $\mathcal{PHF}(M;B_s>k)$ be the space of non-singular partially hyperbolic vector fields like $X$ with the splitting  $TM \simeq E\oplus E^{u}$, and the bunching constant satisfying $B_s>k$. Then, the map defined by $$\begin{cases}\mathcal{D}:\mathcal{PHF}(M;B_s>k)\rightarrow \Pi^k(M) \\
\mathcal{D}(X):=E \end{cases},$$
sending a non-singular partially hyperbolic flow to its weak dominated bundle as a plane field,
is well-defined and continuous. In particular, the weak stable bundle $E^{ws}$ $C^1$-varies as one deforms $X$ through Anosov vector fields.
\end{theorem}

In the above, the map $\mathcal{D}$ being well-defined refers to the non-trivial fact that the weak dominated bundle of the non-singular partially hyperbolic $X$ is $C^k$ under the assumptions (this is showing that the regularity lower bounds of Hasselblatt \cite{reg} for the weak invariant bundles of Anosov 3-flows can be extended to this category). Furthermore, the continuity of $\mathcal{D}$ refers to the fact that thanks to Liouville geometry, the persistence of $C^k$-sections is translated into the weak invariant bundle $C^k$-depending on $X$, another non-trivial fact concluded here. Note that one can take $k>1$ in the Anosov case.

To conclude our study, we record some elementary observations about other related geometric objects, more specifically, the Lagrangian foliations and Hamiltonian flows. The Liouville geometric invariants of Anosov 3-flows introduced in \cite{clmm} rely heavily on understanding such objects and in particular, how they interact with the Liouville skeleton whose dynamics we have explored in this paper. Here, we do not take up the task of studying these objects in depth and simply conclude with pointing out a few remarks. We have already discussed that the weak normal foliations $\mathcal{F}^{wn}$, which are invariant strictly exact Lagrangian foliations, i.e. $\alpha|_{T\mathcal{F}^{wn}}=0$, play a very important role in the theory and in particular, Theorem~\ref{intropers} implies that the their existence is rooted in the $C^1$-persistence of the Liouville skeleton. The significance of $\mathcal{F}^{wn}$ is thoroughly discussed in Section~\ref{4}. We further note that when the skeleton is $C^1$ in the Anosov case, the weak stable foliation of the skeleton dynamics yields an exact Lagrangian foliation of the skeleton. We finally remark that in our Liouville geometric model, the Reeb flows of the supporting contact forms correspond to the Hamiltonian dynamics on hypersrfaces away from the skeleton.

\begin{theorem}
Suppose $(\alpha_-,\alpha_+)_{(\lambda_-,\lambda_+)}\in \mathcal{LIS}(M)$ supports a non-singular partially hyperbolic flow $X$.


(1) when $X$ is Anosov, the Liouville skeleton $\Lambda_s$ is foliated by a $C^1$ strict exact Lagrangian foliation;

(2) the Reeb flows for any supporting $(\alpha_-,\alpha_+)$ can be realized as the Hamiltonian flows on a pair of energy hypersurfaces inside $(\mathbb{R}\times M,L(\alpha_-,\alpha_+)_{(\lambda_-,\lambda_+)})$.
\end{theorem}

\vskip0.5cm

\textbf{Organization of the paper:} The audience of this paper is mostly assumed to be non-experts in dynamical systems. So, we start Section~\ref{2} with bringing the necessary background from hyperbolic dynamics. This includes basic concept from Anosov dynamics, as well as other notions of hyperbolicity which will appear in this paper. We will also discuss elements from the theory of invariant bundles and regularity theory, as needed. In Section~\ref{3}, we revisit the Liouville geometric theory of Anosov 3-flows, which has been under construction in recent years. In particular, we will discuss Mitsumatsu's construction which is central in this work. We will then introduce the generalized framework of Liouville Interpolation Systems and extend our theory to the wider class of non-singular partially hyperbolic 3-flows. We also give explicit constructions (bi-contact DA deformations) to indicate that such class of dynamics is in fact bigger than Anosov 3-flows. In Section~\ref{4}, we bring in our main computations regarding the Liouville dynamics in the LIS model, and use it to explore the Liouville dynamics in our construction. We will see that our computations yield a complete understanding of the Liouville dynamics and prove rigidity results in this regard. In Section~\ref{5}, we recover the similar rigidity phenomena from a geometric viewpoint, study the applications to the Liouville embedding problem and provide a fibration theorem, which can be used as a basis for further defining symplectic geometric invariants of Anosov flows from Liouville geometry. In Section~\ref{6}, we study the linearizations of the Liouville flow at its skeleton. To do so, we need to adopt a low regularity approach, which exploits explicit isotopies rather than the Moser technique. In Section~\ref{7}, we try to prove an inverse theorem for the persistence part of our dynamical rigidity theorem, and in particular show that Mitsumatsu's examples charachterize of Liouville geometry with $C^1$-persistent 3-dimensional transverse Liouville skeleton. In Section~\ref{8}, we will see that our dynamical rigidity theorem provides a Liouville geometric approach in the study of weak invariant bundles of a general non-singular partially hyperbolic 3-flow. We will use this viewpoint to generalize a previous well known results on the regularity of the weak invariant bundles of Anosov 3-flows. In Section~\ref{9}, some remarks about the related geometric objects have been made. These are Lagrangian foliations and Hamiltonian flows defined naturally in our setting. Finally in Section~\ref{10}, we bring in some questions, regarding Liouville dynamics mainly, motivated by or discussed along the way in our analysis.

\vskip0.5cm

\textbf{ACKNOWLEDGEMENTS:} We are grateful to Thomas Massoni for conversations in 2022 which sparked some of the questions addressed in this paper. This work has benefited from many fruitful conversations with Boris Hasselblatt, Joe Breen, Federico Salmoiraghi, Thomas Barthelme, Ko Honda, Kai Cieliebak and Julian Chaidez. This work was impossible without the love and support of Armita.

\section{Elements from hyperbolic dynamics and invariant manifolds}\label{2}

The goal of this section is to provide a rough background on the parts of the vast theory of hyperbolic dynamics which will appear in this manuscript. This paper centers around the 4 dimensional Liouville geometric objects one can construct given an Anosov 3-flow, based on which a characterization of such class of dynamical systems can be presented \cite{mitsumatsu,hoz3} (see Section~\ref{3.3}). Therefore, 3 dimensional Anosov flows are the focus of this paper and we begin with reviewing basic concepts in their theory. However, there are other notions of hyperbolicity in dynamical systems which will appear throughout this paper. Two important generalizations of Anosov dynamics in dimension 3 are important for us. Projectively Anosov flows (also called flows with dominated splitting) are the cornerstones of the connections of the theory and the world of contact and symplectic geometry (see Section~\ref{3.2}), and non-singular partially hyperbolic 3-flows will appear as an important middle ground between Anosovity and projective Anosovity. The generalizations of such ideas to higher dimensions -for us, dimension 4-, as well as the notion of normal hyperbolicity, which captures the persistence features of dynamics, will be important in this paper. Hence, their introduction in Section~\ref{2.2} follows. After basic introduction, we will also discuss elements from the theory of invariant bundles and their regularity, mainly borrowing from the classical formulation of Hirsch-Pugh-Shub~\cite{hps} and its important refinement by Hasselblatt \cite{reg,reg2}.

A beautiful and comprehensive reference on the following topics of this section can be found in~\cite{hyp}.

\subsection{Anosov 3-flows, adapted norms and synchronizations}\label{2.1}

We start with the definition of uniformly hyperbolic invariant sets.

\begin{definition}
Suppose $X$ be a $C^k$ ($k\geq 1$) vector field on a manifold of arbitrary dimension $M$, whose flow preserves a $C^1$ compact invariant set $\Lambda$. The invariant set $\Lambda$ is called {\em uniformly hyperbolic} if for any $p\in \lambda$, we have a continuous and $X$-invariant splitting $TM|_\Lambda\simeq \langle X \rangle \oplus E^{u}\oplus E^{s}$ such that for any $t\in \mathbb{R}$, we have
$$\begin{cases}
 ||X^t_*(u) ||\geq Ae^{Ct} ||u|| \ \ \text{ for any } u\in E^{u} \\
  ||X^t_*(v) ||\leq Ae^{-Ct} ||v|| \ \ \text{ for any } v\in E^{s}
,\end{cases}$$
where $||.||$ is some norm on $TM|_\Lambda$ and $A,C>0$ are positive constants.

Furthermore, we call $X$, or the flow generated by it $X^t$ {\em Anosov}, if $\Lambda=M$ is a closed manifold.
\end{definition}

In the context of this paper, uniform hyperbolicty appears mainly in the case of Anosov 3-flows, as well as on invariant 3-manifolds embedded in 4-manifolds (uniformly hyperbolic Liouville skeletons of codimention 1). 

The primary examples of Anosov flows in dimension 3 were the geodesic flows on the unit tangent space of hyperbolic closed surfaces, as well as, the suspension flows of Anosov linear transformations of $\mathbb{T}^2$. These examples are called {\em algebraic} thanks to their constructions and satisfy many interesting rigidity properties. In fact, these were the only known examples of Anosov 3-flows (up to {\em orbit equivalence}) for about 20 years, until new surgery techniques produced new examples in the early 80s and now we have wide classes of Anosov flows in dimension 3, even on hyperbolic manifolds.

\begin{remark}\label{stability}
A remarkable feature of Anosov flows is their {\em structural stability} in the following sense. For any Anosov flow $X^t$ on a manifold of arbitrary dimension $M$, any sufficiently $C^1$-close flow $Y^t$ is Anosov and {\em orbit equivalent} to $X^t$, i.e. there exists a homeomorphism of $M$ sending (oriented) orbits of $Y^t$ to (oriented) orbits of $X^t$.
As a result, in many contexts, especially ones motivated from topological or geometric viewpoint, this justifies restricting our attentions to $C^\infty$ flows for convenience, as any $C^1$ Anosov flow can be approximated, preserving the orbit structure, by a $C^\infty$ flow (this is more generally true for any {\em Axiom A flow}). Therefore, while most of the result in this paper can be refined in the terms of the regularity of the flow, we restrict our statements and proofs to $C^\infty$ flows for the sake of more straightforward staements and convenience.
\end{remark}

In the above definition, the sub-bundles $E^{s}$ and $E^{u}$ are called the {\em strong stable and unstable bundles}, respectively and it is not hard to see that the sub-bundles $E^{ws}:=\langle X \rangle \oplus E^{s}$ and $E^{wu}:=\langle X \rangle \oplus E^{u}$, which are called the {\em weak stable and unstable bundles}, respectively, are uniquely integrable and as a result, tangent to foliations $\mathcal{F}^{ws}$ and $\mathcal{F}^{wu}$, which are called the {\em weak stable and unstable foliations}, respectively.

\begin{remark}
It is important to notice that in the above definition, the splitting of the tangent space, and therefore all the mentioned invariant bundles and foliations, is a priori only Hölder continuous in general, regardless of the regularity of the flow generated by $X$. In fact, there is an interesting regularity theory to address such subtlety and we will discuss this more in Section~\ref{2.3}. In particular, in the case we are mostly interested in, i.e. the case of Anosov 3-flows, it can be (non-trivially) shown \cite{hps,reg} that the weak bundles $E^{ws}$ and $E^{wu}$ are at least $C^1$ (see Corollary~\ref{c1regweak}).
\end{remark}

The foliations $\mathcal{F}^{ws}$ and $\mathcal{F}^{wu}$ in some sense provide a {\em local picture} for Anosov 3-flows, given the following technical remark.

\begin{remark}\label{adapt}
We notice that the above definition implies the {\em eventual} expansion, which is not {\em immediate} when $0<A<1$. However, an averaging technique of Holmes \cite{holmes}, popularized by Hirsch-Pugh-Shub \cite{hps}, indicates that we can always choose a norm with respect to which, the expansion and contractions are immediate, i.e. we can assume $A=1$. Gourmelon \cite{gour} generalizes this idea to weaker notions of hyperbolicty, like {\em dominated splittings}, which we will encounter in the following. A differentiable refinement of such averaging technique is given in the Appendix of \cite{hoz5} as well.
\end{remark}

Therefore, assuming $A=1$, we achieve a somewhat {\em local} model of Anosov 3-flows  (see Figure~1). One should note that there are elements of this picture which are not truly local, since the invariant foliations $\mathcal{F}^{ws}$ and $\mathcal{F}^{wu}$ depend on the long term behavior of the flow and perturbing the flow in a neighborhood can deform these foliations on the entirety of the underlying manifold.

In order to quantify the notions of exponential expansion and contraction, we use the notion of {\em expansion rates}. These quantities are known to play a significant role in many aspects of Anosov dynamics, including its regularity theory (see Section~\ref{2.3}), as well as the contact and symplectic (or in general, differential) geometric theory of such flows \cite{hoz3,massoni}. Here, we overview their main properties and one should refer to Section~3 of \cite{hoz3} for more discussions on the basic properties of these quantities.

Choosing any norm which is differentiable along the flow, we define the expansion rates of the unstable and stable bundles (as functions $r_u,r_s:M\rightarrow \mathbb{R}$), respectively, as
$$
r_u:=\partial_t \cdot \ln{||X^t_*(u) ||} \ \ \ \text{and} \ \ \ 
r_s:=\partial_t \cdot \ln{||X^t_*(v) ||} 
$$
where $u\in E^{u}$ and $v\in E^{s}$ are arbitrary vectors. In fact, the existence of an adapted norm discussed in Remark~\ref{adapt}, implies that for any Anosov flow, such norm can be chosen such that $r_s<0<r_u$, i.e. chhosing an appropriate norm, we can assume that the expansion and contractions in the unstable and stable directions, respectively, start immediately, giving local picture for Anosov 3-flows in terms of their invariant foliations.

 \begin{figure}[h]\label{localanosov}

 \begin{subfigure}[b]{0.4\textwidth}

  \center \begin{overpic}[width=9cm]{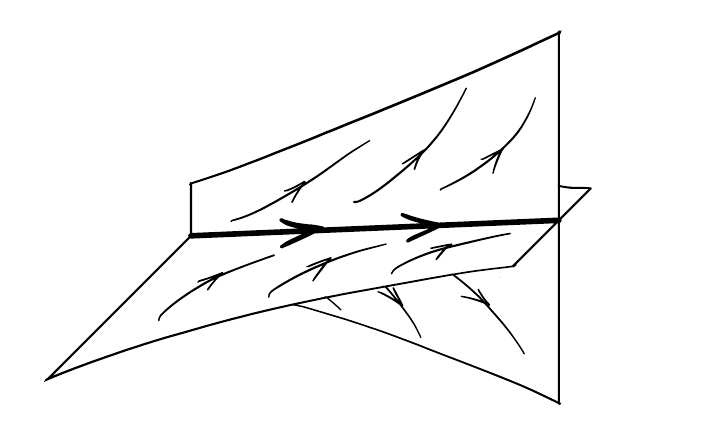}

         \put(45,25){$\mathcal{F}^{ws}$}
          \put(210,145){$\mathcal{F}^{wu}$}

       \end{overpic}
    \caption{Foliation picture of Anosov 3-flows}
    \label{fig:1}
  \end{subfigure}
  \hspace{2cm}
  \begin{subfigure}[b]{0.5\textwidth}
  \center \begin{overpic}[width=7cm]{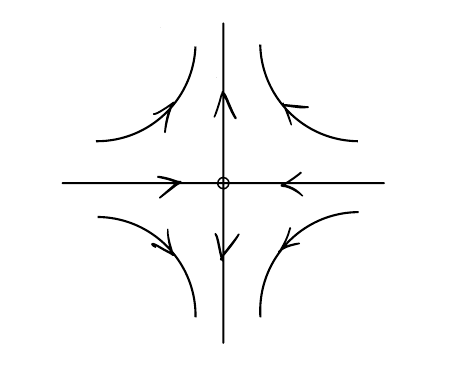}
  
      \put(170,75){$E^{s}$}
          \put(92,153){$E^u$}

  \end{overpic}
    \vspace{0.3cm}
    \caption{Action of Anosov flows on the tangent bundle}
    \label{fig:2}
  \end{subfigure}
  
 \caption{Local geometry of Anosov flows}
\end{figure}

Different characterizations of these quantities can be useful for various purposes and generalizations. For instance, if we take $e_u \in E^{u}$ and $e_s\in E^{s}$ to be  unit vectors with respect to such norm. It is easy to show
$$
\mathcal{L}_X e_u=-r_ue_u \ \ \ \text{and} \ \ \ 
\mathcal{L}_X e_s=-r_se_s .
$$

But more conveniently in the context of this paper, we would like to measure such quantities on the normal bundle of the flow $TM/ \langle X \rangle$ (for instance, this will provide a characterization needed in the generalizations of Anosov flows, e.g. {em projectively Anosov flows}, when strong invariant bundles do not exist and we only can work with weak invariant bundles corresponding to invariant sub-bundles of $TM/ \langle X \rangle$). First note that the norm in the above definition induces an area form on the weak invariant bundles and equivalently, we can describe such expansion rates as
$$
r_u=\partial_t \cdot \ln{det(|X^t_*|_{E^{wu}})}  \ \ \ \text{and} \ \ \ 
r_s=\partial_t \cdot \ln{det(X^t_*|_{E^{ws}})} 
.$$

This can be viewed as measuring the expansion rates by looking at the normal bundle $TM/ \langle X \rangle \simeq E^{ws} \oplus E^{wu}$ (note that weak invariant bundles are projected into invariant line bundles in $TM/ \langle X \rangle$, which abusing notations, we still denote by $E^{ws}$ and $E^{wu}$). To see this more clearly (see Section~3 of \cite{hoz3} for more details), notice that any choice of the plane field $\eta$ transverse to and differentiable along $X$, gives rise to an isomorphism $TM/\langle X\rangle \simeq \eta$ and hence, a correspondence between the norms on $\eta$ and $TM/ \langle X \rangle$. Therefore, any vector field $\hat{e}_u\in E^{wu} \subset TM/ \langle X \rangle$ satisfying $\mathcal{L}_X \hat{e}_u=-r_u \hat{e}_u$ can be lifted to a vector field $e_u\subset \eta \cap E^{wu} \subset TM$ satisfying $\mathcal{L}_X e_u=-r_u e_u+q_u X$, where $q_u$ is a function on $M$ and $q_u\equiv 0$ exactly when $E^{u}\subset\eta$. Finally, we notice that that the expansion rates can be computed in terms of 1-forms whose kernels include $X$, a viewpoint which is useful when we discuss contact geometric methods, i.e. it is easy to see that if we define $\alpha_u$ by $\alpha_u(E^{ws})=0$ and $|\alpha_u(e_u)|=1$ and $\alpha_s$ similarly, we have
$$
\mathcal{L}_X \alpha_u=r_u \alpha_u\ \ \ \text{and} \ \ \ 
\mathcal{L}_X \alpha_s=r_s \alpha_s.
$$

We summarize these facts in the following, where we denote the action of $X^t$ on $TM/ \langle X \rangle$ by $X^t_*$.

\begin{proposition}\label{expansion}
Given an Anosov 3-flow and using the above notation, the followings provide equivalent characterization of the expansion rates for the unstable bundle with respect to a norm. Similar fact is true for the stable bundles.

(0) $r_u$ is the expansion rate of the unstable bundle with respect to some norm $||.||$ on $E^{u}$;

(1) $r_u=\partial_t \cdot \ln{||X^t_*(e) ||}$ for some norm $||.||$ on $E^{uu}$ and any $e\in E^{u}$;

(2) $\mathcal{L}_X e_u =-r_u e_u$, where $e_u \in E^{u}$ is the unit vector with respect to some norm $||.||$ on $E^{u}$;

(3) $\mathcal{L}_X e_u =-r_u e_u+q_u X$, where $e_u \in E^{wu}$ is the unit vector with respect to some norm $||.||$ on $E^{wu}$ and $q_u$ is function on $M$;

(4) $r_u=\partial_t \cdot \ln{det(X^t_*|_{E^{wu}})}$ for some area form on $E^{wu}$;

(5) $r_u=\partial_t \cdot \ln{||X^t_*(\hat{e}_u) ||}$ for some norm $||.||$ on $E^{wu}\subset TM/ \langle X \rangle$ and any $\hat{e}_u\in E^{wu}$;

(6) $\mathcal{L}_X \hat{e}_u =-r_u \hat{e}_u$, where $\hat{e}_u \in E^{wu} \subset TM/ \langle X \rangle$ is the unit vector with respect to some norm $||.||$ on $E^{wu}$;

(7) $\mathcal{L}_X \alpha_u=r_u \alpha_u$, where $\alpha_u$ is a non-vanishing 1-form satisfying $\alpha_u(E^{ws})=0$.
\end{proposition}

The fact that such expansion rates can be computed in the normal bundle $TM/\langle X\rangle$, makes them well-behaved under the reparametrizations of the flow.  In particular, we have
 
 \begin{proposition}\label{exprep}
If $r_s$ and $r_u$ are the expansion rate for $X$, the corresponding expansion rates for the reparametrization of the flow generated by $fX$, where $f:M\rightarrow \mathbb{R}$ is a non-zero function, are $fr_u$ and $fr_s$.
\end{proposition}

The following lemma of Simić \cite{simic} shows that for an appropriate choice of norm, the regularity of the expansion rates can be assumed to the same as of the weak invariant bundles $E^{ws}$ and $E^{wu}$.

\begin{lemma}\label{simiclemma}
Assuming that the vector field $X$ generating an Anosov flow has at least one degree of regularity higher than the weak stable bundle $E^{ws}$, for an appropriate choice of norm on $E^u$, one has the corresponding expansion rate $r_u$ to be as regular as $E^{ws}$. In particular, if $X$ assumed to be a $C^{2+}$ Anosov vector field, such norm can be chosen such that $r_u$ is $C^{1+}$.
 The same holds for expansion rate of $E^s$.

\end{lemma}

\begin{proof}
Choose any $C^\infty$ vector field $\hat{e}_u$ which is transverse to $ E^{wu}$. 
Define the 1-form $\alpha_u$ such that $\ker{\alpha_s}=E^{ws}$ and $\alpha_u(\hat{e}_u)=1$ and note that such 1-form has the same regularity as $E^{ws}$. Note that this yields a norm on $E^u$ by letting $||e_u||=|\alpha_u(e_u)|$ for any $e_u\in E^u$. We have
$$r_u=r_u\alpha_u(\hat{e}_u)=(\mathcal{L}_X \alpha_u)(\hat{e}_u)=-\alpha_u([X,\hat{e}_u])$$
and therefore, this implies that for any $k\geq 1$, $r_u$ is $C^k$ as long as $X$ is $C^{k+1}$ and $E^{ws}$ is $C^k$. The claim about $r_u$ being $C^{1+}$ when $X$ is $C^{2+}$ follows from the fact for such $X$, the weak invariant bundles are known to be $C^{1+}$ \cite{reg}.
\end{proof}

We would like to note that the definition of the expansion rates can be naturally extended to an arbitrary vector bundle over an invariant set of the flow generated by a vector field $X$, which we denote by $E\rightarrow \Lambda$, as long as we have a fiberwise norm differentiable with respect to the action of the flow, which abusing notation, we still write as $X^t_*$. The following definition is consistent with the characterization (2) in Proposition~\ref{expansion}. In the case, $E$ is 1-dimensional, this becomes equivalent to the other formulations in Proposition~\ref{expansion}. We again denote the action of $X$ on $E$ by $X^t_*$.

\begin{definition}\label{generalexp}
Using the above notation and for a differentiable vector bundle $E\rightarrow \Lambda$, where $\Lambda$ is an invariant set for the flow generated by $X$, we define 
$$r=\partial_t \cdot \ln{det(X^t_*|_{E})}|_{t=0},$$
as the expansion rate of $E$ under the flow generated by $X$. 
\end{definition}

Finally, we would like to notice that the expansion information of an Anosov flow can be recorded in terms of a special reparametrization of the flow we call {\em synchronization}. These are reparametrizations with respect to which, the flow has constant unit expansion rate in the unstable bundle, i.e. $r_u=1$. We can naturally extend this definition to an arbitrary vector bundle $E\rightarrow \Lambda$.

\begin{definition}
In the above setting, we say that $X$ is E-{\em synchronized}, if with respect to some norm we have the expansion rate of $E$ is $r\equiv 1$. We call an Anosov 3-flow synchronized, if it is $E^{u}$-synchronized.
\end{definition}

Notice that a necessary condition for a reparametrization of a flow to be $E$-synchornized is to have an action on $E$ with eventual expansion (or eventual contraction if we allow reversing the direction of the flow) with respect to some norm, where after an averaging of the norm as in \cite{hps}, one can achieve immediate expansion. For basically all situations we are interested in, the flow can then be synchronized after rescaling the generating vector field (which corresponds to a $C^0$ reparametrization of the flow). In particular, any Anosov flow can be {\em }synchronized with respect to $E^{wu}$ and $E^{ws}$, by first using an adapted norm (with respect to which $r_s <0<r_u$) and then, multiplying the generating vector field $X$ by the functions $\frac{1}{r_u}$ or $\frac{1}{r_s}$, respectively, where for synchronizing with respect to $E^s$ such operation reverses the direction of the flow (for which the original $E^{ws}$ is the unstable bundle).

We gather the relevant facts in the following:

\begin{proposition}\label{syncuniq}
If $X$ is the generating vector field for an Anosov 3-flow with the weak invariant bundles $E^{wu}$ and $E^{ws}$ equipped with an adapted norm (i.e. expansion rates $r_u>0$ and $r_s<0$, respectively), then

(1) $\frac{1}{r_u} X$ and $\frac{1}{r_s} X$ are $E^{u}$ and $E^{s}$-synchronized, respectively, with respect to the same norm;

(2) $E^{u}$-synchronization is unique up to $C^1$-conjugacy along the flow. The same holds for $E^{s}$-synchronization.

(3) $X$ is volume preserving, if and only if, possibly after a reparametrization of $X$, $X$ and $-X$ are $E^{wu}$ and $E^{ws}$-synchronized, respectively.

\end{proposition}

\begin{proof}
(1) follows from Proposition~\ref{exprep}, noting that we need to reverse the orientation of the flow to have $E^s$-synchronization. But that is done thanks to the fact that $r_s<0$.

To see (2), assume $X$ is $E^u$-synchronized with respect to some norm, i.e. $\mathcal{L}_X \alpha_u=\alpha_u$, where $\alpha_u$ corresponds to such norm. Now if some reparametrization of $X$, like $Y=fX$ for a function $f>0$, is also $E^u$-synchronized with respect to some other norm corresponding to $\bar{\alpha}_u=g\alpha_u$ for the $C^1$ function $g>0$, we have
$$g\alpha_u=\bar{\alpha}_u=\mathcal{L}_Y \bar{\alpha}_u=\mathcal{L}_{(fX)}(g\alpha_u)=f[(X\cdot g)\alpha_u+g\mathcal{L}_X \alpha_u]=f[X\cdot \ln{g} +1]g\alpha_u,$$
implying that $f=\frac{1}{1+X\cdot \ln{g}}$, which means that $X$ is $C^1$-conjugate to $Y$ via the time change $t\mapsto t+\ln{g}$.

For (3), notice that $X$ being volume preserving is equivalent to $r_s=-r_u$ (see \cite{hoz4}). The conclusion follows from (1).
\end{proof}

\begin{remark}
Here, we remark that the rigidity of algebraic Anosov flows can be described in terms of synchronizations. In particular, an Anosov vector field $X$ is algebraic, if and only if, $X$ and $-X$ are synchronized and preserve a $C^1$ transverse plane field. However, an {\em asymptotic} synchronization is possible for any Anosov flow as shown in \cite{hoz5}.
\end{remark}


\subsection{Related notions of hyperbolicity}\label{2.2}

The goal of this section is to introduce other notions of hyperbolicity, beside Anosov flows, which will be used in this paper. More specifically, we are interested in (1) {\em projectively Anosov flows}, as important generalizations of Anosov 3-flows which form the cornerstone of the contact/symplectic geometric theory of Anosov flows; (2) (non-singular ){\em partially hyperbolic flows}, as an important middle ground between Anosov 3-flows and projectively Anosov flows for which a Liouville geometric description is still possible; and (3) {\em normally hyperbolic flows} which show up as the Liouville dynamics near a codimension 1 skeleton under certain conditions.

\subsubsection{Dominated splittings and projectively Anosov flows}\label{2.2.1}

A weak notion of hyperbolicity for flows is provided by the notion of {\em dominated splittings} (vs. the stronger notion of {\em hyperbolic} splittings), where we have invariant weak bundles, as well as {\em relative} hyperbolicity (vs. uniform hyperbolicity).

\begin{definition}\label{padef}
Let $X^t$ be any non-singular flow on a manifold $W$ (of arbitrary dimension) and $\Lambda$ be an invariant set for $X^t$ such that $TM/\langle X \rangle \big|_\Lambda \simeq  E \oplus F$ is an invariant continuous splitting of $TW$ over $\Lambda$ and $E$ and $F$ satisfy
$$|| X^{t}_* (u)||/ ||X^t_* (v)|| \geq Ae^{Ct} ||u||/||v||,$$
for any $u\in F$ and $V\in E$ and some positive constants $A$ and $C$. We call such splitting over $\Lambda$ a {\em dominated splitting}, or say that $F$ {\em dominates} $E$.
When $M$ is a closed 3-manifold and $\Lambda=M$, we call $X^t$ a projectively Anosov flow.
\end{definition}

In other words, a flow is projectively Anosov, if its action on the projectified normal bundle of the flow direction is similar to an Anosov flow. Here, we remark that projectively Anosov flows have been studied in various contexts and sometimes with different names. While the term projective Anosov has been frequently used in the geometry literature \cite{noda,noda2,asaoka}, these flows are usually referred to as {\em flows with dominated splittings} in the dynamics literature \cite{three} and {\em conformally Anosov flows} in the contact and foliation theoretic literature \cite{et}.

The invariant bundles $E$ and $F$ in the above definition can be lifted to invariant plane fields on $TM|_\Lambda$, which abusing notation we still call $E$ and $F$. These plane fields are (possibly non-uniquely) integrable and hence, tangent to {\em branched foliations}. It is noteworthy that projectively Anosov flows are known to be much more abundant than Anosov 3-flows, with examples existing on any 3-manifold, including examples on $S^3$, $T^3$ or Nils manifolds, which don't admit any Anosov flows. On the other hand, rigidity and classification results are provided assuming high regularity of the weak invariant bundles \cite{asaoka,noda}, implying that for most projectively Anosov flows such bundles have low regularity. It is also well-understood that the splitting of the normal bundle $TM/\langle X \rangle$ cannot necessarily be lifted to a an invariant splitting of $TM$ as in the definition of Anosov flows \cite{noda2}. The significance of projectively Anosov flows for us is mainly thanks to the essential role they play in the contact geometric theory of Anosov 3-flows \cite{mitsumatsu,hoz3}, as we will discuss further in Section~\ref{3}.

\begin{remark}
Given a non-singular 3-flow $X^t$ which preserves a continuous splitting $TM/\langle X\rangle \simeq E\oplus F$, the essence of the domination relation in Definition~\ref{padef} is the {\em rotation} of intermediate plane fields in the following sense. We can choose any Riemannian metric with respect to which $E$ and $F$ are orthogonal. Assuming orientability of $E$ and $F$ for convenience, any (oriented) plane field $\xi$ containing the direction of $X$ can be described in terms of an angle function $\theta_\xi:M\rightarrow \mathbb{S}^1$, where the equations $\sin{\theta}=0$ and $\cos{\theta}=0$ define $E$ and $F$, respectively. Now, if $\xi$ is defined by an angle function with $\tan{\theta_\xi}>0$ (i.e. $\xi$ being a plane field in the first or third quadrant of Figure~2) and we define $\theta_\xi^t$ to be the angle $X^t_*(\xi)$ makes with $E$. Then, it is easy to show
$$\partial_t\cdot \theta_\xi^t >0 \Longleftrightarrow \partial_t\cdot \tan{\theta_\xi^t}>0 \Longleftrightarrow \text{ $F$ dominates $E$ in the sense of Definition~\ref{padef}}.$$
In other words, the action of the flow pushes such $\xi$ towards $F$ and away from $E$, if and only if, $F$ dominates $E$. The same is true if $\tan{\theta_\xi}<0$.
\end{remark}

   \begin{figure}[h]
\centering
\begin{overpic}[width=0.41\textwidth]{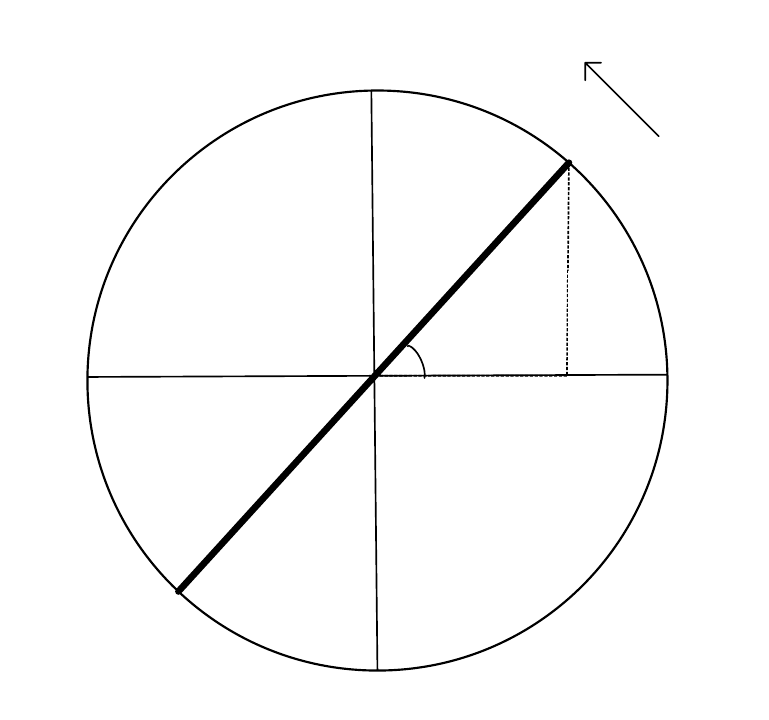}
    \put(90,162){$F$}
    \put(145,145){$\xi$}
        \put(155,160){$X_*^t$}
        \put(170,85){$E$}
         \put(110,93){$\theta_\xi^t$}
          \put(145,120){$F$-component}
           \put(100,75){$E$-component}
  \end{overpic}

\caption{Domination of $E$ by $F$ as rotation of plane fields}
\end{figure}
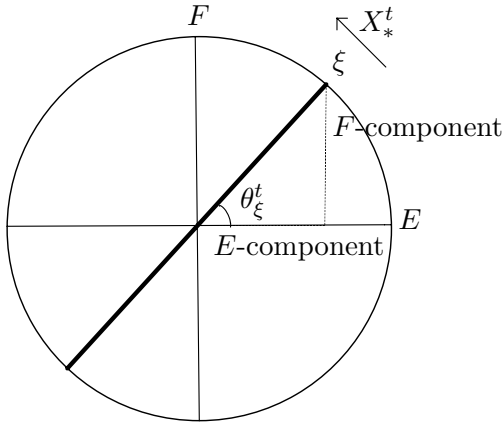

Similar to the case of Anosov flows, it is known that in the above definition, the constant $A$ can be taken to equal 1 by an appropriate choice of norm, which we call {\em adapted} again \cite{gour}. In other words, the domination starts immediately with respect to an adapted norm. Moreover, we can still define, as in Definition~\ref{generalexp} , the expansion rates of $E$ and $F$, which we still denote by $r_s$ and $r_u$, respectively. It is easy to see that for an adapted norm, we have $r_u>r_s$ \cite{hoz3}.


\subsubsection{Partially hyperbolic 3-flows}\label{2.2.2}

An intermediate notion between Anosov and projectively Anosov is partially hyperbolicty. Partial hyperbolicity appears in the literature in various contexts with different definitions. For flows, it usually refers to when the action of a flow provides a domination relation and for non-singular flows on closed 3-manifolds, which is the main focus of this paper, this essentially boils down to the existence of a dominated splitting as in the definition of projectively Anosov flows, while having a strong unstable (or stable depending on the definition) bundle. More precisely, the domination relation of the form $TM \simeq E \oplus F$ is required in the definition of partial hyperbolicity, where the non-singularity of the flow implies that the line bundle $\langle X \rangle$ in entirely included in one of the two bundles, which will then need to be 2 dimensional, if the ambient manifold is 3 dimensional. By convention, we take $E$ to be the 2-dimensional bundle including the flow direction and since the flow action does not change the norm on $X$, the domination of $E$ by $F$ implies absolute expansion in the $F$ direction, i.e. $F$ being in fact an invariant strong unstable bundle for the flow. 

We note that in most of the literature on partially hyperbolic flows \cite{pa,pa2}, the convention is to consider the flow direction being included in $F$ and hence, the existence of a strong stable bundle, while considering the additional condition of sectional hyperbolicity to deal with singularities of the flow. For the purposes of this paper, it is more natural to consider the existence of the strong unstable bundle with no extra assumption as we are only dealing with non-singular flows. Hence, we have the following definition which suits our goals.

\begin{definition}
Let $X^t$ be any projectively Anosov flow on a 3-manifold $M$ with the dominated splitting $TM/\langle X \rangle \simeq  E \oplus F$. We call $X^t$ partially hyperbolic, if for any $u\in F$, we have
$$|| X^{t}_* (u)|| \geq Ae^{Ct} ||u||,$$
with respect to some norm $||.||$ and positive constants $A$ and $C$. In this situation, we also denote $F$ by $E^{wu}$ and call $TM/\langle X \rangle \simeq  E \oplus E^{wu}$ a partially hyperbolic splitting.
\end{definition}

In other words, non-singular partially hyperbolic flows are projectively Anosov flows where we have absolute expansion in the dominating direction. As mentioned above, and will be reiterated in Lemma~\ref{doering}, the splitting in the above definition can always be lifter to a splitting of the form $TM\simeq E\oplus E^u$, where $E^u$ is a string unstable bundle.

The existence of non-Anosov examples of such flows is implied by a construction of Elaishberg-Thurston \cite{et} (also see \cite{bbp}) using a DA (Derived from Anosov) deformation near a periodic orbit of an Anosov flow (a deformation introduced by Frank-Williams \cite{da}), which gives, as far as we know, the only examples of non-singular partially hyperbolic 3-flows in the literature. We will revisit this construction in Section~\ref{3.5} and show explicitly that such examples can be constructed by a deformation through projectively Anosov flows.

The conditions of projective Anosovity and partial hyperbolicity can be naturally written in terms of the expansion rates of the invariant bundles. We record these facts in the following:

\begin{proposition}
Let $X^t$ be a projectively Anosov flow with the dominated splitting $TM/\langle X \rangle \simeq  E \oplus F$ with expansion rates $r_s$ and $r_u$ for $E$ and $F$, respectively. Then,

(a) there exists some norm on $TM/\langle X \rangle$ with respect to which
$$r_u>r_s.$$

(b) the flow $X^t$ is partially hyperbolic, if and only if, there exists some norm on $TM/\langle X \rangle$ with respect to which
$$r_u>r_s \ \ \ \text{and}\ \ \ r_u>0$$

(c) the flow $X^t$ is Anosov, if and only if, there exists some norm on $TM/\langle X \rangle$ with respect to which
$$r_u>0>r_s.$$
\end{proposition}

\begin{remark}\label{syncuniqph}
It is easy to observe that the argument of Simić on the regularity of the expansion rates, i.e. Lemma~\ref{simiclemma}, still holds in such general context. Moreover, in the category of non-singular partially hyperbolic flows, we can define synchronization with respect to $E^u$ and the same proof as Proposition~\ref{syncuniq} shows that such synchronization is unique up to conjugacy.
\end{remark}


\subsubsection{Normally hyperbolic invariant manifolds}\label{2.2.3}

We finally would like to introduce normal hyperbolicity which is shown to be essential in the persistence theory of dynamical systems, in particular, using the work of Hirsch-Pugh-Shub \cite{hps} and Mañe \cite{mane}, and in our context, appears as the natural skeleton dynamics in the Liouville geometric objects we construct from Anosov flows.

In the following, for the bundle map $A:E\rightarrow E$, where $E$ is equipped with a norm $||.||$, the (maximum) norm is defined as $$||A||:=\sup\{||Ax||:||x||=1\}$$ and the minimum norm is defined as
$$m(A)=\inf\{||Ax||:||x||=1\}.$$

\begin{definition}
Let $X^t$ be any non-singular flow on a manifold $W$ and $\Lambda$ a compact $C^1$ invariant set for $X^t$ such that $TW|_\Lambda \simeq  E^{s} \oplus T\Lambda \oplus E^{u}$ is an invariant continuous splitting of $TW$ over $\Lambda$.
We call $\Lambda$ a {\em normally hyperbolic invariant manifold} for $X^t$, or say that $X^t$ is normally hyperbolic at $\Lambda$, when
$$||X^1_*|_{E^s}||<m(X^1_* |_{T\Lambda}) \leq ||X^1_* |_{T\Lambda} ||<m(X^1_*|_{E^u} ),$$ with respect to some norm $||.||$ on $TW|_{\Lambda}$. When $E^s=\emptyset$, we call such invariant manifolds {\em normally repelling}.
\end{definition}

A very important feature of normal hyperbolicity is that it characterizes $C^1$-persistent invariant manifolds, thanks to important development in the study of hyperbolic dynamical systems \cite{hps,mane}.

\begin{definition}
Let $\Lambda$ be a compact $C^1$ invariant set for the flow $X^t$. We say that $\Lambda$ is $C^1$-persistent, if

(1) $\Lambda$ has an open neighborhood $U$ such that $\Lambda=\cap_{t\in\mathbb{R}} X^t (U)$;

(2) for any other flow $\hat{X}^t$ which is $C^1$-close to $X^t$, the set $\hat{\Lambda}=\cap_{t\in\mathbb{R}} \hat{X}^t (U)$ is a $C^1$ invariant set for $\hat{X}^t$ which is $C^1$-close to $\Lambda$.
\end{definition}

Hirsch-Pugh-Shub studied normally hyperbolic flows extensively in their seminal book \cite{hps} and among other things, prove the $C^1$-persistence of such invariant sets.

The flow version of the {\em fundamental theorem of normal hyperbolicity} reads as

\begin{theorem}(Hirsch-Pugh-Shub 1970 \cite{hps})\label{normalh}
Let $X$ be a $C^1$ vector field, defined on the manifold $W$, leaving a compact $C^1$ submanifold $\Lambda$ invariant and is normally hyperbolic at $\Lambda$, respecting $T_\Lambda W=E^{s}\oplus T\Lambda \oplus E^{u}$. Then,

(a) There are $C^1$ invariant foliations (locally near $\Lambda$) $\mathcal{W}^{ws}$ and $\mathcal{W}^{wu}$ tangent to $E^{u}\oplus T\Lambda$ and $E^{s}\oplus T\Lambda$ (existence of weak invariant bundles);

(b) Each leaf of $\mathcal{W}^{ws}$ or $\mathcal{W}^{wu}$ is fibered by $C^1$ submanifolds  $\mathcal{W}^{s}$ or $\mathcal{W}^{u}$ (existence of strong invariant bundles);

(c) If $\hat{X}^t$ is a flow $C^1$ close to $X$, then $\hat{X}^t$ is normally hyperbolic at an invariant submanifold $\hat{\Lambda}$ which is $C^1$-close to $\Lambda$ ($C^1$-persistence);

(d) Near $\Lambda$, the flow $X^t$ is topologically conjugate to its linearization (linearization).
\end{theorem}

The celebrated result of Mañe \cite{mane} then proves the inverse of the $C^1$-persistence statement, giving a complete characterization of $C^1$-persistence in terms of normal hyperbolicity.

\begin{theorem}(Mañe 1976 \cite{mane})
The invariant set $\Lambda$ is $C^1$-persistent for $X^t$, if and only if, $X^t$ is normally hyperbolic at $\Lambda$.
\end{theorem}

It is noteworthy that Mañe characterizes when the invariant set persists as a differentiable manifold, while it is well-known that in the absence of the rate condition in normal hyperbolicity, persistence of the invariant set as a set, i.e. $C^0$-persistence, is known to be much more delicate~\cite{floer2,floer1,c0}.


\subsection{Invariant bundles and regularity theory}\label{2.3}

In this section, we want to discuss the tools we need in order to find invariant bundles and determine their regularities. Our main references on this topic are \cite{hps} and \cite{reg,reg2}. More specifically, our main tools to find invariant manifolds come from the theory of Hirsch-Push-Shub \cite{hps} on normal hyperbolicity, as an important take on the previous work of Hadamard and Perrone, while our regularity arguments for the weak invariant bundles rely on the refinements introduced by Hasselblatt~\cite{reg}.

The main idea of the invariant manifold theory, which goes back to Hadamard-Perrone, is called {\em graph transformation}. This is manifested for instance in the $C^r$-section theorem \cite{hps}, which states that a fiberwise contraction of a fiber bundle (with respect to some Finsler norm) has a unique continuous invariant set and furthermore, a lower bound for the regularity of such invariant set is provided by the rate of fiberwise contraction (we will revisit this theorem in Section~\ref{8}).

An important ingredient in Hirsch-Pugh-Shub's revisiting and refinement of the $C^r$-section theorem is the following lemma, which we will be one of our main tools in finding invariant sub-bundles of vector bundles over invariant sets.

\begin{lemma}[Hirsch-Pugh-Shub 1970 \cite{hps} Lemma~2.18]\label{strong}
Let
\[\begin{tikzcd}
	0 & {E_1} & {E_2} & {E_3} & 0 \\
	0 & {E_1} & {E_2} & {E_3} & 0 \\
	& \Lambda & \Lambda & \Lambda
	\arrow["{T_1}", from=1-2, to=2-2]
	\arrow["{T_2}", from=1-3, to=2-3]
	\arrow["{T_3}", from=1-4, to=2-4]
	\arrow[from=1-1, to=1-2]
	\arrow["i", from=1-2, to=1-3]
	\arrow["j", from=1-3, to=1-4]
	\arrow[from=1-4, to=1-5]
	\arrow[from=2-2, to=3-2]
	\arrow[from=2-3, to=3-3]
	\arrow[from=2-4, to=3-4]
	\arrow["j", from=2-3, to=2-4]
	\arrow["i", from=2-2, to=2-3]
	\arrow["{=}"{description}, draw=none, from=3-2, to=3-3]
	\arrow["{=}"{description}, draw=none, from=3-3, to=3-4]
	\arrow[from=2-1, to=2-2]
	\arrow[from=2-4, to=2-5]
\end{tikzcd}\]
be a commutative ladder of short exact sequences of Finslered vector bundles, all over the same compact base $\Lambda$, where $T_k$ is a bundle map over the base homeomorphism, $f:\Lambda\rightarrow \Lambda$, $k=1,2,3$. If $T_3$ is invertible and $$m(T_3|E_{3x})>||T_1|E_{1x} || \ \ \ \ for\ \ \ \ x\in\Lambda,$$
then $iE_1$ has a unique $T_2$-invariant complement in $E_2$.
\end{lemma}

We will mainly use this lemma in order to find invariant sub-bundles given an appropriate domination relation.

A famous application of the above lemma is the Doering lemma \cite{doering}, which states that a non-vanishing vector field $X$ induces a hyperbolic splitting on the normal bundle of the flow $TM / \langle X \rangle \simeq E^{ws} \oplus E^{wu}$, if and only if, $X$ induces the Anosov splitting on the tangent space $TM \simeq \langle X\rangle \oplus E^{s}\oplus E^{u}$. 

Note that the proof of reverse is trivial. But for the forward implication, one needs to prove the existence of an invariant line bundle (the strong stable bundle) $E^{s}$ inside $E^{ws}$, as well the same for $E^{wu}$, something which is not necessarily possible for a general dominated splitting \cite{noda}. To see how the Doering's argument works, it is sufficient to let $E_1=E^{ws}$, $E_2=TM$ and $E_3=TM / iE^{ws}$, where $i$ and $j$ are the natural inclusion and projection, respectively, and $T_k$, $k=1,2,3$ are the natural actions of the time-1 map of the flow generated by $X$ on each of those subspaces over $\Lambda=M$. Now, we are in the situation of the above lemma, since at any point we have $$e^{\int_0^1 r_u \circ X^\tau d\tau}=m(T_3|E_{3})>||T_1|E_{1} ||=\max{\{1,e^{\int_0^1 r_s \circ X^\tau d\tau}\}}=1$$ and therefore, $iE^{ws}$ has an invariant complement in $TM$, implying the existence of the strong unstable bundle $E^{u}$ (an alternative proof can also be given letting $E_1=\langle X \rangle$, $E_2=E^{wu}$ and $E_3=E^{wu}/\langle X \rangle$). Similar argument yields the existence of the strong stable bundle $E^{s}$.

However, we note that the same argument in fact works for a general non-singular partially hyperbolic 3-flows, i.e. when $r_u>r_s$ and $r_u>0$, as they would still satisfy $m(T_3|E_{3})>||T_1|E_{1} ||$ in the above formulation, i.e. when we replace $E^s$ with $E$. We record this as follow.

\begin{lemma}[Doering 1987 \cite{doering}]\label{doering}
Let $X^t$ be a non-singular partially hyperbolic flow with the weak splitting $TM/\langle X \rangle\simeq E\oplus E^{wu}$. Then, the plane field $E^{wu}$ contains a strong unstable bundle $E^{u}\subset E^{wu}$, which is invariant under the flow. In other words, there exists a continuous invariant splitting $TM\simeq E \oplus E^u$, where $X\subset E$.
\end{lemma}

The above lemma in fact holds in any dimension. However, in dimension 3, as we will later see, for non-singular partially hyperbolic flows, the flow direction $X$ admits an invariant complement inside $E$ as well.

The machinery developed by Hirsch-Pugh-Shub \cite{hps} is aimed at studying the regularity of the weak stable and unstable invariant sub-manifolds (or foliations) of an Anosov flow. However, the {\em bunching} technology of Hasselblatt \cite{reg,reg2} provides the refinement needed to control the regularity of the weak invariant bundles (i.e. regularity at the level of the tangent bundle).

To describe Hasselblatt's result, we need the following description. Let $X$ be a non-vanishing vector field on $M$  (of arbitrary dimension) with a hyperbolic invariant set $\Lambda$ and the splitting $TM|_{\Lambda}\simeq E^{s}\oplus \langle X \rangle \oplus E^{u}$ and there are $C,\epsilon >0$ such that for all $p\in M$, there exists
$\mu_f<\mu_s <1-\epsilon<1+\epsilon <\nu_s< \nu_f$ (here, the subscripts $s$ and $f$ refer to slow and fast, respectively) so that for $v\in E^{s}|_p$ and $u\in E^{u}|_p$ and $t>0$, we have 

$$\frac{1}{C} \mu_f^t ||v|| \leq || X_*^t (v)|| \leq C \mu_s^t ||v||$$
and
$$\frac{1}{C} \nu_f^{-t} ||u|| \leq || X_*^{-t} (u)|| \leq C \nu_s^{-t} ||u||.$$

\begin{definition}\label{bunchdef}
The unstable and stable bunching constants are defined as
$$B^u(X):=\inf_{p\in M}[(\ln{\mu_s}-\ln{\nu_s})/\ln{\mu_f}]$$
and
$$B^s(X):=\inf_{p\in M}[(\ln{\nu_s}-\ln{\mu_s})/\ln{\nu_f}],$$
respectively.
\end{definition}

The main result of Hasselblatt in \cite{reg,reg2} states that the above constants provide a guaranteed level of regularity for the weak invariant bundles, where in the case of {\em symplectic} Anosov flows, i.e. when an Anosov flow preserves a transverse symplectic form, such regularity is in fact optimal (note that for Anosov 3-flows, being symplectic is equivalent to being volume preserving).

\begin{theorem}(Hasselblatt 1994 \cite{reg,reg2})\label{hassbunch}
If $B^u(X)\notin \mathbb{N}$, then $E^{wu}$ is $C^{B^u(X)}$. If $B^u(X)\in \mathbb{N}$, then $E^{wu}$ is $C^{B^u(X)-\epsilon}$ for any $\epsilon>0$. The same holds for $E^{ws}$ and $B^s(X)$.
\end{theorem}

It is not hard to see in the case that $E^{wu}$ is codimension 1, we have $B^u(X)>1$ and therefore, $E^{wu}$ is $C^{1+}$. More precisely, in this case $E^{s}$ is 1-dimensional and after a reparametrization of the flow we can assume $r_s\equiv -1$ and at each $p\in M$, we have
$$\ln{\mu_s}=\ln{\mu_f}=-1,$$
which yields $$B^u(X)=\inf_{p\in M} [1+\ln{\nu_s}]>1,$$
boiling down to
$$B^u(X)=\inf_{p\in M} [1+\inf_{t>0}\frac{1}{t}\int_{-t}^0 r_u\circ X^\tau (p)d\tau]>1,$$
when $E^u$ is also 1 dimensional, i.e. the case of an Anosov 3-flow. When $X^t$ is a volume preserving, the norm can be chosen such that we also have
$$\ln{\nu_s}=\ln{\nu_f}=1$$
and therefore, $$B^u(X)=B^s(X)=2$$
and the weak invariant bundles have regularity arbitrary close to $C^2$. In fact Katok-Hurder \cite{katok} show such Anosov flows satisfy a regularity condition stronger than being $C^{2-\epsilon}$ for any $\epsilon >0$, which is called {\em Zygmund regularity} and a well known result of Ghys \cite{ghys} shows that the $C^2$-regularity is only achieved for both $E^{wu}$ and $E^{ws}$ simultaneously, only in the case of algebraic Anosov flows. It is noteworthy that Paternain \cite{patreg} has shown that there exists $C^\infty$ non-algebraic Anosov flows for which one of the weak invariant foliation is $C^\infty$ and the other one is only $C^{1+k}$ for some $k<1$. These examples are constructed as {\em magnetic} Anosov flows and in particular, are orbit equivalent to algebraic Anosov flows (geodesic flows more specefically).

We record these in the following.

\begin{corollary}(Ghys 87 \cite{ghys}, Hurder-Katok 90 \cite{katok}, Hasselblatt 94 \cite{reg,reg2})\label{c1regweak}
Let $X$ be a $C^k$ vector field, $k\geq 2$, generating an Anosov 3-flow. Then,

(1) $E^{ws}$ and $E^{wu}$ are $C^{1+}$;

(2) If $X^t$ is volume preserving. Then, $E^{ws}$ and $E^{wu}$ are $C^{2-\epsilon}$ for any $\epsilon>0$.

(3) $X^t$ is an algebraic Anosov 3-flow, if and only if, $E^{ws}$ and $E^{wu}$ are $C^{2}$, if and only if, $E^{ws}$ and $E^{wu}$ are $C^{k}$.
\end{corollary}

\begin{remark}\label{bunchex}
Hasselblatt \cite{reg,reg2} defines the bunching constants of Definition~\ref{bunchdef} in the context of hyperbolic invariant sets and use them to give (sometimes optimal) lower bounds on the regularity of the weak invariant bundles. We notice here that the definition has natural generalizations which we will use later in. Let $X$ be a non-singular flow with an invariant set $\Lambda$ with dominated splitting $TM/\langle X \rangle |_\Lambda \simeq E\oplus E^{u}$ with absolute expansion of $E^u$ (in dimension 3, these are non-singular partially hyperbolic flows, but we can consider arbitrary dimensions). We again define
$$B^s(X):=\inf_{p\in M}[(\ln{\nu_s}-\ln{\mu_s})/\ln{\nu_f}],$$
where $\nu_s,\mu_s,\nu_f$ are defined as in Definition~\ref{bunchdef}, except we drop the condition $\mu_s<1$, and instead we assume $\nu_f\geq\nu_s>1$ (absolute expansion of $E^u$) and $\nu_s>\mu_s$ ($E^u$ dominating $E$). Assuming $X$ is synchronized (i.e. parametrized such that $r_u\equiv 1$), we have $\ln{\nu_s}=\ln{\nu_f}=1$ above and consequently,
$$B^s(X)=\inf_{p\in M}[1-\ln{\mu_s}]=\inf_{p\in M}[1-\sup_{t>0}\frac{1}{t}\int_0^t r_s\circ X^\tau (p)d\tau].$$
This well-defined constant positive number $B^s(X)$ helps us have estimates on what to expect about the regularity of the weak dominated bundle, where Theorem~\ref{hassbunch} states that such degree of regularity in fact holds in the case of hyperbolic invariant sets. Note that $B^s(X)>1$ corresponds to Anosovity of $X$. We will use this generalization to show that the result of Hasselblatt's on the regularity of the weak invariant bundle of Anosov flows can be extended to non-singular partially hyperbolic 3-flows (see Theorem~\ref{regex}).
\end{remark}

\section{Liouville geometry of non-singular partially hyperbolic 3-flows}\label{3}

The main purpose of this section is to summarize, generalize and contextualize the previous results in the interactions between Anosov flows in dimension 3 and contact and symplectic geometries, in order to have a unified theory of the subject which goes beyond conventions.

The study of such interaction was initiated in the mid 90s, mainly in Mitsumatsu's study of {\em non-Weinstein Liouville domains} \cite{mitsumatsu}, and independently, in Eliashberg-Thurston's study of {\em confoliations} \cite{et}. Their observation basically gives a contact geometric characterization of projectively Anosov flows, which can be seen as the corner stone of a high regularity and stable geometric models for Anosov 3-flows (compared to low regularity and unstable models based on foliations). Later on, this observation was promoted to a Liouville geometric characterization of Anosov 3-flows by the author \cite{hoz3}, exploiting approximation techniques which facilitates going from a low regularity description of foliations to a more flexible high regularity picture in terms of contact structures.

We start with reviewing some elementary notions in contact and Liouville geometry and then, describe the observations of Mitsumatsu \cite{mitsumatsu} and Eliashberg-Thurston \cite{et} as the bridge between hyperbolic dynamics in dimension 3 and contact geometry, use the construction of Mitsumatsu \cite{mitsumatsu} to provide a Liouville geometric characterization of Anosov 3-flows. We then introduce the general framework of {\em Liouville interpolation systems} with the goal of generalizing such characterization to the broader class of non-singular partially hyperbolic 3-flows and an arbitrary interpolation regime.

\subsection{Notions from Liouville and contact geometries}\label{3.1}

Here, we review elementary notions from Liouville and contact geometries in dimension 4 and 3, respectively. For a comprehensive introduction to Liouville (symplectic) geometry one can consult \cite{symptop,ce}, and \cite{contop} provides the necessary background in contact topology.

\subsubsection{Liouville domains in dimension 4}\label{3.1.1}

Liouville geometry is the study of {\em exact symplectic manifolds}. On a compact oriented 4-dimensional manifold $W$, this means the study of 1-forms $\alpha$ such that $d\alpha \wedge d\alpha>0$ is a volume form. We call such 1-form a {\em Liouville form}.

The Stoke's theorem implies that the boundary of such compact 4-manifold is necessarily non-empty, i.e. $\partial W\neq \emptyset$. Naturally, one needs a suitable condition on the boundary of $W$ in order to have a nice theory of such objects. Such condition is naturally defined from a dynamical point of view. More precisely, the non-degeneracy of the symplectic form $d\alpha$ provides a 1-to-1 correspondence between 1-forms and vector fields on $W$. The vector field $Y$ which is dual to $\alpha$ under such correspondence, i.e. the vector field satisfying $\iota_Y d\alpha=\alpha$, is called the {\em Liouville vector field} and using Cartan's formula, it is easy to observe $\mathcal{L}_Y \alpha=\alpha$ and therefore,
$$\mathcal{L}_Y (d\alpha \wedge d\alpha)=2d\alpha \wedge d\alpha,$$
which means that the flow generated by $Y$ expands the volume form induced on $W$ from the symplectic form, another indication that $\partial W\neq \emptyset$. Therefore, a natural condition on the boundary is to assume $Y$ is transversely pointing outward on $\partial W$.

\begin{definition}\label{lioudef}
The pair $(W,\alpha)$ is called a Liouville domain, if $\alpha$ is a Liouville 1-form on the compact 4-manifold $W$, i.e.
$d\alpha \wedge d\alpha >0$,
and the vector field $Y$ defined by $\iota_Y d\alpha=\alpha$, named the Liouville vector field, is positively transverse to $\partial W$.
\end{definition}

A basic example of Liouville domains in dimension 4 is to consider the unit disk in $\mathbb{R}^4$, defined by $\mathbb{D}=\{ (x_1,y_1,x_2,y_2)\in \mathbb{R}^4:x_1^2+y_1^2+x_2^2+y_2^2=1\}$, equipped with the 1-form $\alpha_{std}=\frac{1}{2}[x_1dy_1-y_1dx_1+x_2dy_2-y_2dx_2]$, since it is easy to check $d\alpha \wedge d\alpha$ is the standard volume form on $\mathbb{R}^4$ and the corresponding Liouville vector field can be computed as $Y=\frac{1}{2}[x_1\partial_{x_1}+y_1\partial_{y_1}+x_2\partial_{x_2}+y_2\partial_{y_2}],$ which points outward on the boundary $\partial \mathbb{D}$. Another example is given as the tautological 1-form $\alpha_{taut}$ on the unit cotangent bundle of a surface $\Sigma$. That is $(T^*_1 (\Sigma),\alpha_{taut})$ with the property that $\beta^*\alpha_{taut}=\beta$ for any 1-form $\beta$ on $\Sigma$ considered as $\beta:\Sigma\rightarrow T^*\Sigma$.

As in any other geometric theory, we need a reasonable deformation theory of Liouville structures. This is mainly provided by the so-called Moser technique. Moser technique has various applications in geometry and we will also use it in the context of Liouville interpolating systems to show the geometric rigidity of such objects. Therefore, we bring the main idea of this omportant technique for Liouville domains to indicate how it is applied later on.

 Suppose we have a 1-parameter family of Liouville domains $(W,\alpha_t)$ for $t\in [0,1]$ such that $\alpha_t=\alpha_0|_{\partial W}$ for all $t$ s, i.e. we assume the Liouville structure not changing on the boundary of the ambient manifold. The goal of the Moser technique is to show that there exists an isotopy $\phi^t$ of $W$ such that we have $\phi^t|_{\partial W}=Id$ and
$$\phi^{t*}(\alpha_t)=\alpha_0+dh_t,$$
where $h_t:W\rightarrow \mathbb{R}$ is a family of real functions with $h_t\equiv 0 |_{\partial W}$ for $t\in [0,1]$. To show this, we first assume such isotopy exists and is generated by the vector field $v_t$ and then compute such $v_t$ satisfying the necessary and sufficient conditions. Differentiating both sides of the above equation yields
$$\phi^{t*}[\partial_t \cdot \alpha_t +\mathcal{L}_{v_t} \alpha_t]=d(\partial_t \cdot h_t)$$
$$\Rightarrow \phi^{t*}[\partial_t \cdot \alpha_t +\iota_{v_t} d\alpha_t+d(\alpha_t(v_t))]=d(\partial_t \cdot h_t).$$

Thanks to the fact that $d\alpha_t$ is non-degenerate, there exists a unique vector field $v_t$ satisfying $\iota_{v_t} d\alpha_t=-\partial_t \cdot \alpha_t$, reducing the above equation to
$$\phi^{t*}[d(\alpha_t(v_t))]=d(\partial_t \cdot h_t) \Rightarrow d[\phi^{t*}(\lambda_t(v_t))]=d(\partial_t \cdot h_t).$$ 

Note that such $v_t$ necessarily vanishes on $\partial W$ when assuming $\alpha_t=\alpha_0$ for $t\in [0,1]$. Therefore, if $\phi^t$ is the isotopy (relative $\partial W$) generated by $v_t$, it suffices to let $h_t:W\rightarrow \mathbb{R}$ be the unique solution to the ODE
$$\begin{cases}
\partial_t \cdot h_t=\phi^{t*}(\alpha_t(v_t)) \\
 h_0\equiv 0
\end{cases}$$
and we have computed the vector field $v_t$ generating the desired isotopy.

\begin{remark}
It is important to note that $v_t$ is in fact a {\em Hamiltonian isotopy}. Recall that given a 1-parameter family of functions $h_t:W\rightarrow \mathbb{R}$ on $(W,\alpha)$, there exists a unique family of vector fields $v_t$, named a {\em Hamiltonian vector field}, satisfying $\iota_{v_t}d\alpha_t=dh_t$. The flow generated by a Hamiltonian vector field preserves the symplectic structure and deforms any Liouville form by an exact term i.e. $\alpha_t:=\phi^{t*}(\alpha)=\alpha+dh_t$, which also gives the Liouville vector field of $\alpha_t$ as $Y_t=Y+v_t$, where $Y$ is the Liouville vector field of $\alpha$. When $h_t=h$ does not depend on $t$, a Hamiltonian vector field also preserves the level sets of $h:W\rightarrow \mathbb{R}$ which are generically codimension-1 submanifolds of $W$.
\end{remark}

Given Definition~\ref{lioudef}, it is natural to investigate the consequences of various dynamical conditions on the geometry of Liouville domains. A useful condition to consider is to assume that $Y$ is gradient-like with respect to a differentiable function $f:W\rightarrow \mathbb{R}$, i.e. for some $\delta>0$ and norm $||.||$, we have $$Y\cdot f \geq \delta (||Y||^2+||df||^2).$$ It is not hard to show that with the above condition, the zeros of $Y$ are exactly the critical points of $f$ and $Y\cdot f >0$ elsewhere, hence justifying the terminology gradient-like. See \cite{gradient} for more on this condition.

\begin{definition}
We call a Liouville domain $(W,\alpha)$ Weinstein, if, possibly after a Liouville homotopy through Liouville domains, its corresponding Liouville vector field $Y$ is gradient-like. Otherwise, we call it non-Weinstein.
\end{definition}

In the Weinstein case, one can take a Morse theoretic approach and hence achieve a topological theory of Liouville geometry in terms of Liouville handle decompositions. An important topological consequence of such Morse theoretic approach is the fact that in the Weinstein case, the underlying manifold can be constructed using at most half dimensional handles. In dimension 4, this means that the topological type of the underlying manifold is at most 2. In other words, the Morse skeleton in this case, which coincides with the set of all points of $W$ which remain in $W$, under the flow generated by the Liouville vector field $Y$, consists of at most 2-dimensional handles. Theory of Liouville geometry in the Weinstein case, as well as it close relation to complex geometry, is well-developed in the last few decades \cite{ce}.

A useful notion to describe the above idea is the {\em skeleton} of a (general) Liouville domain. In short, the skeleton of a Liouville domain $(W,\alpha)$ is the set of points which do not exit $W$ under the flow generated by the corresponding Liouville vector field $Y$.

\begin{definition}
Let $(W,\alpha)$ be a Liouville domain with Liouville vector field $Y$ defined by $\iota_Y d\alpha=\alpha$. We define
$$Skel(W,\alpha)=\cap_{t>0} Y^{-t} (W)$$
as the (Liouville) skeleton of $(W,\alpha)$. We also denote the skeleton by $Skel(\alpha)$ or $Skel(Y)$, when there is no risk of ambiguity.
\end{definition}

Note that $Skel(Y)$ is invariant under $Y^t$ and moreover, $Y^t$ defines a topological retraction of $W$ onto $Skel(Y)$. In the case that $(W,\alpha)$ is Weinstein, $Skel(Y)$ is in fact the Morse complex induced from the gradient-like vector field $Y$. Such Morse complex contains no cells of dimension more than half of $W$, i.e. here no 3 or 4 cells, and hence, the topological type of $W$ is at most 2-dimensional.

 On the other hand, the non-Weinstein case much less understood, including the most basic questions. Note that it is not hard to homotope a given Liouville domain to make the resulting Liouville vector field not gradient-like, hence the subtlety in the non-Weinstein case resides in the study of Liouville domains up to homotopy. The construction of the first examples of non-Weinstein Liouville domains were carried by McDuff \cite{mcduff} and Geiges \cite{geiges} in the early 90s, by introducing Liouville forms on 4-manifolds with disconnected boundary. More specifically, they constructed Liouville domains $(W,\alpha)$ with $W=[-1,1]\times M$ for some closed 3-manifold $M$. Such examples are automatically non-Weinstein, since the underlying manifold has the topological type of a closed 3-manifold (in Theorem~\ref{dynrig} of this paper, we show that the skeleton of these examples, as well as their generalizations by Mitsumatsu, are in fact 3-dimensional submanifolds of $W$ diffeomorphic to $M$). Their constructions was later generalized by Mitsumatsu \cite{mitsumatsu} given any 3-manifold equipped with an Anosov flow and other examples of non-Weinstein Liouville domains have been constructed since \cite{bowden} based on attaching symplectic handles to Mitsumatsu's examples, where the non-Weinstein claim again relies on the topoloigcal type of the underlying manifold. Moreover, Mitsumatsu's approach provides a foundation for a Liouville geometric theory of Anosov 3-flows \cite{hoz3}, which we refine in this paper, by showing that such construction can be promoted to a 1-to-1 correspondence (see Theorem~\ref{1to1}).

\subsubsection{Contact geometry as the boundary geometry for Liouville domains}\label{3.1.2}

The\ geometric interpretation of the boundary condition in Definition~\ref{lioudef} relates 4-dimensional Liouville geometry to 3-dimensional contact geometry. More specifically, on the 3-manifold $\partial W$ (with thew orientation induced as the boundary of $W$) we have
$$\alpha|_{\partial W} \wedge d\alpha|_{\partial W}=\frac{1}{2} \iota_Y(d\alpha \wedge d\alpha) >0,$$
where the inequality follows from the positive transversality of $Y$ at the boundary. Notice this condition is valid for codimension-1 submanifold of $W$ which is positively transverse to the Liouville vector field. Such 1-form on $\partial W$ or any poistively transverse submanifold of $W$ is in fact a {\em positive contact form}.

\begin{definition}
A 1-form $\alpha$ on an oriented 3-manifold $M$ is called a positive (negative) contact form if $\alpha \wedge d\alpha >0 \ (<0)$, with respect to the given orientation on $M$. The plane field $\xi=:\ker{\alpha}$ is then called a positive (negative) contact structure on $M$ and the pair $(M,\xi)$ is called a positive (negative) contact manifold. Whenever not mentioned, we assume the contact forms, structures and manifolds to be positive.
\end{definition}

Notice that by the Frobenius theorem, a contact structure is a (co-orientable) maximally non-integrable plane field. 

\begin{remark}\label{lowcontact}
In the above definition, a contact form is assumed to be at least $C^1$, while for various reasons in the literature, contact structures with lower regularity have also been studied. In this paper, we also deal with non-integrable plane fields with lower regularity. However, in our context there is always a high regularity vector field tangent to such plane fields along which the plane field is differentiable. Therefore, we still call a (possibly $C^0$) 1-form $\alpha$ a contact form, if for some vector field $X\subset \ker{\alpha}$, we have $ \alpha \wedge(\mathcal{L}_X \alpha) \neq 0$, a condition which is equivalent to $\alpha \wedge d\alpha\neq 0$ whenever $\alpha$ is $C^1$. Note that we can still make sense of positive vs. negative contact forms in this context by comparing $\alpha \wedge(\mathcal{L}_X \alpha)$ (whose kernel includes $\langle X \rangle$) and $\iota_X \Omega$ where $\Omega$ is any positive volume form on $M$.
\end{remark}

An example of a (positive) contact form can be given by $\alpha_{std}=dz-ydx$ on $\mathbb{R}^3$, where $\alpha \wedge d\alpha$ is the standard volume form on the Euclidean space and the kernel of such 1-form is called the standard contact structure on $\mathbb{R}^3$ and the Darboux theorem for contact structures (another application of the Moser technique) implies that any contact structure is locally equivalent to such plane field. We can also define the standard contact form on $\mathbb{S}^3$ by restricting the standard Liouville form of $\mathbb{R}^4$ described above to the unit sphere, i.e. $\alpha:=\alpha_{std}|_{U\mathbb{S}^3}$. This is in fact can be seen to be the one point compactification of the standard contact structure on $\mathbb{R}^3$.

As in Liouville geometry, contact geometry can be viewed from a dynamical viewpoint.  More specifically, given a contact form $\alpha$ on $M$, a unique vector field $R_\alpha$, named a Reeb vector field, exists satisfying $\iota_{R_\alpha} d\alpha=0$ and $\alpha(R_\alpha)=1$. Note that $R_\alpha$ is transverse to $\xi=\ker{\alpha}$ and the flow generated by $R_\alpha$ preserves $\xi$, since we have $\mathcal{L}_{R_\alpha}\alpha=0$ by the Cartan's formula. Equivalently, any vector field which is transverse to a contact structure $\xi$ and preserve it is a Reeb vector field for an appropriate choice of contact form $\alpha$. It is easy to check that for $(\mathbb{R}^3,\alpha_{std})$ the Reeb vector field can be computed as $R_{\alpha_{std}}=\partial_z$.

As discussed above, given any Liouville domain $(W,\alpha)$ and a codimension submanifold $i:M\rightarrow W$ which is positively transverse to the Liouville vector field, $(M,i^*\alpha)$ is a positive contact manifold. Conversely, for any contact manifold $(M,\alpha)$, one can define the pair $(W:=\mathbb{R}_s\times M,\alpha_W:=e^s \alpha)$ which is easy to check to be a Liouville 1-form defined on a non-compact manifold. Now, $(M,\alpha)$ can be embedded as $\{0\}\times M$ in $(W,\alpha_W)$. Moreover, it can be realized as a level curve of the Hamiltonian $h=e^s$, whose corresponding Hamiltonian vector can be seen to be the Reeb vector field $R_\alpha$. More generally, an embedding of any contact form $(M,e^f\alpha)$, it is easy to show that the map $i_f:M\rightarrow \mathbb{R}\times M$ defined by $i_f(p)=(e^,p)$ provides a similar {\em contact embedding}, where the corresponding Reeb vector field can again be realized as a Hamiltonian vector field restricted to such image.

This in particular, gives a a unique model for the Liouville geometry near any transverse submanifold. If the submanifold $M$ is in the interior of $W$, this means that there is a standard tubular neighborhood $N(M)\simeq (-T,T)_s\times M$ such that $\{ 0 \}\times M=M$ and $\partial_s$ is the Liouville vector field and in this coordinates, the Liouville form can be written as $\alpha=e^s(\alpha|_M)$. When the submanifold is a boundary component of a Liouville domian $(W,\alpha)$, this means that here is a tubular neighborhood of the boundary $N(\partial W)$ and the strict Liouville equivalence
$$\begin{cases}
\phi:N(\partial W)\rightarrow (-T,0]_s\times\partial W \\
\phi^*(e^s (\alpha|_{\partial W}))=\alpha
\end{cases}.$$

\subsubsection{Liouville manifolds: A non-compact theory}\label{3.1.3}

For various technical purposes, one might need to work on non-compact manifolds. In the context of this paper, that is mainly to enjoy a more convenient deformation theory, which is standard to define Liouville geometric invariants of Anosov 3-flows \cite{clmm}. The following definition capture the right condition we need to enforce at infinity (see \cite{ce}), when the underlying manifold has {\em cylinderical ends at infinity}.

\begin{definition}
The pair $(W,\alpha)$ is a Liouville manifold, if $W$ is a (necessarily non-compact) 4-manifold without boundary and $\alpha$ is a Liouville form such that the Liouville vector field $Y$ defined by $\iota_Y d\alpha=\alpha$ is complete and $W$ is convex, i.e. there exists an exhaustion $W=\cup_{k=1}^\infty W_k$ by Liouville domains $W_k \subset W$.
\end{definition}

One can extend the useful notion of skeletons to the non-compact setting of Liouville manifolds, using a given exhaustion as in the above definition, and by letting
$$Skel(W,\alpha)=\cup_{k=1}^{\infty}\cap_{t>0} Y^{-t} (W_k).$$

We then call a Liouville manifold finite-type if its skeleton is compact. This is equivalent to the existence of a proper function $f:W\rightarrow \mathbb{R}$ which is bounded from below and, outside a compact set of $W$, is Lyapunov for $Y$ and without critical points. In some sense, for finite-type Liouville manifolds, there is no geometric complications as one approaches infinity. Hence, they have a very similar theory as Liouville domains, but with a more convenient deformation theory.

More precisely, consider a Liouville domain $(W,\alpha)$ with the contact form defined on its boundary $\alpha_\partial=\alpha|_{\partial W}$ induced on the boundary. Then, one can define the Liouville form $\bar{\alpha}:=e^s \alpha_\partial$ on $\partial W \times [0,\infty)_s$ and it can be shown that using an appropriate gluing map, which is achieved by gluing Liouville flow lines of $(W,\alpha)$ and $(\partial W\times[0,\infty),\bar{\alpha})$, we get a Liouville form on $W\cup \partial W\times[0,\infty)$, which is unique up to strict Liouville equivalence. The result is in fact a finite-type Liouville manifold with the same skeleton as the Liouville domain $(W,\alpha)$. We call such Liouville manifold the {\em completion} of $(W,\alpha)$  Conversely, if $(\bar{W},\alpha)$ is a finite-type Liouville manifold, away from the compact skeleton, one can trim the cylinderical infinity ends of $\bar{W}$ to achieve a Liouville domain of the same topological type $(W\subset \bar{W},\alpha|_{\bar{W}})$, where the cut piece is strictly Liouville equivalent to $(\partial W\times[0,\infty)_s,e^s (\alpha|_{\partial \bar{W}}))$.

\begin{remark}\label{isoextension}
Suppose that two Liouville domains $(W_1,\alpha_1)$ and $(W_2,\alpha_2)$ are {\em strictly Liouville equivalent}, i.e. there exists a diffeomorphism $\phi:W_1\rightarrow W_2$ such that $\phi^*\alpha_2=\alpha_1$. Since the completion of a Liouville domain is defined, uniquely up to strict Liouville equivalence, via the relation $\mathcal{L}_{Y_i}\alpha_i=\alpha_i$ for $i=1,2$ which gives the standard model for Liouville geometry of the cylinderical end (see Section~\ref{3.1.2}). Any strict Liouville equivalence of Liouville domains can be naturally extended to their (uniquely defined up to strict Liouville equivalence) completions $(\bar{W}_1,\bar{\alpha}_1)$ and $(\bar{W}_2,\bar{\alpha}_2)$, i.e. there exist a map $\bar{\phi}:\bar{W}_1\rightarrow \bar{W}_2$ such that $\bar{\phi}^*\bar{\alpha}_2=\bar{\alpha}_1$. In general, we need $\phi$ to be $C^2$ in order to preserve the Liouville vector field as well. However, suppose we only have the strict Liouville equivalence $\phi$ to be only Hölder continuous $C^k$ (for $k>0$), but with enough partial derivatives existing to still satisfying $\phi_*(Y_1)=Y_2$ and $\phi^*\alpha_s=\alpha_1$ (we always have this if we assume $\phi$ is $C^2$, since that would mean $\phi^*(d\alpha_2)=d\alpha_1$). Then, there exists a $C^k$ homeomorphism $\bar{\phi}:(\bar{W}_1,\alpha_1)\rightarrow(\bar{W}_2,\alpha_2)$ extending $\phi$.
\end{remark}

The use of the Moser technique, with more care about what happens near infinity, yields the desired deformation theory of Liouville manifolds, which does not have the boundary limitations of Liouville domains, since we can exploit a complete Liouville vector field,  and is more suitable for defining invariants of Liouville geometry. First, we give the following definition for a Liouville homotopy of finite type Liouville manifolds. We note that Liouville homotopies are defined in more generality and abstract setting and the following is in fact a consequence of the original definition. But since we only care about finite type Liouville manifolds, this suffices for us as definition.

\begin{definition}
A smooth family of finite type Liouville manifolds $(W,\alpha_s)$ for $s\in[0,1]$ such that the closure $\cup_{s\in[0,1]} Skel(W\alpha_s)$ is compact is called a Liouville homotopy (of finite type Liouville manifolds).
\end{definition}

As an application of Moser technique, we have the following.

\begin{lemma}[\cite{ce}, Lemma~11.8]\label{moserclassic}
Suppose $(W,\alpha_t)$ for $t\in[0,1]$  is a Liouville homotopy of finite type Liouville manifolds. Then, there exists an isotopy $\psi^t:W\rightarrow W$ such that $\psi^{t*}\alpha_t-\alpha_0$ is exact and vanishes outside a compact set for all $t\in [0,1]$.
\end{lemma}

In the view of our discussion on Moser technique in Section~\ref{3.1.1}, this means that the Hamiltonian isotopy needed to construct any Liouville homotopy can be done in a compactly supported manner.


\subsection{Observation of Mitsumatsu and Eliashberg-Thurston:\\ A bridge between hyperbolic dynamics and contact geometry}\label{3.2}

The foundation of the Liouville geometry of Anosov 3-flows is based on an important but simple observation of Mitsumatsu \cite{mitsumatsu} and Eliashberg-Thurston \cite{et} in the mid 90s. That is, the vector field generating an Anosov 3-flow lies in the intersection of a transverse pair of positive and negative contact structures. 

The observation is a straight forward application of the Frobenius theorem, if one considers the flow action on the bi-sectors of the two invariant bundles. However, it paves the road from the world of hyperbolic dynamics to contact structures as stable high regularity differential geometric objects. More specifically, contactness is an open condition (unlike foliations) and therefore, if one achieves a contact geometric description os Anosov flows, the underlying contact structures can be perturbed to be as regular as the flow, usually at the price of breaking some symmetry. Furthermore, such model can be defined to be truly local, in the sense that deforming such model in a neighborhood does not affect the geometric desription in the rest of the underlying manifold. This is in sharp contrast to the common model of Anosov flows in terms of invariant foliations as such objects depend on the long term behavior of the flow (and hence, not truly local), and this often causes technical difficulties when one needs to deform the flow in some neighborhood (for instance for surgery or gluing purposes). In this sense, contact geometry is consistent with the stability features in Anosov dynamics (see Remark~\ref{stability}).

One might hope that the mentioned condition is sufficient to characterize when a 3-flow is Anosov. But in fact there are many non-singular flows with such properties. This includes flows on manifolds like $\mathbb{S}^3$ and $\mathbb{T}^3$, which do not admit any Anosov flows. However, it turns out that the condition of lying in the intersection of such pair of contact structures has a dynamical interpretation.

\begin{convention}\label{orientableflow}
From now on, for any projectively Anosov 3-flow $X^t$ with the splitting $TM/\langle X\rangle E\oplus F$, we assume $E$ and $F$ to be orientable, i.e. we assume our projectively Anosov flows (or more particularly, partially hyperbolic or Anosov flows) to be {\em orientable}. This is always possible after possibly lifting to a double cover of the underlying manifold $M$ (since we are assuming $M$ to be orientable itself). This should be noted however that, this is mostly for more straightforward statements in the contact and symplectic geometric theory of such flows, and we expect similar geometric phenomena to be present in the unorientable case, as it relies on the local behavior of the flow.
\end{convention}

\begin{lemma}[Eliashberg-Thurston \cite{et}, Mitsumatsu \cite{mitsumatsu} 1995]\label{pabicontact}
A non-vanishing vector field $X$ on a 3-manifold $M$ is projectively Anosov, if and only if $X\subset \xi_- \cap \xi_+$, where $\xi_-$ and $\xi_+$ negative and positive contact structures, respectively, and $\xi_- \pitchfork \xi_+$.
\end{lemma}

Notice that in the above lemma, we are assuming the contact structures to be orientable and the projectively Anosov flows to have orientable invariant bundles.

   \begin{figure}[h]
\centering
\begin{overpic}[width=1\textwidth]{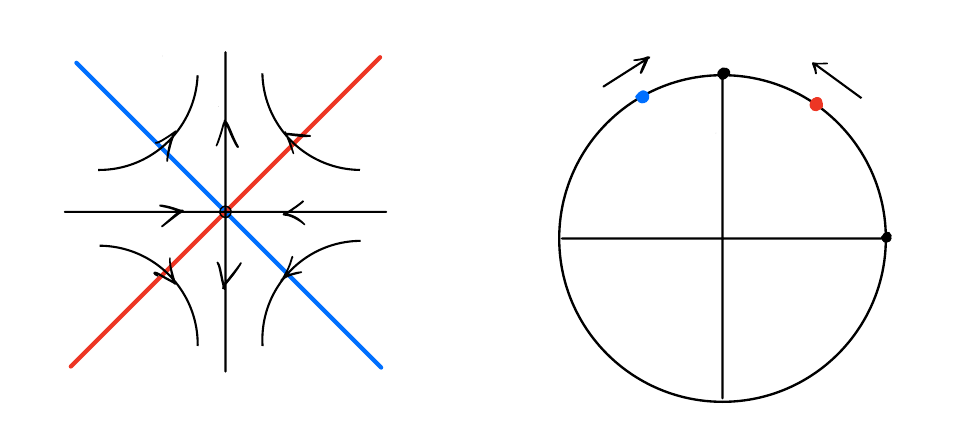}
    \put(105,189){$E^{wu}$}

        \put(189,175){$\xi_+$}
          \put(22,175){$\xi_-$}
        \put(188,104){$E^{ws}$}
        
                \put(432,94){$E$}
                  \put(345,183){$F$}
                   \put(405,177){$X$-action}
                   
                        \put(312,150){$\xi_-$}
          \put(377,150){$\xi_+$}

  \end{overpic}

\caption{Bi-contact condition for Anosov flows and the action on the projectified normal bundle}
\end{figure}
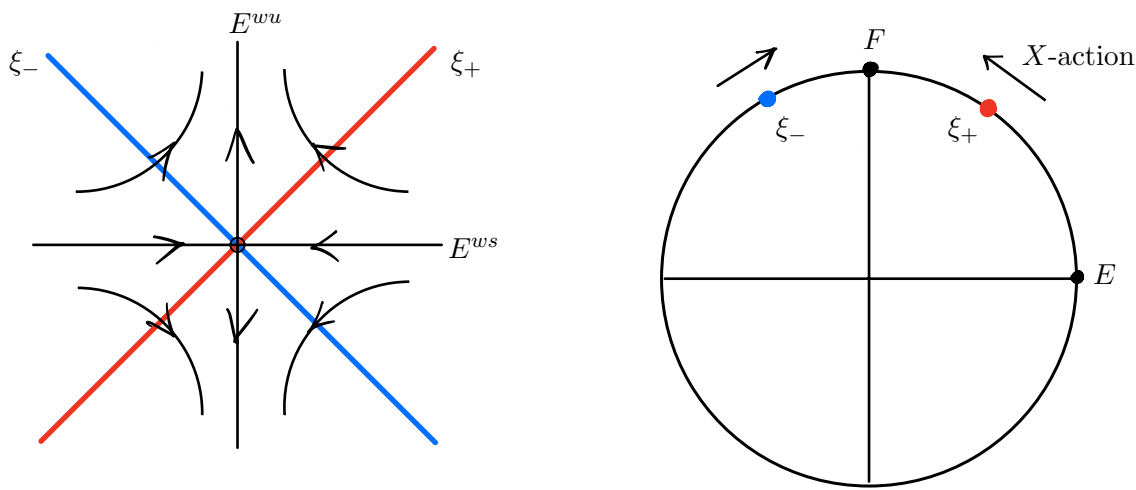

For discussion on projective Anosovity, see Section~\ref{2.2.1}. Thanks to this lemma, one can construct many projectively Anosov flows using contact geometry. For instance, any pair of positive and negative tight contact structures on $\mathbb{T}^3$ can be made transverse and hence, yielding a tight projectively Anosov flow \cite{et,hoz3}. Moreover, in any homotopy class of plane fields, there are unique, up to isotopy, positive and negative overtwisted contact structures, which can be made transverse assuming vanishing Euler class, producing a large class of overtwisted projectively Anosov flows \cite{et,adn}. In particular, this implies that any 3-manifold carries a non-Anosov projectively Anosov flow.

The above lemma justifies the following definition.

\begin{definition}
The pair $(\xi_-,\xi_+)$ is called a bi-contact structure, if $\xi_-$ and $\xi_+$ are negative and positive contact structures, respectively. Furthermore, $(\xi_-,\xi_+)$ is called a transverse bi-contact structure, if $\xi_- \pitchfork \xi_+$. We say $(\xi_-,\xi_+)$ is a supporting bi-contact structure for the non-vanishing (projectively Anosov) vector field $X$, if it is transverse and $X\subset \xi_-\cap \xi_+$.
\end{definition}

The above lemma then shows a correspondence between homotopy classes of projectively Anosov flows and transverse bi-contact structures, while for a study of non-transverse bi-contact structures one can refer to \cite{colin,massoni2}. As it will be shown in Section~\ref{3.5}, a homotopy class of projectively Anosov flows can contain distinct flows up to orbit equivalence and therefore, the equivalence between bi-contact structures and projectively Anosov flows given in Lemma~\ref{pabicontact} is not a 1-to-1 correspondence if one considers flows up to orbit equivalence. 

A natural question then is whether the we can add contact geometric conditions to a bi-contact structure in order to make the resulting dynamical interpretation equivalent to Anosovity of flows. It turns out that the answer is affirmative and in fact, there are different ways to achieve such characterization. The one which is the focus of this paper is built on the construction of Mitsumatsu and is closely related to the historical motivations of our study. We will further show in this paper that such characterization in fact can be promoted, in an appropriate sense, to a 1-to-1 correspondence between Anosov flows (up to $C^1$-conjugacy) and such geometric objects (see Theorem~\ref{1to1}).


\subsection{Mitsumatsu's construction and the Liouville geometric theory of Anosov flows in dimension 3}\label{3.3}

In this section, we discuss the construction of 4-dimensional Liouville domains with disconnected boundary given a closed 3-manifold equipped with an Anosov 3-flow. The construction was carried by Mitsuamtsu \cite{mitsumatsu} for smooth volume preserving Anosov 3-flows, but smoothness and the existence of invariant volume form can be dropped easily. This construction was later used as a basis for the symplectic geometric theory of Anosov 3-flows in \cite{hoz3} and one of the main goals of this paper is to dig deeper in understanding the resulting Liouville geometry, concluding that the Liouville geometry in this case is determined by the underlying Anosov flow in a strong sense and is independent of all the choices made along the way.

\begin{convention}\label{orientation}
The description of an Anosov flow does not induce a canonical orientation on the underlying manifold. Therefore, we here fix our orientation convention for the computations to follow (we use the same convention for a general projectively Anosov flow).

Let $M$ be a closed oriented 3-manifold and $X$ a non-vanishing vector field, generating an Anosov flow on $M$. The 1-forms $\alpha_u$ and $\alpha_s$, whose kernels coincides with the invariant weak stable and unstable bundles, respectively, can be defined (we are assuming the weak invariant bundles to be oriented, by Convention~\ref{orientableflow}). Our orientation convention is to choose $\alpha_s$ and $\alpha_u$ such that the volume form $\Omega$ defined by $\iota_X \Omega=\alpha_s \wedge \alpha_u$ is positive. This is equivalent to $(X,e_s,e_u)$ being an oriented basis of $TM$, where $e_s\in E^s$, $e_u\in E^u$ and $\alpha_s(e_s),\alpha_u(e_u)>0$.
\end{convention}


We also choose orientations such that the volume form $\Omega$ defined by $\iota_X \Omega=\alpha_s \wedge \alpha_u$ is positive. The weak invariant bundles are known to be $C^1$ \cite{hps,reg,reg2} and hence, $\alpha_u$ and $\alpha_s$ are $C^1$ 1-forms (while they are infinitely many times differentiable along the flow). In particular, we have
$$\begin{cases}
\mathcal{L}_X \alpha_u=\iota_X d\alpha_u=r_u\alpha_u \\
\mathcal{L}_X \alpha_s=\iota_X d\alpha_s=r_s\alpha_s
\end{cases},$$
where $r_u$ and $r_s$ are the expansion rates of the unstable and stable bundles, respectively, satisfying $r_s<0<r_u$ (note that we can assume $r_u$ and $r_s$ to be induced from an adapted norm and be $C^1$ by Lemma~\ref{simiclemma}). Let
$$\begin{cases}
\alpha_+:=\alpha_u-\alpha_s \\
\alpha_-:=\alpha_u+\alpha_s
\end{cases}$$
and notice that $\ker{\alpha_+}\pitchfork \ker{\alpha_-}=\langle X \rangle$ and
$$\begin{cases}
\alpha_+\wedge d\alpha_+=-\alpha_u \wedge d\alpha_s -\alpha_s \wedge d\alpha_u=(r_u-r_s)\Omega \\
\alpha_-\wedge d\alpha_-=\alpha_u \wedge d\alpha_s +\alpha_s \wedge d\alpha_u=(r_s-r_u)\Omega
\end{cases}.$$ 

In particular, $\alpha_+$ and $\alpha_-$ are positive and negative contact forms, respectively, and $(\xi_-:=\ker{\alpha_-},\xi_+:=\ker{\alpha_+})$ is a supporting bi-contact structure for $X$.

Now, consider the compact 4-manifold $W:=[-1,1]_s\times M$ ($s$ is used as the parameter in the interval direction), define
$$\alpha:=(1-s)\alpha_- +(1+s)\alpha_+=2\alpha_u-2s\alpha_s$$ and compute
$$d\alpha=(1-s)d\alpha_- +(1+s)d\alpha_++ds\wedge(\alpha_+-\alpha_-)=2d\alpha_u-2sd\alpha_s-2ds\wedge \alpha_s$$
and therefore,
$$d\alpha \wedge d\alpha =-4ds \wedge \alpha_s \wedge d\alpha_u=4r_u ds\wedge \Omega,$$
indicating that $\alpha$ is in fact a Liouville form on $W$, thanks to the absolute expansion on $E^{u}$, i.e. $r_u>0$. To see that $(W,\alpha)$ is a Liouville domain, we note that (here $-M$ is $M$ with reversed orientation)
$$\partial W =\big\{ \{-1\}\times -M\big\} \cup \big\{ \{1\}\times M \big\}$$
as oriented manifolds and
$$\begin{cases}
\alpha|_{\{1\}\times M}=2\alpha_+ \\
\alpha|_{\{-1\}\times -M}=2\alpha_-
\end{cases},$$
which means that the restriction of $\alpha$ to $\partial W$ is a positive contact form. One should note that in the above construction, $\alpha$ is only as regular as the invariant bundles and hence, only $C^{1+}$ in general. 
However, it is not hard to see that $\alpha_-$ and $\alpha_+$ can be $C^1$-approximated with $C^\infty$ contact forms $\tilde{\alpha}_-$ and $\tilde{\alpha}_+$ (at the cost of breaking the symmetry in the above computations), still containing $X$ in their kernel (the approximation can be performed on $TM /\langle X \rangle$ to make sure the inclusion of the vector field $X$ in the plane fields is unaffected). The fact that the approximation is $C^1$ makes it easy to preserve the contactness conditions for $\tilde{\alpha}_-$ and $\tilde{\alpha}_+$, and the resulting $\tilde{\alpha}$ constructed from the interpolation of these two contact forms stays a Liouvilel form on $W$, as long as the approximation is $C^1$-small.

It is shown in \cite{hoz3} that using more careful approximations, one can achieve the same construction for flows generated by $C^1$ vector fields and without assuming the weak invariant bundles to be $C^1$. Note that in this case (without assuming the derivatives of $X$ being Hölder continuous), the regularity theory \cite{hps,reg,reg2} discussed in Section~\ref{2.3} fails. As a matter of fact, bypassing the low regularity of the invariant bundles is the essence of the Liouville geometric theory of Anosov flows.

\begin{remark}
Even though we restrict our attention to $C^\infty$ flows in the paper, we still need to be familiar with the mentioned approximation ideas and therefore, we will briefly discuss them in Section~\ref{3.6}. The reasons are twofold. Firstly, we want to generalize our constructions and theory to the category of non-singular partially hyperbolic flows where the flow being $C^\infty$ does not imply that the weak invariant bundles are necessarily $C^1$. In fact, these invariant plane fields might even fail to be uniquely integrable in this case \cite{et}. Secondly, it is often considerably easier to carry the computations in the {\em symmetric} construction of Liouville forms, while, except in the case of algebraic Anosov flows, that is only possible if we allow low regularity. Hence, we would like to know how such geometries with low regularity can be approximated by asymmertric high regularity models. The upshot of this paper is that Liouville geometry is not sensitive to such conventions and approximations.
\end{remark}


The construction of Mitsumatsu can be exploited to achieve a Liouville geometric characterization of Anosov 3-flows as shown in \cite{hoz3}. To see this, first note that similar to the above, the absolute contraction of $E^s$ gives rise to another Liouville domain by considering the 1-form
$$\bar{\alpha}=(1-s)\alpha_- -(1+s)\alpha_+,$$
where we can similarly show
$$d\bar{\alpha}\wedge d\bar{\alpha}=-4ds\wedge \alpha_u \wedge d\alpha_s=-r_s ds\wedge \Omega.$$

The contact condition on the boundary of $W$, as well as the approximation claims, apply similar to above. This means that $X$ being an Anosov vector field on $M$ yields the construction of two distinct Liouville domains $([-1,1] \times M,\alpha)$ and $([-1,1] \times M,\bar{\alpha})$ as above.

Conversely, suppose there exists a pair of negative and positive contact forms $(\alpha_-,\alpha_+)$ on $M$, whose kernels form a transverse bi-contact structure supporting a non-vanishing vector field $X\subset \ker{\alpha_-}\cap\ker{\alpha_+}$, and the 1-form $\alpha:=(1-s)\alpha_-+(1+s)\alpha_+$ is a Liouville form on $W=[-1,1]_s\times M$. Note that by Lemma~\ref{pabicontact}, the existence of the supporting bi-contact structure for $X$ yields the existence of weak invariant bundles via projective Anosovity, i.e. $TM/\langle X \rangle \simeq E\oplus F$ (as well as relative expansion of $F$ compared to $E$). Furthermore, assume the interpolation of the kernels is through the weak (relatively) stable bundle, i.e. for some $s_0\in [-1,1]$ we have $\ker{\alpha}:=\{(1-s_0)\alpha_-+(1+s_0)\alpha_+\}=E$. A simple computation shows
$$\iota_X \iota_{\partial_s}(d\alpha \wedge d\alpha)=2 (\mathcal{L}_X \alpha)\wedge (\mathcal{L}_{\partial_s} \alpha) =2 (\mathcal{L}_X \alpha) \wedge (\alpha_+-\alpha_-),$$
where at $s=s_0$ boils down to
$$\iota_X \iota_{\partial_s}(d\alpha \wedge d\alpha) |_{s=s_0}=2   (\mathcal{L}_X \alpha_u) \wedge (\alpha_+-\alpha_-)|_{s=s_0}=2   (r_u\alpha_u) \wedge (-2\alpha_s)|_{s=s_0}=4r_u\alpha_s \wedge \alpha_u,$$
where $\alpha_u:=\{(1-s_0)\alpha_-+(1+s_0)\alpha_+\}$ is a 1-form with $\ker{\alpha_u}=E$ and expansion rate $r_u$, and $\alpha_s$ is the 1-form defined by $\alpha_s(F)=0$ and $\alpha_s(e)=(\alpha_--\alpha_+)(e)\neq 0$ for any $e\in F$. Keeping track of our orientation conventions (see Convention~\ref{orientation}), the Liouville condition of $\alpha$ implies that $r_u>0$, i.e. $X$ induces an absolute expansion on $F$ (which now can be thought as $E^{wu}$), i.e. $X$ is partially hyperbolic.

   \begin{figure}[h]
\centering
\begin{overpic}[width=1\textwidth]{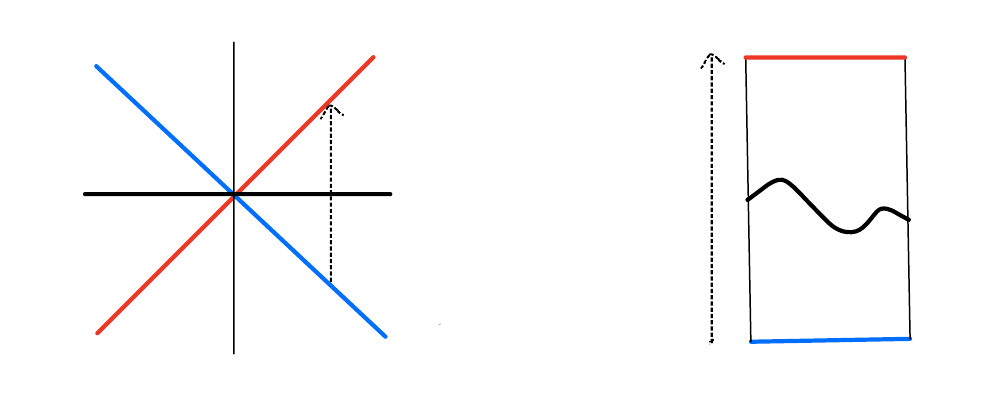}
    \put(108,182){$E^{wu}$}
        \put(164,130){interpolation}
            \put(177,118){through}
        \put(162,175){$\xi_+=\ker{\alpha_+}$}
          \put(32,175){$\xi_-=\ker{\alpha_-}$}
        \put(189,104){$E^{ws}$}
        
                \put(386,102){$\Lambda$}
                  \put(355,175){$(\{1\}\times M,2\alpha_+)$}
                   \put(350,24){$(\{-1\}\times M,2\alpha_-)$}
                     \put(334,175){$s$}
                        \put(273,99){$[-1,1]_s\times M$}

  \end{overpic}

\caption{3- vs 4-dimensional interpretation of a Liouville pair supporting an Anosov 3-flow}
\end{figure}
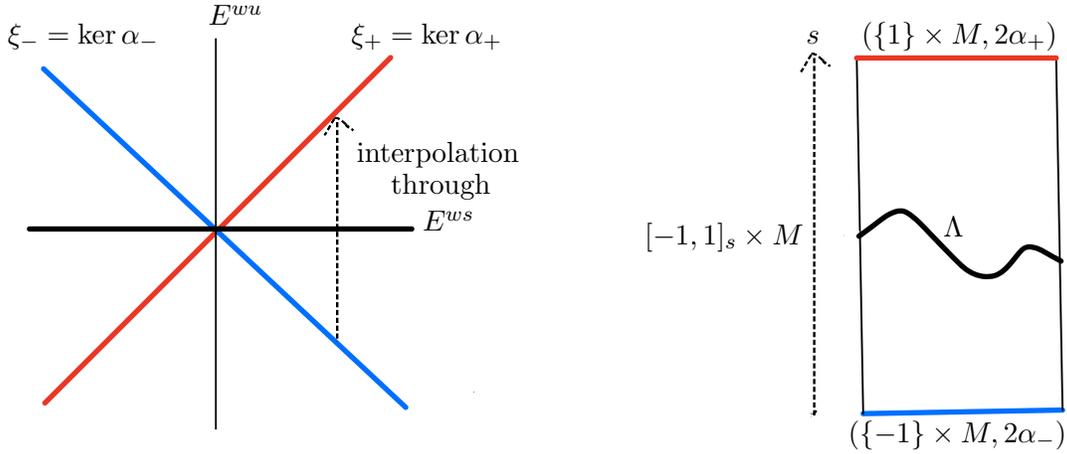

Similarly, we can show that the Liouville condition for the 1-form $\bar{\alpha}=(1-s)\alpha_--(1+s)\alpha_+$, whose kernel interpolates through $\ker{\alpha_-}$ and $\ker{\alpha_+}$ through $F$, implies $r_s<0$ and hence, the absolute contraction of $E$, (which now can be thought as $E^{ws}$). Therefore, the two assumption together impliy the  hyperbolicity of splitting and therefore, Anosovity of the flow is achieved by the Doering's lemma (Lemma~\ref{doering}). 

Essentially, the above argument yields Liouville geometric characterization of Anosov 3-flows. More specifically, the existence of a transverse bi-contact structures guarantees the existence of weak invariant bundles and a domination relation, where in an interpolation between the two contact structures and passing through each weak bundle, the Liouville condition corresponds to the absolute expansion and contraction of the dominating and dominated bundles, respectively, as required to establish the hyperbolicity of the flow.

To encapsulate this in a theorem, we first name the pairs of contact forms used in this construction, while noting an orientation convention for consistency.

\begin{definition}
The pair of 1-forms $(\alpha_-,\alpha_+)_l$ is called a {\em linear Liouville pair}, if $(\ker{\alpha_-},\ker{\alpha}_+)$ is a transverse bi-contact structure and the 1-form defined by $$\alpha:=(1-s)\alpha_-+(1+s)\alpha_+$$
is a Liouville form on $[-1,1]_s\times M$. We call $(\alpha_-,\alpha_+)$ a {\em supporting} linear Liouville pair for the flow generated by the non-vanishing vector field $X$, if $X$ is a projectively Anosov flow supported by $(\ker{\alpha_-},\ker{\alpha_+})$ with the dominated splitting $TM/\langle X \rangle \simeq E\oplus F$ and $\ker{\{ (1-s_0)\alpha_-+(1+s_0)\alpha_+\}}=E$ for some $s_0\in[-1,1]$.
\end{definition}

We also consider the space of consider linear Liouville pairs $(\alpha_-,\alpha_+)_l$, where $\alpha_-$ and $\alpha_+$ are $C^\infty$ contact forms, denoted by $\mathcal{LLP}(M)$, and naturally equip it with the topology inherited as an open subset of $\Omega^{1}(M)\times \Omega^{1}(M)$, where $\Omega^{1}(M)$ is the space of $C^\infty$ 1-forms on $M$. A natural map is then defined as 
$$\begin{cases}
P:\mathcal{LLP}(M)\rightarrow S\chi (M) \\
P(\alpha_-,\alpha_+)_l= [X]\subset \ker{\alpha_-}\cap \ker{\alpha_+}
\end{cases},$$
where by $S\chi (M)$ refers to the space of $C^\infty$ vector fields on $M$ up to positive reparametrization, and the above maps any linear Liouville pair to the (positive reparametrization class of) a (non-singular projectively Anosov) vector field $X$ supported by its kernel and respecting the orientation as discussed above.
We will later justify a natural refinement of such map by showing that a somewhat canonical choice of parametrization can be made.

We note that in the above definition a choice of orientation is implicit. In other words, a (linear) Liouville pair $(\alpha_-,\alpha_+)_l$ defines an orientation on $\ker{\alpha_-}\cap\ker{\alpha_+}$, which is characterized by the fact that, for any projectively Anosov vector field $X\subset \ker{\alpha_-}\cap\ker{\alpha_+}$ representing the image of the above map, the interpolation of kernels in the above definition is through the dominated bundle, i.e. $\ker{\{ (1-s_0)\alpha_-+(1+s_0)\alpha_+\}}=E$ for some $s_0\in[-1,1]$. This convention will be justified in Section~\ref{4}, when we show that this orientation agrees with the orientation induced from the resulting Liouville vector field. Using this terminology, the above construction implies that for any Anosov vector field $X$, there are two Liouville pairs $(\alpha_-,\alpha_+)_l$ and $(\alpha_-,-\alpha_+)_l$, one supporting $X$ and the other supporting $-X$.

Moreover, the construction of $C^1$ symmetric Liouville pairs show that, even when dealing with a high regularity flow, it is useful to consider lower regularity Liouville pairs, while assuming more differentiability along the flow. In our context, we typically want to restrict our attention to $C^\infty$ flows (see Remark~\ref{stability}), while the 1-forms involved in the symmetric construction are in general as regular as the dominated bundle $E$, which is $C^{1+}$ in the Anosov case, but is only Hölder continuous (i.e. $C^k$ for some $0<k$), in the category of partially hyperbolic flows. But even for such low regularity 1-forms we can assume extra differentiability along the flow, i.e. the Lie derivative of the 1-forms along the flow can be also assumed to be $C^k$. We will see in Section~\ref{3.6} that such 1-forms can be approximated appropriately by $C^{\infty}$ 1-forms.
Therefore, we also consider the space of $C^k$ linear Liouville pairs $(\alpha_-,\alpha_+)_l$ such that for some $C^\infty$ projectively Anosov vector field $X \subset  \ker{\alpha_-}\cap \ker{\alpha_+}$ with $C^k$ dominated bundle $E$, and we have $\mathcal{L}_X \alpha_-$ and $\mathcal{L}_X \alpha_+$ are also $C^k$. We denote the sapce of such linear Liouville pairs by $\mathcal{LLP}^{k*}(M)$. The construction of symmetric linear Liouville pairs above yields linear Liouville pairs of the form $(\alpha_u+\alpha_s,\alpha_u-\alpha_s)_l$, which are in $\mathcal{LLP}^{1*}(M)$ in the Anosov case.

We can now summarize the characterization of Anosov 3-flows in terms of linear Liouville pairs \cite{hoz3} as follows.

\begin{theorem}(H. 2020 \cite{hoz3})\label{anosovchar}
Let $X$ be a non-vanishing vector field on a closed oriented 3-manifold. Then, the followings are equivalent:

(1) The flow generated by $X$ is Anosov;

(2) There exists negative and positive contact forms $\alpha_-$ and $\alpha_+$, respectively, such that $(\ker{\alpha_-},\ker{\alpha_+})$ is a supporting bi-contact structure for $X$, and $(\alpha_-,\alpha_+)_l$ and $(\alpha_-,-\alpha_+)_l$ are linear Liouville pairs.
\end{theorem}


\subsection{Liouville interpolation systems: A generalized setting}\label{3.4}

While the construction of the previous subsection provides a Liouville geometric theory of Anosov 3-flows, in other contexts, one might want to use a model of Liouville geometry other than the linear one discussed. For instance, the non-compact model (as Liouville manifolds) is technically required in order to define Liouville geometric invariants of homotopies of Anosov 3-flows \cite{clmm}. A natural candidate in such setting is the exponential model, i.e. a Liouville form defined as
$$\alpha:=e^{-s}\alpha_- +e^s\alpha_+$$
on the (non-compact) 4-manifold $\mathbb{R}_s\times M$, where as before, $\alpha_-$ and $\alpha_+$ are negative and positive contact forms defining a projectively Anosov flow when transverse. Such pairs are called {\em exponential Liouville pairs} in the literature and it is easy to see that they satisfy the conditions required for Liouville manifolds (i.e. completeness of the Liouville vector field). We denote by $(\alpha_-,\alpha_+)_e$. These were studied in \cite{mnw,massoni,clmm}. In particular, \cite{mnw} studies them without the transversality assumptions and with generalizations to higher dimensions, \cite{massoni} shows that the argument of Theorem~\ref{anosovchar} carries over such setting in dimension 3, and \cite{clmm} exploits the formulation of \cite{massoni} to define new Liouville geometric invariants of Anosov 3-flows. 

Similar to linear Liouville pairs, we can consider the space of such $C^\infty$ pairs by $\mathcal{ELP}(M)$, naturally inheriting a topology as an open subset of $\Omega^{1}(M)\times \Omega^{1}(M)$, while the space of low regularity exponential Liouville pairs with additional conditions similar to the definition of $\mathcal{LLP}^{k*}(M)$ (see Section~\ref{3.3}), is denoted by $\mathcal{ELP}^{k*}(M)$.

While the constructions and arguments of \cite{massoni} for the exponential version of Theorem~\ref{anosovchar} are similar to the ones for linear Liouville pairs, some ambiguity remains on whether there is any genuine difference between the constructed Liouville geometries other than compactness. Moreover, for other motivations, other models might be appropriate. One of the main purposes of this paper is to show Liouville geometry is encapsulated in the interpolation of a plane field between a negative and positive contact structures and is completely independent of the chosen model. Therefore, we want to reformulate the theory in a generalized setting, namely {\em Liouville interpolation systems}, which generalizes and unifies the previous theories, while leaving room for other modifications needed in other contexts.

\begin{definition}\label{lis}
If $(\alpha_-,\alpha_+)$ is a transverse bi-contact form on $M$, we call $(\alpha_-,\alpha_+)_{(\lambda_-,\lambda_+)}$ a {\em Liouville interpolation system} (or LIS) if the 1-form defined by
$$\alpha=\lambda_+\alpha_+ +\lambda_- \alpha_-$$
is a Liouville form on $\mathbb{R}\times M$,
where

(1) $\alpha_\pm$ is a $\pm$ contact structure on $M$;

(2) $\lambda_{\pm}:\mathbb{R}\times M \rightarrow\mathbb{R}_+$ are positive $C^{\infty}$ functions;

(3) $s\mapsto \ln(\frac{\lambda_+(s,x))}{\lambda_-(s,x)})$ is a positive reparametrization of $\mathbb{R}$ for any $x\in M$.

We define $\mathcal{LIS}(M)$ to be the space of LIS s with the topology induced as an open subset of $\Omega^{1}(M)\times \Omega^{1} (M) \times C^\infty(\mathbb{R}\times M) \times C^\infty(\mathbb{R}\times M)$.
\end{definition}

Notice that
$$\alpha=\lambda_+[\alpha_++\frac{\lambda_+}{\lambda_-}\alpha_-]$$
means that $\partial_s\cdot (\frac{\lambda_+}{\lambda_-})>0$, implied from condition (3) above, forces $\ker{\alpha}$ to interpolate as a plane field between $\ker{\alpha_+}$ and $\ker{\alpha_-}$ as $s$ increases and in the above definition, we even allow such interpolation functions $\lambda_{\pm}$ to depend on the point in $M$. This can be summarized in the fact that with such local conditions, we have
$$\alpha \wedge \mathcal{L}_{\partial_s}\alpha \neq 0.$$
Moreover, as we will see in Lemma~\ref{liscomp}, the pair $(\mathbb{R}\times M, L(\alpha_-,\alpha_+)_{(\lambda_-,\lambda_+)})$ is in fact a (finite type) Liouville manifold for any LIS, thanks to the assumptions we force on $\lambda_{\pm}$ in the above definition.

\begin{remark}\label{yorient}
As we will see in Section~\ref{4} and \ref{5}, the geometry of interpolation in essence boils down to the fact that $$\alpha \wedge \mathcal{L}_{\partial_s}\alpha \neq 0,$$
which ($Y$ being the Liouville vector field) considering $\iota_Y\iota_{\partial_s}(d\alpha\wedge d\alpha )=2\alpha\wedge \mathcal{L}_{\partial_s}\alpha ,$ is equivalent to $$Y\pitchfork \partial_s.$$ For most local computations we do, this condition is enough and this is indeed enough for any compact theory. However, discussion in the non-compact setting is unavoidable as it provides the pragmatic deformation theory of our objects. The rest depends on the condition we impose at infinity. For us, that is the setting of Liouville manifolds, i.e. assuming completeness of the corresponding Liouville vector field. We will show that the condition (3) in the Definition~\ref{lis} guarantees the resulting Liouville vector field to be complete, by showing its flow is conjugate to one induced from an exponential LIS (which is known to be complete). Relaxations of this condition can be and in fact should be considered. While the use of Moser technique shows that any complete Liouville form of the form $\lambda_-\alpha_-+\lambda_+\alpha_+$ on $\mathbb{R}\times M$ is strictly Liouville equivalent to one induced from a LIS, the linearization of the Liouville vector field at the skeleton provides an important complete example which does not satisfy the condition of our definition. However, our study of such examples i still useful to determine the regularity of the strong normal Lagrangian bundles.
\end{remark}

We naturally define the following continuous maps sending a LIS to its induced Liouville form on $\mathbb{R}_s\times M$ and the positive reparametrization class of its supported (projectively Anosov) vector field, extending similar definitions for linear Liouville pair.

\begin{definition}\label{support}
We define the continuous maps $$\begin{cases}
L:\mathcal{LIS}(M)\rightarrow \Omega^{1} (\mathbb{R}\times M) \\
L(\alpha_-,\alpha_+)_{(\lambda_-,\lambda_+)}=\lambda_+\alpha_+ +\lambda_- \alpha_-
\end{cases}$$
and
$$\begin{cases}
P:\mathcal{LIS}(M)\rightarrow S\chi(M) \\
P(\alpha_-,\alpha_+)_{(\lambda_-,\lambda_+)}=[X]
\end{cases},$$
where the orientation of $[X]$ is chosen such that the interpolation of kernels is through the dominated bundle of $X$. We say a LIS $(\alpha_-,\alpha_+)_{(\lambda_-,\lambda_+)}$ supports a (projectively Anosov) vector field $X$ (or the flow generated by it), if $X\in P(\alpha_-,\alpha_+)_{(\lambda_-,\lambda_+)}$
\end{definition}

 \begin{remark}\label{lowliouville}

Similar to the linear and exponential cases, it is useful to consider Liouville interpolation systems with low regularity. More specifically, given a $C^\infty$ (projectively Anosov) vector field with $C^k$ invariant plane fields (for $k>0$ possibly less than $1$), $(\alpha_-,\alpha_+)_{(\lambda_-,\lambda_+)})$ is a $C^{k*}$ LIS, if 

(1) $\alpha_-$ and $\alpha_+$ are negative and positive $C^{k*}$ contact forms in the sense of Remark~\ref{lowcontact}, such that $\mathcal{L}_X \alpha_-$ and $\mathcal{L}_X\alpha_+$ are also $C^k$ (in this situation, we say, $\alpha_-$ and $\alpha_+$ are $C^{k*}$); and 

(2) $\lambda_-$ and $\lambda_+$ in Definition~\ref{lis} are also $C^k$ with both $\mathcal{L}_X \lambda_\pm$ and $\mathcal{L}_{\partial_s} \lambda_\pm$ are also $C^k$. Note that $k=\infty$ corresponds to $\mathcal{LIS}(M)$ defined above (in this situation, we say, $\lambda_-$ and $\lambda_+$ are $C^{k*}$).

 We can still make sense of the Liouville condition in the low regularity setting introduced above. Notice that when $\alpha=L(\alpha_-,\alpha_+)_{(\lambda_-,\lambda_+)}$ is $C^1$, we have
 $$d\alpha \wedge d\alpha >0 \Longleftrightarrow \iota_X\iota_{\partial_s}(d\alpha \wedge d\alpha) >0 \Longleftrightarrow \mathcal{L}_X \alpha \wedge \mathcal{L}_{\partial_s} \alpha >0,$$
 where in the last two inequalities, the comparison is made with respect to $\iota_X \iota_{\partial_s} \Omega$ as a basis for the space of 2-forms vanishing on $\langle X,\partial_s \rangle$, and $\Omega$ being any positive volume form on $\mathbb{R}\times M$. Since $\partial_s$ and $X$ are assumed to be $C^\infty$, the last condition still makes sense for 1-forms induced from $\mathcal{LIS}^{k*}(M)$ for any $k>0$ and we consider that to be the Liouville condition in such category of 1-forms, i.e. 1-forms annihilating $\langle \partial_s,X\rangle$.
 \end{remark}
 
 We have the following.


\begin{lemma}\label{v}
Using the above notation, for any $0<k\leq \infty$, the map
$$\begin{cases}
V:\mathcal{LIS}^{k*}(M)\rightarrow Annih^{2;k}(\mathbb{R}\times M; \langle X, \partial_s\rangle) \\
V(\alpha_-,\alpha_+)_{(\lambda_-,\lambda_+)}=[\mathcal{L}_X L(\alpha_-,\alpha_+)_{(\lambda_-,\lambda_+)}] \wedge [\mathcal{L}_{\partial_s} L(\alpha_-,\alpha_+)_{(\lambda_-,\lambda_+)}]
\end{cases}$$
are continuous, where $Annih^{2;k}(\mathbb{R}\times M; \langle X, \partial_s\rangle)$ is the space of $C^k$ 2-forms on $\mathbb{R}\times M$ annihilating  $\langle \partial_s,X\rangle$.
\end{lemma}

As mentioned before, one of the goals of this paper is to unify the previous interpolation models, i.e. linear vs. exponential, while in the above only non-compact version of an LIS only is introduced. In fact, we will show in this paper that fixing the flow any supporting linear Liouville pair can be strictly embedded in any supporting exponential Liouville pair, emphasizing the the difference between the two models is only in compactness. Hence, we carry most computations in the non-compact setting, as that provides the context for a convenient deformation theory. However, to establish the non-trivial fact mentioned above, we also define a compact LIS $(\alpha_-,\alpha_+)_{(\lambda_-,\lambda_+)}$ by only dropping the conditions of Definition~\ref{lis} at the non-compact ends and requiring the resulting 1-form to be restricted to positive contact forms on the boundary of $[-C,C]\times M$. We denote the space of such objects $\mathcal{LIS}_c(M)$ and can similarly define $\mathcal{LIS}^{k*}_c(M)$ as above. The above lemma still holds in the compact setting. 

\begin{convention}
We assume LIS s to be non-compact unless stated otherwise. \end{convention}

We reformulate and generalize the characterization of the previous section in this generalized setting.

\begin{theorem}\label{generalchar}
The followings are equivalent:

(1) $X$ is non-singular and partially hyperbolic with the splitting $TM/\langle X \rangle \simeq E\oplus  E^u$;

(2) $X$ admits some supporting LIS $(\alpha_-,\alpha_+)_{(\lambda_-,\lambda_+)}\in \mathcal{LIS}(M)$;

(3) $X$ admits some supporting compact LIS $(\alpha_-,\alpha_+)_{(\lambda_-,\lambda_+)}\in \mathcal{LIS}_c(M)$
\end{theorem}

\begin{proof}
The construction of symmetric linear and exponential Liouville pairs as highlighted in Section~\ref{3.3} and the beginning of this section yields $C^k$ compact and non-compact LIS s supporting any non-singular partially hyperbolic flow with $C^k$ invariant weak bundles. Using approximations of Section~\ref{3.6} (more specifically, Lemma~\ref{approx2}), we can then achieve $C^\infty$ supporting linear and exponential Liouville pairs, since $X$ is assumed to be $C^\infty$. Hence, we have $(1)\Rightarrow (2)$ and $(1)\Rightarrow (3)$.

The argument for $(2)\Rightarrow (1)$ and $(3)\Rightarrow (1)$ is also parallel with the discussion in Section~\ref{3.3}, but with a bit more care. Let $(\alpha_-,\alpha_+)_{(\lambda_-,\lambda_+)}$ be a (compact or non-compact) LIS and $\Lambda_s:M\rightarrow \mathbb{R}$ the function defined by $\ker{[\lambda_-(\Lambda_s(x),x)\alpha_-+\lambda_+(\Lambda_s(x),x)\alpha_+]}=E$, where $E$ is the weak dominated bundle of a projectively Anosov vector field $X\subset \ker{\alpha_-}\cap \ker{\alpha_+}$ and inducing a dominated splitting $TM/\langle X \rangle \simeq E\oplus F$.

Noting that the vector field $X+ (X\cdot \Lambda_s)\partial_s$ is tangent to the graph $s=\Lambda_s$, the Liouville condition at such submanifold reads
$$0< \iota_X\iota_{\partial s} (d\alpha \wedge d\alpha)|_{s=\Lambda_s}= \iota_{[X+ (X\cdot \Lambda_s)\partial_s]}\iota_{\partial s} (d\alpha \wedge d\alpha)|_{s=\Lambda_s}$$
$$=-2(\mathcal{L}_{\partial_s}\alpha)\wedge(\mathcal{L}_{[X+ (X\cdot \Lambda_s)\partial_s]}\alpha)=2\alpha_s \wedge (\mathcal{L}_X \alpha_u),$$
where $$\alpha_u:=\lambda_-(\Lambda_s(x),x)\alpha_-+\lambda_+(\Lambda_s(x),x)\alpha_+$$ is the restriction of $\alpha$ to the graph $s=\Lambda_s$ (with $\ker{\alpha_u}=E$) and $\alpha_s$ is defined as follows. Considering Convention~\ref{orientation}, we can then write
$$\begin{cases}
\alpha_+=h_1\alpha_u -h_2\bar{\alpha}_s \\
\alpha_-=h_3\alpha_u+h_4\bar{\alpha}_s
\end{cases},$$
for some $\bar{\alpha}_s$ and positive functions $h_1,h_2,h_3,h_4:M\rightarrow \mathbb{R}_{>0}$,
yielding
$$\mathcal{L}_{\partial_s}\alpha=(\partial_s\cdot \lambda_-) \alpha_-+(\partial_s\cdot \lambda_+) \alpha_-=A\alpha_u+[(\partial_s\cdot \lambda_-) h_4-(\partial_s\cdot \lambda_+) h_2]\bar{\alpha}_s,$$
for some function $A$.
This justifies the definition of the 1-form
$$\alpha_s:=[-(\partial_s\cdot \lambda_-) h_4+(\partial_s\cdot \lambda_+) h_2]\bar{\alpha}_s,$$
which is non-vanishing at the graph $s=\Lambda_s$ and defines the same orientation as $\bar{\alpha}_s$, since we have $h_4 \lambda_- -h_2\lambda_+=0$ at such graph, yielding
$$-(\partial_s\cdot \lambda_-) h_4+(\partial_s\cdot \lambda_+) h_2|_{s=\Lambda_s}=h_2[-(\partial_s\cdot \lambda_-) \frac{\lambda_+}{\lambda_-}+(\partial_s\cdot \lambda_+)]|_{s=\Lambda_s}=h_2\lambda_+^2[\partial_s \cdot (\frac{\lambda_+}{\lambda_-})]|_{s=\Lambda_s}>0.$$

Therefore, we have 
$$0< \iota_X\iota_{\partial s} (d\alpha \wedge d\alpha)|_{s=\Lambda_s}=(2r_u)\alpha_s \wedge \alpha_u,$$
where $r_u$ is the expansion rate of $F$ with respect to the norm induced from $\alpha_u$, completing the proof.

\end{proof}

Note that using this,  a characterization of Anosov 3-flows  can be formulated as non-singular vector fields like $X$, such that $X$ and $-X$ are both partially hyperbolic.

We will later see that the function $\Lambda_s:M\rightarrow \mathbb{R}$ defined in the proof of the above theorem in fact plays an important role in the theory as it determines the underlying Liouville skeleton. The above proof then means that the Liouville condition at the skeleton is equivalent to the absolute expansion of dominating bundle and hence, the partial hyperbolicity of the supported projectively Anosov flow.

\subsection{Bi-contact DA deformation and non-Anosov examples of partially hyperbolic flows}\label{3.5}

The goal of this section is to describe a deformation of flow near a periodic orbit of an Anosov flow which deforms it to a non-Anosov partially hyperbolic flow. This deformation is called DA, i.e. {\em derived from Anosov}, and its idea goes back to Smale and was used by Franks-Williams in their construction of non-transitive Anosov flows \cite{da}. In \cite{et}, the authors descrribe a DA deformation by deforming the invariant foliationsm claiming that the such deformation can be done while keeping the dominated splitting. Our construction revisits their idea by showing that the resulting projectively Anosov flow can be achieved by a homotopy of the flow through projectively Anosov flows, or equivalently, can be thought of as a bi-contact homotopy. Our description mostly follows, but generalizes, the description of a general DA deformation in \cite{bby}, while keeping track of the underlying bi-contact deformation.

Recall that projectively Anosov flows are generically axiom A flows and therefore, structurally stable. Let $\gamma$ be a periodic orbit of a structurally stable projectively Anosov flow $X$ and for simplicity, assume the invariant bundles are orientable along such orbit. Let $e^{T\nu}$ and $e^{T\mu}$ be the eigenvalues of the return map, where $T$ is the period of $\gamma$ and $\mu>0>\nu$. In particular, we are assuming $\gamma$ is saddle periodic orbit. After a perturbation of $X$, we can assume its structural stability and the fact that it is $C^1$-linearizable at $\gamma$. Therefore, after a homotopy of the flow on a tubular neighborhood $V$ of $\gamma$,  we can find a coordinate system near $V\rightarrow [-1,1]_x\times [-1,1]_y\times \mathbb{R}_\theta/T\mathbb{Z}$, in which we can write
$$X(x,y,\theta)=\nu x \partial_x +\mu y \partial_y +\partial_\theta.$$

Notice that a supporting bi-contact structure for $X$ can be defined on $V$ by
$$\begin{cases}
\alpha_+:=dy-dx+(\nu x -\mu y)d\theta \\
\alpha_-:=dy+dx+(-\nu x - \mu y)d\theta
\end{cases},$$
noting that
$$\begin{cases}
\ker{\alpha_+}=\langle X, \partial_x +\partial_y \rangle \\
\ker{\alpha_-}=\langle X, -\partial_x +\partial_y \rangle 
\end{cases}$$
and
$$\begin{cases}
\alpha_+ \wedge d\alpha_+=(\mu -\nu) dx \wedge dy \wedge d\theta \\
\alpha_- \wedge d\alpha_-=(\nu-\mu) dx \wedge dy \wedge d\theta
\end{cases}.$$

Now, for any choice of real numbers $0<\eta<1$ and $\bar{\nu}<\mu$, we can define a deformation of $X$, which vanishes outside $V$ and is defined by
$$\begin{cases}
Y_\eta (x,y,\theta):=[(\bar{\nu} -\nu)\phi (\frac{x}{\eta}) \phi (\frac{y}{\eta}) ]x\partial_x \\
X_\eta:=X+Y_\eta
\end{cases},$$
where $\phi(t)=(1-t^2)^2 \mathbf{1}_{[-1,1]}$ is a non-negative real bump function. We show that for any choice of such $\eta$ and $\bar{\nu}$, the resulting flow is projectively Anosov and furthermore, notice that $\lim_{\eta\rightarrow 0} Y_\eta=0$. Therefore, the resulting flow is in fact homotopic, through projectively Anosov flows, to the original flow.

Now, we let $$\begin{cases}
\bar{\alpha}_+:= dy-dx+(\hat{\nu} x -\mu y)d\theta \\
\bar{\alpha}_-:=dy+dx+(-\hat{\nu} x -\mu y)d\theta
\end{cases},$$
where
$$\hat{\nu}=\nu+(\bar{\nu} -\nu)\phi (\frac{x}{\eta}) \phi (\frac{y}{\eta})$$
and note that $\bar{\alpha}_+(X_\eta)=\bar{\alpha}_-(X_\eta)=0$ and $\ker{\bar{\alpha}_+} \pitchfork \ker{\bar{\alpha}_-}$. This is enough for us to show that $\bar{\alpha}_+$ and $\bar{\alpha}_-$ are positive and negative contact structures on $[-\eta ,\eta]_x \times[-\eta,\eta ]_y\times \mathbb{R}_\theta/T\mathbb{Z} \subset [-1,1]_x\times [-1,1]_y\times \mathbb{R}_\theta/T\mathbb{Z}$, respectively, in order to establish $X_\eta$ being projectively Anosov (notice $\bar{\alpha}_{\pm}=\alpha_{\pm}$ outside $[-\eta ,\eta]_x \times[-\eta,\eta ]_y\times \mathbb{R}_\theta/T\mathbb{Z}$). 
We show this claim for $\bar{\alpha}_+$. The computations for $\bar{\alpha}_-$ are similar. We compute
$$\bar{\alpha}_+\wedge d \bar{\alpha}_+ = (\mu-\hat{\nu})dx \wedge dy \wedge d\theta +(dy-dx)\wedge (xd\hat{\nu})\wedge d\theta$$
$$=\bigg[(\mu-\hat{\nu}) -(\frac{\bar{\nu}-\nu}{\eta} )\phi'(\frac{x}{\eta})\phi(\frac{y}{\eta})x-(\frac{\bar{\nu}-\nu}{\eta} )\phi(\frac{x}{\eta})\phi'(\frac{y}{\eta})x \bigg]dx\wedge dy \wedge d\theta$$
$$=\bigg[(\mu-\nu) +(\nu-\bar{\nu})[\phi(\frac{x}{\eta})\phi(\frac{y}{\eta})+\frac{1}{\eta}\phi'(\frac{x}{\eta})\phi(\frac{y}{\eta})x+\frac{1}{\eta}\phi(\frac{x}{\eta})\phi'(\frac{y}{\eta})x] \bigg]dx\wedge dy \wedge d\theta$$
$$=\bigg[(\mu-\nu) +(\nu-\bar{\nu})[(1-\frac{x^2}{\eta^2})^2(1-\frac{y^2}{\eta^2})^2  -\frac{4x^2}{\eta^2} (1-\frac{x^2}{\eta^2})(1-\frac{y^2}{\eta^2})^2 - \frac{4xy}{\eta^2}(1-\frac{x^2}{\eta^2})^2(1-\frac{y^2}{\eta^2})] \bigg]dx\wedge dy \wedge d\theta$$
$$=\bigg[(\mu-\nu) +(\nu-\bar{\nu}) (1-\frac{x^2}{\eta^2})(1-\frac{y^2}{\eta^2})[ (1-\frac{5x^2}{\eta^2})(1-\frac{y^2}{\eta^2})- \frac{4xy}{\eta^2}(1-\frac{x^2}{\eta^2})] \bigg]dx\wedge dy \wedge d\theta$$
$$=\bigg[(\mu-\nu) +(\nu-\bar{\nu})A(\frac{x}{\eta},\frac{y}{\eta}) \bigg] dx\wedge dy \wedge d\theta,$$
where
$$\begin{cases}
A:[-1,1]_x\times[-1,1]_y\rightarrow \mathbb{R} \\
A:=(1-x^2)(1-y^2)\big[ (1-5x^2)(1-y^2)- 4xy(1-x^2)\big] 
\end{cases}.$$

Using direct computation one can show the following claim. See Figure~5.

\begin{claim}
The function $A$ has unique local maxima at $(0,0)$ where $A=1$.
\end{claim}

\begin{figure}
\centering
\begin{tikzpicture}[scale=1.5]
\begin{axis}[
  colormap/cool, 
]
\addplot3[
    mesh,
  samples=50,
    domain=-1:1,
]
{(1-x^2)*(1-y^2)*((1-5*x^2)*(1-y^2)-4*(x*y)*(1-x^2))};

\end{axis}

\end{tikzpicture}
\caption{Graph of the map $A:[-1,1]_x\times[-1,1]_y\rightarrow \mathbb{R}$ }
\end{figure}
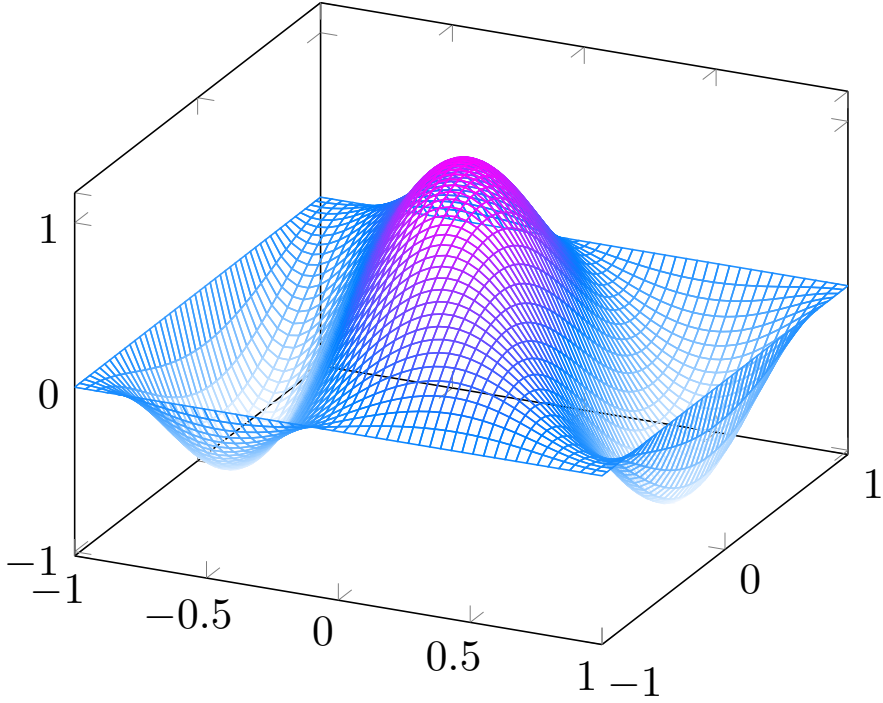


This in fact finishes the proof of the fact that as long as $\mu>\nu$ and $\mu >\bar{\nu}$, the (localized) DA deformation can be done while the preserving the generating vector field in the intersection of a bi-contact structure and hence, preserving projectively Anosovity of the flow via the bi-contact condition (see Lemma~\ref{pabicontact}).

Moreover, we can show that if the original flow is a non-singular partially hyperbolic flow, the resulting flow stays partially hyperbolic. This can be checked in fact using the Liouville geometric description of Theorem~\ref{anosovchar}. More precisely, start with the Anosov flow as above (it is sufficient to assume non-singular partial hyperbolicity) and consider the supporting linear Liouville pair $(\alpha_-,\alpha_+)_l$ which is defined (and is symmetric) near the periodic orbit $\gamma$. To have the claim, we need to show that for any $0<\eta<1$ the 1-form $$\bar{\alpha}:=(1-s)\bar{\alpha}_-+(1+s)\bar{\alpha}_+$$
is a Liouville form on $[-1,1]_s\times M$, which is equivalent to $X_\eta$ being partially hyperbolic. Similar computations as above, we can check the Liouville condition explicitly.
$$d\bar{\alpha}\wedge d\bar{\alpha}= [(1-s)d(-\hat{\nu} x-\mu y)\wedge d\theta+(1+s)d(\hat{\nu} x-\mu y)\wedge d\theta]\wedge ds \wedge [-2dx+2\hat{\nu} x d\theta],$$
where as discussed in Section~\ref{3.4}, we only need to check the Liouville condition at $s=0$. We have 
$$d\bar{\alpha}\wedge d\bar{\alpha}|_{s=0}=[-2\mu dy\wedge d\theta]\wedge ds \wedge [-2dx+2\hat{\nu} x d\theta]$$
$$=4\mu ds\wedge dx \wedge dy \wedge d\theta,$$
i.e. we did not ruin the absolute expansion of $E^u$ in the above deformation. We have proved the following.

\begin{theorem}\label{nonanosov}
The DA deformation of a projectively Anosov (non-singular partially hyperbolic) flow near a saddle periodic orbit can be done through projectively Anosov (non-singular partially hyperbolic) flows. In particular, the DA deformation of an Anosov flow near a periodic orbit results in a non-Anosov non-singular partially hyperbolic flow. 
\end{theorem}

\subsection{Approximation lemmas}\label{3.6}

In this section, we briefly discuss the approximation ideas which was used in the results in Section~\ref{3.3}, as well the throughout this paper. As we have observed so far, it is often useful to consider LIS s with low regularity (see Remark~\ref{lowliouville}). The following lemmas are useful to approximate LIS s of low regularity with ones of higher regularity. 

The key idea of these approximations boils down to Lemma~4.3 in \cite{hoz3}.

\begin{lemma}\label{approx1}
Let $X^t$ be the flow generated by a non-singular $C^{k+1}$ vector field on a manifold of arbitrary dimension $M$ ($k\geq 0$). Assume $f:M\rightarrow \mathbb{R}$ is a $C^k$ function on $M$ such that $\mathcal{L}_X f$ is also $C^k$. Then, for any $\epsilon>0$, there exists a $C^{k+1}$ function $\tilde{f}:M\rightarrow \mathbb{R}$ such that
$$||f-\tilde{f}||_{C^k}<\epsilon\ \ \ \ \ \ \text{and} \ \ \ \ \ \ ||\mathcal{L}_X f-\mathcal{L}_X\tilde{f}||_{C^k}<\epsilon,$$
where $||.||_{C^k}$ is the $C^k$-norm.
\end{lemma}

The above lemma is proved in \cite{hoz3} in the case $k=0$ and using a careful partition of unity on flow boxes. But the exact same proof yields the above, if one replace the $C^0$-norm with the $C^k$-norm for arbitrary $k\geq 0$. One can expect to generalize the same approximation lemma, possibly with more care on the regularity degrees, for sections of any Banach bundle over a manifold $M$ and an action induced from a flow on $M$. However, for the purposes of this paper, we only require such approximations on 1-forms containing the flow direction, in order to provide approximations of the weak invariant plane fields and the related geometric objects we construct from them. This has been carried out explicitly in \cite{hoz3} and more abstractly and generally in \cite{massoni}. Here we quote Lemma~A.2 of \cite{massoni} which builds upon the above lemma.

\begin{lemma}\label{approx2}
Let $k\geq 0$ and $X$ be a non-singular $C^{k+1}$ vector field on a manifold of arbitrary dimension $M$. Then, for any $0\leq l \leq k$, the space of $C^{k+1}$ 1-forms annihilating $X$, denoted by $Annih^{1;k+1}(M;X)$, is dense in $Annih^{1;l*}(M;X)$, where $Annih^{1;l*}(M;X)$ is the space of $C^l$ 1-forms annihilating $X$ whose Lie derivative along the flow is also $C^l$.
 \end{lemma}

   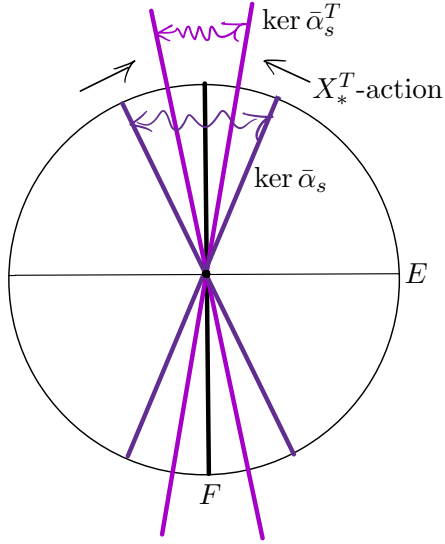
\begin{figure}[h]
\centering
\begin{overpic}[width=0.5\textwidth]{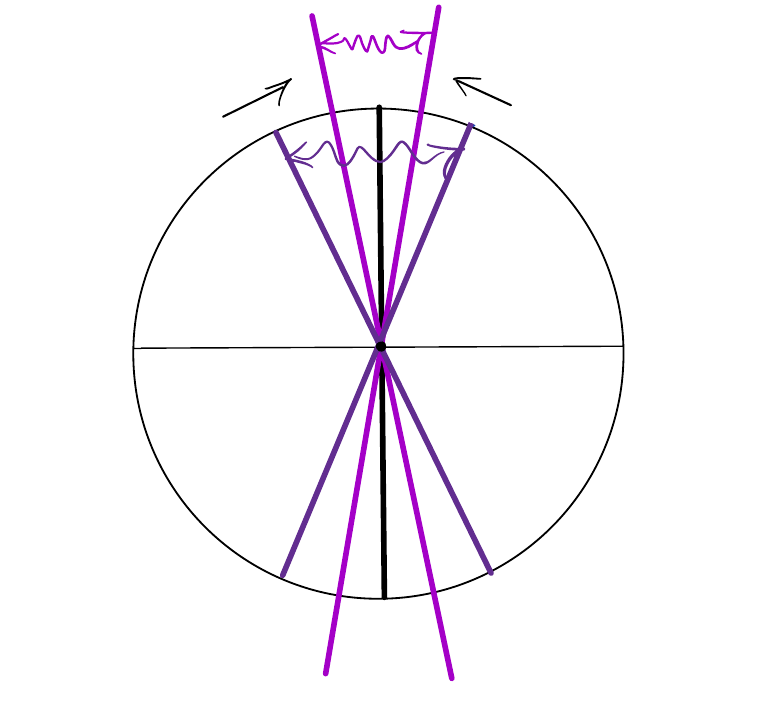}
    \put(112,20){$F$}
    \put(133,140){$\ker{\bar{\alpha}_s}$}
        \put(136,200){$\ker{\bar{\alpha}^T_s}$}
          \put(155,174){$X^T_*$-action}
        \put(190,104){$E$}

  \end{overpic}

\caption{Using the domination to approximate the weak invariant bundles}
\end{figure}
 
 \begin{remark}
 It is insightful to note that in our context, i.e. when $X$ is a $C^{k+1}$ (projectively) Anosov flow on a 3-manifold, explicit approximations of 1-forms annihilating $X$ is possible as shown in \cite{hoz3}. The rough idea is as follows. Since, the 1-forms $\alpha_s$ and $\alpha_u$ annihilating the weak invariant bundles, which we assume to be $C^l$ for some $l>0$, form a basis for $Annih^{1;l*}(M;X)$, it is enough to give an appropriate approximation of such 1-forms with ones which are as the regular as $X$. This is possible as follow. We start with a {\em bad} approximation of $\alpha_s$ with a 1-form $\bar{\alpha}_s$ which is as regular as $X$ and annihilates it (this is possible by performing the approximation in $TM/\langle X \rangle$). Then, one can replace such 1-from with $\bar{\alpha}_s^T:=\frac{1}{A_T}X^{T*}\bar{\alpha}_s$ for some large $T>0$ and an appropriate renormalization function $A_T$. This means that applying the action of the flow, not only pushes $\bar{\alpha}_s$ $C^l$-closer to $\alpha_s$, but also {\em damps down} its $C^{k+1}$ variations along the flow. See Figure~6.
  \end{remark}

\section{Dynamical rigidity of Liouville interpolations}\label{4}

The main purpose of this section is to take a closer look at the Liouville geometry and dynamics of non-Weinstein Liouville manifolds (or domains) one can construct given an Anosov 3-flow, as discussed in Section~\ref{3}. We will carry explicit computations in the generalized platform of Liouville interpolation systems we developed previously. The bottom line is that Liouville dynamics and geometry is strongly determined by the underlying flow, and the interpolation of contact structures through the weak invariant bundles, regardless of all the other choices. In terms of dynamical rigidity, this is encapsulated in the fact that in these examples, the Liouville dynamics can be understood in terms of a normally repellent 3-dimensional skeleton diffeomorphic to the manifold ambient to the flow. In view of the Moser technique, the dynamical rigidity can translated into geometric rigidity, manifested strict Liouville embeddings os such objects into one another. We will discuss the geometric rigidity in Section~\ref{5}.

We begin this section with some elementary observations.


\subsection{Normal foliations and elementary normal isotopies: the symmetries of Liouville interpolations}\label{4.1}

We begin this section with some elementary observations about the normal foliations and the space of Liouville interpolation systems associated with a flow. 

\subsubsection{Weak normal foliations as invariant strictly exact Lagrangian foliations}\label{4.1.1}

As we will see more clearly throughout this section, given a LIS $(\alpha_-,\alpha_+)_{(\lambda_-,\lambda_+)}$ on $W:=\mathbb{R}_s\times M$, the 2-dimensional foliations tangent to the plane field $\langle\partial_s , X\rangle$, where $X$ is vector field supported by $(\alpha_-,\alpha_+)_{(\lambda_-,\lambda_+)}$, plays a special role in the theory. 

\begin{definition}
Using the above notation, we denote the plane field $\langle\partial_s , X\rangle$ by $E^{wn}$ and call it the {\em weak normal bundle}. The corresponding foliation $\mathcal{F}^{wn}$ is called the weak normal foliations.
\end{definition}

The terminology used in this definition is thanks to the fact that as we will see in Section~\ref{4.2}, such foliation contains and is invariant under the corresponding Liouville vector field and furthermore, is laminated by a 1-dimensional foliation, which we will call the {\em strong normal bundle} and denote by $E^n$, transverse to and invariant under the Liouville flow.

First and foremost, we observe that such foliation is in fact a foliation of $(W,\alpha)$ by (strict) exact Lagrangians. To see this, just note that
$$L(\alpha_-,\alpha_+)_{(\lambda_-,\lambda_+)}\big|_{E^{wn}}=0.$$

In Section~\ref{7}, we will see that the existence of these foliations is not a coincidence by construction, but in fact the result of Liouville dynamics near the skeleton.

Notice that the leaf space of such foliation is the same as the orbit space of $X$. In particular, a periodic orbit $\gamma$ of $X$ corresponds to an exact Lagrangian (non-compact) annulus $\mathbb{R}_s\times \gamma$, while non-periodic orbits are represented by leaves of this foliation diffeomorphic to a plane. We notice that by construction, the tangent field of this foliation is always as regular as the generating vector field $X$, even when dealing with LISs with lower regularity. Moreover, this foliation is preserved under Liouville deformations as long as we fix the underlying flow $X$, a fact which is crucial when establishing geometric rigidity in Section~\ref{5}.

 An elementary observation we prove here is the fact this foliation is invariant under Liouville dynamics. We will later explicitly compute the Liouville vector field in order to reach a complete understanding of Liouville dynamics.

\begin{lemma}\label{ypresl}
Let $Y$ be the Liouville vector field for $\alpha:=L(\alpha_-,\alpha_+)_{(\lambda_-,\lambda_+)}$ for some $(\alpha_-,\alpha_+)_{(\lambda_-,\lambda_+)} \in \mathcal{LIS}(M)$ and assume $\alpha$ is $C^1$. We have $Y\subset \langle X,\partial_s\rangle$.
\end{lemma}

\begin{proof}
First notice $\alpha(Y)=d\alpha(Y,Y)=0$. Since $\alpha$ and $\mathcal{L}_{\partial_s} \alpha$ are linearly independent, it is enough to show $(\mathcal{L}_{\partial_s} \alpha)(Y)=0$. Compute
$$\iota_Yd\alpha=\alpha \Rightarrow \iota_{[\partial_s,Y]}d\alpha+\iota_Y(d\mathcal{L}_{\partial_s}\alpha)=\mathcal{L}_{\partial_s}\alpha$$
$$\Rightarrow (\mathcal{L}_{\partial_s}\alpha)(Y)=d\alpha([\partial_s,Y],Y)=-\alpha([\partial_s,Y])=d\alpha(\partial_s,Y)+\partial_s\cdot \alpha(Y)+Y\cdot \alpha(\partial_s)=0.$$

The last equality follows, since each term equals $0$.
\end{proof}

This means that the Liouville vector field can be written as $Y=fX+g\partial_s$, where a simple computation gives
$$\alpha=\iota_Y d\alpha=f\mathcal{L}_X\alpha +g \mathcal{L}_{\partial_s}\alpha.$$
In other words, the Liouville condition implies the $\mathcal{L}_X\alpha$ and $\mathcal{L}_{\partial_s}\alpha$ form a basis for $Annih^1(\mathbb{R}\times M ;E^{wn})$. Therefore, the above means that the coordinate functions of writing $\alpha$ in such basis is the same as the coordinate functions of $Y$ in the basis $(X,\partial_s)$.

\begin{remark}\label{lowvf}
The above computation also lets us make sense of the Liouville vector field in the low regularity setting of Remark~\ref{lowliouville}. More precisely, for any $\alpha=(\alpha_-,\alpha_+)_{(\lambda_-,\lambda_+)} \in \mathcal{LIS}^{k*}(M)$, the $C^k$ 1-forms $\mathcal{L}_X \alpha$ and $\mathcal{L}_{\partial_s}\alpha$ form a basis for $Annih^1(\mathbb{R}\times M ;E^{wn})$ and we can write $\alpha=f\mathcal{L}_X\alpha +g \mathcal{L}_{\partial_s}\alpha$ for some functions $f$ and $g$. We call the vector field $Y:=fX+g\partial_s$ the Liouville vector field for $\alpha$ and we have
$$\mathcal{L}_Y \alpha=\alpha.$$
\end{remark}

We have shown the following.

\begin{corollary}\label{lioureg}
For any $k>0$, the map 
$$\begin{cases}
F:\mathcal{LIS}^{k*}(M)\rightarrow \mathcal{\chi}^{k}(\mathbb{R}\times M) \\
F(\alpha_-,\alpha_+)_{(\lambda_-,\lambda_+)}=Y
\end{cases},$$
sending a $C^{k*}$ LIS to its Liouville vector field is continuous, where $\mathcal{\chi}^{k}(\mathbb{R}\times M)$ is the space of $C^k$ vector fields on $\mathbb{R}\times M$.
\end{corollary}

In order to understand the space of LIS s associated with a fixed Anosov flow, we next need to understand the space of Liouville deformations fixing the weak normal foliation. This can achieved using some maps (isotopies) we introduce next, which we call {\em elementary normal maps (isotopies)}, as they preserve the weak normal foliation leaf-wise: Change of basis, horizontal and scaling maps (isotopies). We will see in Section~\ref{6} that their explicit description makes these isotopies helpful in the low regularity setting where the Moser technique cannot be applied as usual.

\subsubsection{Change of basis maps}\label{4.1.2}

We can naturally change the representation of a Liouville interpolation by a change of bi-contact forms. That is, for any $z_-,z_+:M\rightarrow \mathbb{R}$, we can write

$$L(\alpha_-,\alpha_+)_{(\lambda_-,\lambda_+)}=\lambda_-\alpha_- + \lambda_+ \alpha_+=(e^{-z_-}\lambda_-)(e^{z_-}\alpha_- )+( e^{-z_+}\lambda_+ )(e^{z_+}\alpha_+)$$
$$=L(e^{z_-}\alpha_-,e^{z_+}\alpha_+)_{(e^{-z_-}\lambda_-,e^{-z_+}\lambda_+)}.$$

This means that assuming $z_-$ and $z_+$ are $C^{k*}$, they induce a $C^{k*}$ action on $\mathcal{LIS}^{k*}(X)$ and we have the following.

\begin{lemma}
For $C^{k*}$ functions $z_-,z_+:M\rightarrow \mathbb{R}$, we have the {\em change of basis map} defined as
$$\begin{cases}
I_{(z_-,z_+)}: \mathcal{LIS}^{k*}(X) \rightarrow \mathcal{LIS}^{k*}(X)\\
I_{(z_-,z_+)} (\alpha_-,\alpha_+)_{(\lambda_-,\lambda_+)}=(e^{z_-}\alpha_-,e^{z_+}\alpha_+)_{(e^{-z_-}\lambda_-,e^{-z_+}\lambda_+)}
\end{cases}$$
and we have
$$LI_{(z_-,z_+)}=L.$$
\end{lemma}

 Note that the functions $e^{z_-}\alpha_-$ and $e^{z_+}\alpha_+$ satisfy all the conditions in Definition~\ref{lis}, since $$\ln{\frac{e^{-z_+}\lambda_+}{e^{-z_-}\lambda_-}}=\ln{\frac{\lambda_+}{\lambda_-}}+z_--z_+.$$

Using such maps we can write any Liouville interpolation in terms of an arbitrary bi-contact form $(\alpha_-,\alpha_+)$ and the map in the above lemma isotopic, through change of basis maps, to $Id$ via $I_{(tz_-,tz_+)}$ for $t\in [0,1]$.

\subsubsection{Horizontal maps}\label{4.1.3}

Consider a LIS $(\alpha_-,\alpha_+)_{(\lambda_-,\lambda_+)}$ defined on $\mathbb{R}\times M$. We call maps of the form 
$$\begin{cases}
H_\psi:\mathbb{R}\times M\rightarrow \mathbb{R}\times M
\\
H_\psi(s,x)=(\psi(s,x),x)
\\
\partial_s\cdot \psi>0
\end{cases}$$
{\em horizontal}, where $\psi(.,x):\mathbb{R}_s\times M \rightarrow \mathbb{R}$ is an oriented reparametrization of $\mathbb{R}_s$ for any $x\in M$. In other words, a horizontal map shifts points of $\mathbb{R}_s\times M$ in the direction of $\partial_s$.

Note that $$dH_{\psi }=\begin{bmatrix}

\partial_s \cdot \psi \ \ \ \ \ B \\
0 \ \ \ \ \ \ \ Id_{M}
\end{bmatrix}.$$
and we have $H_\psi^*\alpha_\pm =\alpha_\pm$, since $\alpha_\pm(\partial_s)=0$. Therefore, for $\alpha=\lambda_-\alpha_-+\lambda_+\alpha_+$ we have $$H_\psi^*\alpha=( \lambda_-\circ H_\psi^{-1})\alpha_- +(\lambda_+\circ H_\psi^{-1})\alpha_+ ,$$
where $H_\psi^{-1}(s,x):=(\psi^{-1}(x,s),x)$ with $\psi^{-1}(.,x)$ is the inverse of $\psi(.,x)$ for any $x\in M$. In particular, the result of applying such map is another LIS $(\alpha_-,\alpha_+)_{( \lambda_-\circ {H_\psi}^{-1}, \lambda_+\circ (H_\psi)^{-1})}$ and therefore, assuming $\psi$ is $C^{k*}$, it induces an action on $\mathcal{LIS}^{k*}(X)$ which we denote by $H_\psi^*$.

\begin{lemma}
The horizontal map
$$\begin{cases}
H_\psi^*:\mathcal{LIS}^{k*}(M)\rightarrow \mathcal{LIS}^{k*}(M) \\
H_\psi^*(\alpha_-,\alpha_+)_{( \lambda_-, \lambda_+)}=(\alpha_-,\alpha_+)_{( \lambda_-\circ H_\psi^{-1}, \lambda_+\circ H_\psi^{-1})}
\end{cases}$$
is continuous, if $\psi$ is $C^{k*}$.
\end{lemma}

Notice that the map in the above lemma is isotopic, through horizontal maps, to $Id$ via $H_{t\psi+(1-t)s}$ for $t\in [0,1]$.

Here, we would like to point out using the two introduced maps above, i.e. a change of basis map and a horizontal map, we can write transform any LIS to a {\em distorted} exponential one. 

More precisely, consider an LIS $(\alpha_-,\alpha_+)_{(\lambda_-,\lambda_+)}$. By Theorem~4.13 of \cite{massoni}, there exists an exponential Liouville pair $(\bar{\alpha}_-,\bar{\alpha}_+)_e$ such that $\ker{\bar{\alpha}_-}=\ker{\alpha_-}$ and $\ker{\bar{\alpha}_+}=\ker{\alpha_+}$. Therefore, there is a change of basis map such that $I_{(f,g)}(\alpha_-,\alpha_+)_{(\lambda_-,\lambda_+)}=(\bar{\alpha}_-,\bar{\alpha}_+)_{(\bar{\lambda}_-,\bar{\lambda}_+)}$.

Then, use the horizontal map $H_\psi$, where $\psi(s,x)= \frac{1}{2}\ln{(\frac{\bar{\lambda}_+}{\bar{\lambda}_-})}$. Notice that we are using the fact that $s\mapsto \ln{(\frac{\bar{\lambda}_+}{\bar{\lambda}_-})}$ is an oriented reparametrization of $\mathbb{R}_s$ for any $x\in M$, as required in Definition~\ref{lis}. As a result, we get an LIS of the form $(\bar{\alpha}_-,\bar{\alpha}_+)_{(e^{-s+w},e^{s+w})}$, where $w=\frac{1}{2}\ln{(\bar{\lambda}_-\bar{\lambda}_+)}:\mathbb{R}\times M \rightarrow \mathbb{R}$. We have
$$L(\bar{\alpha}_-,\bar{\alpha}_+)_{(e^{-s+w},e^{s+w})}=e^{-s+w(s)}\alpha_- + e^{s+w(s)}\alpha_+=e^{w(s)}[e^{-s}\alpha_- + e^{s}\alpha_+].$$

Now, that our Liouville form is only a scaling away from one induced from an exponential Liouville pair.

\subsubsection{Scaling maps}\label{4.1.4}

Another useful operation for us is scaling a Liouville form without changing its kernel. The following general lemma provides the basic properties of such maps.

\begin{lemma}\label{genscale}
Let $(W,\alpha_0)$ be a Liouville manifold with the associated Liouville vector field $Y_0$. Also, assume $\alpha_1:=e^f \alpha_0$ is a Liouville form on $W$ (i.e. only assuming $d\alpha_1\wedge d\alpha_1>0$). Then, for any $0\leq t \leq 1$

(a) ${Y_0}\cdot f>-1$ and the Liouville vector field of $\alpha_t:=e^{tf}\alpha_0$ is $Y_t:=\frac{Y_0}{1+t{Y_0}.f}$ (in particular, $Y_t$ is conjugate to $Y_0$);

(b) $(W,\alpha_t)$ is a Liouville manifold;

(c) the map $\psi_t:W\rightarrow W$ defined by $\psi_t(x):=\phi_{tf}(x)$ is an isotopy of Liouville manifolds, i.e. $\psi_t^* \alpha_t=\alpha_0$.
\end{lemma}

\begin{proof}
Compute $$d\alpha_t=e^{tf} d\alpha_0 +te^{tf}df\wedge \alpha_0$$
$$\Rightarrow e^{-2tf}d\alpha_t \wedge d\alpha_t = d\alpha_0\wedge d\alpha_0 +2t df\wedge \alpha_0\wedge d\alpha_0.$$
Therefore, $\alpha_0$ and $\alpha_1$ being Liouville implies the same for $\alpha_t$.

Now, we have $$\iota_{Y_0} d\alpha_t =e^{tf} \iota_{Y_0} d\alpha_0+(te^{tf}{Y_0}\cdot f) \alpha_0$$
$$=(1+t{Y_0}\cdot f)e^{tf}\alpha_0=(1+t{Y_0}\cdot f)\alpha_t.$$

Since $d\alpha_t$ is non-degenerate, we have $1+t{Y_0}\cdot f>0$ for $0\leq t \leq 1$ or equivalently, ${Y_0}\cdot f>-1$ (notice that this is automatically satisfied at points where $\alpha_0$, and consequently $Y_0$ vanish). Furthermore, letting $Y_t:=\frac{Y_0}{1+t{Y_0}\cdot f}$ yields $\iota_{Y_t}d\alpha_t=\alpha_t$ as desired. This moreover implies that $Y_t$ and $Y_0$ are conjugate via the map $\psi_t(x):=\phi_{tf}(x)$. Therefore, if $Y_0$ is complete, so is $Y_t$ for all $0\leq t \leq 1$. Hence, $(W,\alpha_t)$ is a Liouville manifold for all $0\leq t \leq 1$. Furthermore, $\psi_t^* \alpha_t=\alpha_0$ provides an isotopy of Liouville manifolds.
\end{proof}

In the context of Liouville interpolation systems, suppose we have a Liouville manifold of the form $(\mathbb{R}\times M,\alpha)$, this justifies defining a {\em scaling function} associated to the function $f:\mathbb{R}\times M \rightarrow \mathbb{R}\times M$ to be
$$\begin{cases}
S_f: \mathbb{R}\times M\rightarrow  \mathbb{R}\times M \\
S_f(s,x)=Y^{f}(s,x)
\end{cases},$$
where $Y^t$ is the flow generated by the Liouville vector field, which implies $S_f^*\alpha=e^f\alpha$. Note that by part (a) of the previous lemma, not all functions $f$ are admissible in this construction, but they form a contractible open subset of all functions on $\mathbb{R}\times M$. Now, if $\alpha=L(\alpha_-,\alpha_+)_{( \lambda_-, \lambda_+)}$ and $f$ are $C^{k*}$, this induces an action on $\mathcal{LIS}^{k*}(X)$ which we denote by $S_f^*$.

\begin{lemma}
The scaling map
$$\begin{cases}
S_f^*:\mathcal{LIS}^{k*}(M)\rightarrow \mathcal{LIS}^{k*}(M) \\
S_f^*(\alpha_-,\alpha_+)_{( \lambda_-, \lambda_+)}=(\alpha_-,\alpha_+)_{(e^f \lambda_-,e^f  \lambda_+)}
\end{cases}$$
is continuous, if $f$ is $C^{k*}$ and is admissible in the sense of the Lemma~\ref{genscale}.
\end{lemma}

Similar to the change of basis and horizontal maps, any scaling map is isotopic, through scaling maps, to $Id$ via $S_{tf}$ for $t\in [0,1]$.

\subsection{Liouville flows, skeleton and dynamical rigidity}\label{4.2}

The goal of this section is compute Liouville dynamics in details and show rigidity beyond all choices.

Lemma~\ref{ypresl} implies that we can write $Y=fX+g\partial_s$, where $f,g:\mathbb{R}_s\times M\rightarrow \mathbb{R}$. We have
$$\iota_Y d\alpha=f\mathcal{L}_X \alpha +g \mathcal{L}_{\partial_s} \alpha=\alpha.$$
It is easy to see that $f$ is nowhere vanishing, since $\alpha=\lambda_-\alpha_-+\lambda_+ \alpha_+$ and $\mathcal{L}_{\partial_s}\alpha=(\partial_s\cdot\lambda_-)\alpha_-+(\partial_s\cdot\lambda_+) \alpha_+$ are always linearly independent. In fact we will see that $f>0$, if we assume $(\alpha_-,\alpha_+)_{(\lambda_-,\lambda_+)}$ supports $X$ (respecting the orientation convention in Definition~\ref{support}).


For the sake of simplicity, we would like to do the computations of Liouville vector field in the exponential setting. To do so, we need to show that any LIS can be isotoped to an exponential one using the maps introduced in Section~\ref{4.1}, i.e. change of basis, horizontal and scaling maps. First, we record our observations of the previous section about how applying such maps affects the underlying Liouville dynamics in the following lemma. Note that non of these maps affects the supported  (positive reparametrization class of) non-singular partially hyperbolic vector field(s).

\begin{lemma}\label{lioumaps}
The map 
$$\begin{cases}
F:\mathcal{LIS}_X^{k*}\rightarrow \mathcal{\chi}^{k}(\mathbb{R}\times M) \\
F(\alpha_-,\alpha_+)_{(\lambda_-,\lambda_+)}:=Y
\end{cases},$$
where $Y$ is the Liouville vector field of $L(\alpha_-,\alpha_+)_{(\lambda_-,\lambda_+)}$, is continuous and for any change of basis map $I_{(z_-,z_+)}$, horizontal map $H_\psi$ or scaling map $S_f$, we have 
$$FI_{(z_-,z_+)}=F,$$
$$FH_\psi^*=dH_{\psi } F$$
and
$$FS_f^*=S_{f*}F=(1+ [F,f])^{-1}F.$$
In particular, applying such maps does not affect the conjugacy class of the underlying Liouville vector field.
Furthermore, we have
$$P(\alpha_-,\alpha_+)_{(\lambda_-,\lambda_+)}=PI_{(z_-,z_+)}(\alpha_-,\alpha_+)_{(\lambda_-,\lambda_+)}=PH_\psi^*(\alpha_-,\alpha_+)_{(\lambda_-,\lambda_+)}=PS_f^*(\alpha_-,\alpha_+)_{(\lambda_-,\lambda_+)},$$
i.e. applying an elementary map (isotopy) does not affect the supported (non-singular partially hyperbolic) flow.
\end{lemma}

Now, the next lemma shows that the Liouville form induced from any LIS is strictly isotopic to one induced from an exponential one and in particular, yields a Liouville manifold as claimed in Section~\ref{3.4}, i.e. its induced Liouville vector field is complete. Part of the argument is already carried out in Section~\ref{4.1.3}.

\begin{lemma}\label{liscomp}
Any LIS $(\alpha_-,\alpha_+)_{(\lambda_-,\lambda_+)}$ is isotopic to an exponential LIS and in particular, defines a Liouville manifold.
\end{lemma}

\begin{proof}
This can be done in three steps and using the maps introduced in Section~\ref{4.1}. Given any LIS $(\alpha_-,\alpha_+)_{(\lambda_-,\lambda_+)}$, as described in Section~\ref{4.1.3}, after an appropriate change of basis and a horizontal map, we have 
$$H^*_{\psi}I_{(f,g)}(\alpha_-,\alpha_+)_{(\lambda_-,\lambda_+)}=(\bar{\alpha}_-,\bar{\alpha}_+)_{(e^{-s+w},e^{s+w})},$$
where $(\bar{\alpha}_-,\bar{\alpha}_+)_e$ is also an exponential Liouville pair.

Now, by Lemma~\ref{genscale}, $(\bar{\alpha}_-,\bar{\alpha}_+)_{(e^{-s+w},e^{s+w})}$ and $(\bar{\alpha}_-,\bar{\alpha}_+)_e$ can be isotoped using the scaling map $S_w$. Therefore, we have achieved $S_wH_\psi I_{(f,g)}(\alpha_-,\alpha_+)_{(\lambda_-,\lambda_+)}=(\bar{\alpha}_-,\bar{\alpha}_+)_e$, where each map induces an isotopy of LISs as claimed. 

The fact that Liouville forms induced from the exponential model is implicitly shown and used throughout \cite{massoni,clmm}, and in fact, boils down to the fact that we have exponential expansion of a contact from as we approach any of the ends. We will confirm this fact, when we do our explicit computations of Liouville vector field in Section~\ref{4.2}. In short, in the exponential case we have
$$
\begin{cases}
\alpha=e^{-s}\alpha_- +e^s\alpha_+ \\
 d\alpha=e^{-s}d\alpha_- +e^s d\alpha_++ds\wedge (-e^{-s}\alpha_- +e^s\alpha_+) \\
 \iota_Y d\alpha=\alpha
 \end{cases}$$
 $$\text{as}\ \ s\rightarrow +\infty \Longrightarrow
 \begin{cases}
 \alpha \approx e^s \alpha_+ \\
 d\alpha \approx e^s d\alpha_+ +e^s ds\wedge \alpha_+=d(e^s\alpha_+)
 \end{cases} \Longrightarrow Y\approx \partial_s, $$
 which in particular implies the completeness of $Y$. 
\end{proof}

We want to compute the Liouville vector field for $C^{k*}$ exponential Liouville pair $(\alpha_-,\alpha_+)_e$ (note that by Remark~\ref{lowvf}, this still make sense if $0<k<1$), since after $C^{k*}$ operations of change of basis, horizontal shifting and scaling, we can reduce the computations for a general LIS to the exponential case. So, consider $(\alpha_-,\alpha_+)_e$ and write

$$\begin{cases}
\alpha_+=\alpha_u-\alpha_s \\
\alpha_-=h_u\alpha_u+h_s\alpha_s
\end{cases}$$
for foliation 1-forms $\alpha_u$ and $\alpha_s$ and positive functions $h_u,h_s:M\rightarrow \mathbb{R}_+$, all of which are differentiable along $X$. We furthermore let $r_u$ and $r_s$ be the expansion rates associated with $\alpha_u$ and $\alpha_s$, respectivel,y and note that the positive and negative contactness of $\alpha_+$ and $\alpha_-$, respectively, implies
$$\begin{cases}
r_u-r_s>0 \\
r_u-r_s+X\cdot \ln{(\frac{h_u}{h_s})}>0
\end{cases}.$$

We write the associated Liouville vector field $Y=fX+g\partial_s$  and $\alpha=E\alpha_u+F\alpha_s$, where
$$\begin{cases}
E=e^s+h_u e^{-s} \\
F=-e^{s}+h_s e^{-s}
\end{cases}$$
and note 
$$\begin{cases}
\mathcal{L}_X \alpha= (X\cdot E +r_u E)\alpha_u+(X\cdot F+r_s F)\alpha_s \\
\mathcal{L}_{\partial_s} \alpha= (e^{s}-h_u e^{-s})\alpha_u+(-e^{s} -h_s e^{-s})\alpha_s
\end{cases}.$$

This implies that
$$\alpha=\iota_Y d\alpha=f\mathcal{L}_X\alpha+g\mathcal{L}_{\partial_s}\alpha=[Af+Bg]\alpha_u+[Cf+Dg]\alpha_s,$$
where
$$\begin{cases}
A=X\cdot E+r_u E=(r_u)e^s+(X\cdot h_u+r_uh_u)e^{-s} \\
B=e^{s} -h_u e^{-s} \\
C=X\cdot F+r_s F=(-r_s)e^s+(X\cdot h_s+r_sh_s )e^{-s} \\
D= -e^{s} -h_s e^{-s}
\end{cases}.$$

In other words, the above matrix is the change of coordinate matrix between $(\alpha_s,\alpha_u)$ and $(\mathcal{L}_X \alpha,\mathcal{L}_{\partial_s}\alpha)$, which $C^k$ when $\alpha$ is $C^{k*}$.

Considering the system of equations
$$\begin{cases}
Af+Bg=E \\
Cf+Dg=F
\end{cases} \Longleftrightarrow \begin{bmatrix} f \\ g \end{bmatrix}=\frac{1}{AD-BC}\begin{bmatrix} D \ -B \\-C \ A\end{bmatrix} \begin{bmatrix} E\\ F \end{bmatrix},$$ we compute 
$$\frac{g}{f}=\frac{-CE+AF}{DE-BF}=\frac{-(X\cdot F+r_s F)E+(X\cdot E+r_u E)F}{(-e^{s} -h_s e^{-s})E-(e^{s} -h_u e^{-s})F}$$
$$=\frac{EF(r_u-r_s)+FX\cdot E-EX\cdot F}{-2(h_u+h_s)}=\frac{-2F}{{E(h_u+h_s)}} [(r_u-r_s)+X\cdot {(\frac{F}{E})}] .$$

Note that 
$$FX\cdot E-EX\cdot F=(-e^{s}+h_s e^{-s})(X\cdot h_u)e^{-s}-(e^s+h_u e^{-s})(X\cdot h_s)e^{-s}$$
$$=-X\cdot (h_s+h_u)+(h_s X\cdot h_u-h_u X\cdot h_s)e^{-2s}$$
and
$$e^{2s}{EF}=(e^{2s}+h_u)(-e^{2s}+h_s).$$

This in particular implies

$$EF=0 \Longleftrightarrow h_s=e^{2s} \Longleftrightarrow s=\frac{\ln{h_s}}{2} \Longleftrightarrow \ker{\alpha}=E^s.$$

We want to show that $$\Lambda_s:=\{(s,x) | \ker{(e^{-s}\alpha_-+e^s\alpha_+)}=\text{the dominated bundle of } S(\alpha_-,\alpha_+)_e \}=\{ (s,x) | s=\frac{\ln{h_s}}{2}\}$$ is in fact the Liouville skeleton. Note that $S(\alpha_-,\alpha_+)_e$ is non-singular and partially hyperbolic by Theorem~\ref{generalchar}.
\begin{lemma}
$Y$ is tangent to the section $\Lambda_s$ and therefore, $\Lambda_s$ is invariant under the Liouville flow.
\end{lemma}

\begin{proof}
First note that
$$2(\frac{g}{f})=(e^{2s}+h_u)(1-e^{-2s}h_s)\frac{r_u-r_s}{h_u+h_s}+(1-e^{-2s}h_s)(\frac{X\cdot (h_u +h_s)}{h_u+h_s})+e^{-2s}(X\cdot h_s),$$
which at $\Lambda_s$ it boils down to 
$$ 2(\frac{g}{f})=e^{-2s}(X\cdot h_s) \Longleftrightarrow \frac{g}{f}=X\cdot \frac{\ln{h_s}}{2}.$$

To see that the Liouville vector field preserves $\Lambda_s$, note
$$Y\cdot (e^{2s}-h_s) \bigg|_{\Lambda_s}=(fX+g\partial_s)\cdot (e^{2s}-h_s)=-fX\cdot h_s +2ge^{2s}\bigg|_{\Lambda_s}=0,$$ where the last equality follows from the above observation.

\end{proof}

This is particular implies $\Lambda_s \subset Skel(Y)$. To show the equality, we need to show normal expansion of $Y$ at $\Lambda_s$. It is enough to show that $\tilde{Y}:=X+\frac{g}{f}\partial_s$ expands the $\partial_s$ direction. When considered at $\Lambda_s$, it implies the a $C^1$-expansion of $E^{wn}/\langle X\rangle$ at $\Lambda_s$ which eventually yields $Skel(Y)=\Lambda_s$.

First note that we have

$$[\tilde{Y},\partial_s ]=[X+\frac{g}{f}\partial_s,\partial_s]=-(\partial_s \cdot \frac{g}{f})\partial_s ,$$
which means that $\tilde{Y}$ naturally preserves the $\partial_s$ direction and locally expands it, if and only if, $\partial_s \cdot \frac{g}{f}>0$.

We write
$$\partial_s \cdot \frac{g}{f}=(e^{2s}+h_sh_ue^{-2s})\frac{r_u-r_s}{h_u+h_s}+e^{-2s} h_s\frac{X\cdot (h_u +h_s)}{h_u+h_s}-e^{-2s}(X\cdot h_s)$$
$$=e^{2s}\bigg[ \frac{r_u-r_s}{h_u+h_s}\bigg]+e^{-2s}\bigg[\frac{h_sh_u(r_u-r_s+X\cdot \ln{\frac{h_u}{h_s}})}{h_u+h_s}\bigg]>0,$$
indicating the expansion of the $\partial_s$ direction under the Liouville flow everywhere. This in particular implies that the invariant set $\Lambda_s$ is the skeleton of $Y$ and the underlying Liouville manifold is in fact of finite type.

In particular, at $\Lambda_s=\{ e^{2s}=h_s\}$ we have
\begin{equation}\label{expansioneq}
\partial_s \cdot \frac{g}{f} \bigg|_{\Lambda_s}=h_s\bigg[ \frac{r_u-r_s}{h_u+h_s}\bigg]+h_u\bigg[\frac{r_u-r_s+X\cdot \ln{\frac{h_u}{h_s}}}{h_u+h_s}\bigg]>0,
\end{equation}
indicating the normal expansion of $\langle \partial_s \rangle$ at $\Lambda_s$ along $\tilde{Y}$. The same is true for $Y$ as expansion can be measured in the normal bundle of $Y$, i.e. $T(\mathbb{R}\times M)/\langle Y\rangle$ (see Proposition~\ref{expansion}). 

The following lemma shows that the Liouville vector field preserves a transverse invariant foliation inside the weak normal bunlde $E^{wn}=\langle X,\partial_s \rangle$, which we denote by $E^n$ and call the strong normal bundle. In the Anosov case, the implication is immediate thanks to the normal hyperbolicity of the Liouville flow at its skeleton (by Theorem~\ref{normalh}) which will be observed next. But the following lemma shows that this can be extended to when the LIS supports any non-singular partially hyperbolic flow (where normal hyperbolicity at the skeleton can fail).

\begin{lemma}\label{strongnormal}
$Y$ admits a strong repelling line bundle inside $E^{wn}=\langle X,\partial_s \rangle=\langle Y,\partial_s \rangle$.
\end{lemma}

\begin{proof}
This is a direct consequence of invariant bundle theory of \cite{hps} discussed in Section~\ref{2.3}. In Lemma~\ref{strong}, let $\Lambda=\Lambda_s=Skel(Y)$, $E_1=\langle Y\rangle$, $E_2=\langle Y,\partial_s \rangle$ and $E_3=\langle Y,\partial_s \rangle/ \langle Y\rangle$ and we have
$$m(T_3|E_{3x}) > || T_1 |E_{1x}|| \ \ \text{for all}\ x\in\Lambda \Longleftrightarrow \partial_s \cdot \frac{g}{f} \bigg|_{\Lambda_s}>0,$$
which is the case here. Therefore, $\langle Y \rangle$ has a $Y$-invariant complement in $\langle Y,\partial_s \rangle$, which is a strong unstable bundle.
\end{proof}

The next lemma shows the dynamics of the supported non-singular partially hyperbolic flow is realized as the skeleton dynamics for a LIS. This also means the following. Note that when $\Lambda_s$ is $C^{1+}$, this is equivalent to the restriction of $Y$ to $\Lambda_s$ being a synchronization of $X$.

\begin{lemma}\label{sync}
If $\pi:\mathbb{R}\times M\rightarrow M$ is the natural projection, then $\pi_*(Y)|_{\Lambda_s}=f|_{\Lambda_s}X$ is a synchronization of $X$.
\end{lemma}

Write

$$f=\frac{DE-BF}{AD-BC}=\frac{-2(h_u+h_s)}{J}$$

and compute

$$J=AD-BC$$
$$=[-r_u+r_s]e^{2s}+[h_uh_s(r_s-r_u-X\cdot \ln(\frac{h_u}{h_s}))]e^{-2s}
+[-h_sr_u-X\cdot h_u-r_uh_u -r_sh_u-X\cdot h_s -r_sh_s]$$
$$=[-r_u+r_s]e^{2s}+[h_uh_s(r_s-r_u-X\cdot \ln(\frac{h_u}{h_s}))]e^{-2s}
+[-(h_s+h_u)(r_s+r_u)-X\cdot (h_u+h_s)]$$

Note that $\lim_{s\rightarrow \pm \infty}J=-\infty$ and the maximum of $J$ happens at $$\bigg\{{ [r_s-r_u]e^{2s}=[h_uh_s(r_s-r_u-X\cdot \ln(\frac{h_u}{h_s}))]e^{-2s}} \bigg\},$$
and considering the fact that $\alpha \wedge \mathcal{L}_{\partial_s} \alpha \neq 0$ (see Remark~\ref{yorient}), means $J<0$ everywhere,
implying that $f>0$ everywhere.

At $\Lambda_s$ we have
$$J=h_s(r_s-r_u)+h_u(r_s-r_u-X\cdot \ln(\frac{h_u}{h_s}))+[-(h_s+h_u)(r_s+r_u)-X\cdot (h_u+h_s)]$$
$$J=-2r_u(h_s+h_u)-h_u X\cdot \ln(\frac{h_u}{h_s})-X\cdot (h_u+h_s),$$
which implies
$$\frac{1}{f}=\frac{J}{-2(h_u+h_s)}=r_u+\frac{1}{2}X\cdot{\ln(h_u+h_s)}+\frac{h_u\frac{X\cdot \frac{h_u}{h_s}}{\frac{h_u}{h_s}}}{2(h_u+h_s)}=r_u+\frac{1}{2}X\cdot \ln{(h_u+h_s)}+\frac{X\cdot \frac{h_u}{h_s}}{2(\frac{h_u}{h_s}+1)}$$
$$=r_u+\frac{1}{2}X\cdot \ln{(h_u+h_s)}+\frac{1}{2}X\cdot \ln{(\frac{h_u}{h_s}+1)}.$$

To see that $f|_{\Lambda_s}X$ is in fact a synchronization (we will later discuss a more abstract proof of this fact), consider the 1-form $$\alpha\big|_{\Lambda_s}=[h_u e^{-s}+e^s]\alpha_u \big|_{\Lambda_s}=[\frac{h_u}{\sqrt{h_s}}+{\sqrt{h_s}}]\alpha_u,$$
which defines a norm on $E^u$ and subsequently an expansion rate $\tilde{r}_u$ with respect to $X$. Claiming that $f|_{\Lambda_s}X$ is synchronized is equivalent to show that $f|_{\Lambda_s}=\frac{1}{\tilde{r}_u}$. We have
$$\tilde{r}_u=r_u+X\cdot \ln[\frac{h_u}{\sqrt{h_s}}+{\sqrt{h_s}}]=r_u+X\cdot \ln[(\frac{h_u}{h_s}+1)\sqrt{h_s}]$$
$$=r_u+\frac{1}{2}X\cdot \ln(\frac{h_u}{h_s}+1)+\frac{1}{2}X\cdot \ln{(h_u+h_s)},$$
which proves the claim.


Our arguments so far in this section give a complete and explicit understanding of Liouville dynamics, and in particular the Liouville skeleton, in the case of exponential Liouville pair. Lemma~\ref{lioumaps} indicates how the Liouville dynamics of a general LIS can be related to an exponential one, and in order to formulate our dynamical rigidity theorem for a general LIS, we record the following lemma on how the Liouville skeleton is affected under the elementary maps applied to a general LIS.

\begin{lemma}\label{skel}
For any $(\alpha_-,\alpha_+)_{(\lambda_-,\lambda_+)}\in\mathcal{LIS}^{k*}(M)$, the Liouville skeleton of $\alpha=L(\alpha_-,\alpha_+)_{(\lambda_-,\lambda_+)}$ is graph of a Hölder continuous function in $\mathbb{R}\times M$, which we define as 
$$Skel:\mathcal{LIS}^{k*}(M)\rightarrow C^0(M),$$
which satisfies
$$Skel=Skel\ I_{(z_-,z_+)},$$
$$Skel=Skel\ S_f$$
and
$$Skel\ H^*_\psi=H_\psi Skel.$$
We have $Skel\equiv\Lambda_s:=\{ (s,x)| \ker{(\lambda_-(s,x)\alpha_-+\lambda_+(s,x)\alpha_+)}=\text{the dominated bundle of }S(\alpha_-,\alpha_+)_{(\lambda_-,\lambda_+)} \}$. In particular, $Skel$ is exactly as regular as the dominated bundle of any supported (non-singular partially hyperbolic) flow.
\end{lemma}

\begin{proof}
The properties of $Skel$ as a map $Skel:\mathcal{LIS}^{k*}(M)\rightarrow C^0(M)$ are implied by similar properties of the Liouville vector field in Lemma~\ref{lioumaps}. The fact that we always have $Skel\equiv\Lambda_s$ is because that is the case for exponential Liouville pairs and $\Lambda_s$ is transformed under the elementary maps similar to $Skel$.
\end{proof}



The above lemma indicates that the regularity of the weak dominated bundle of non-singular partially hyperbolic flows, and in particular, the weak invariant bundles of Anosov flows in dimension 3, can be investigated as the regularity of the Liouville skeleton for Liouville interpolation systems, which is proved to be a section of the projection $\pi:\mathbb{R}\times M \rightarrow M$. In this context, one can use classical methods of graph transformations to study the regularity of such invariant bundles instead of the more modern bunching technology discussed in Section~\ref{2.3}. In Section~\ref{8}, we will exploit this viewpoint to reprove the result of Hasselblatt on the $C^1$-regularity of weak invariant bundles of Anosov 3-flows \cite{reg,reg2} (see (1) of Corollary~\ref{c1regweak}), generalize it to non-singular partially hyperbolic 3-flows and also give a parametric version of it using the standard theory of normal hyperbolicity.

We are now ready to formulate our dynamical rigidity theorem. The following implies that the Liouville dynamics is unique up to a certain regularity and is strongly determined by the dynamics a supported non-singular partially hyperbolic 3-flow. We bring this theorem here in its strongest form and will postpone the proof of some elements to the next sections.

\begin{theorem}\label{dynrig}
Assume $X^t$ be a non-singular partially hyperbolic flow on $M$ admitting the dominated splitting $TM/\langle X\rangle\simeq E\oplus E^{wu}$, the weak dominated bundle $E$ is $C^k$ for some $0<k$, both weak invariant bundles are $C^l$ for some $0<l\leq k$, $(\alpha_-,\alpha_+)_{(\lambda_-,\lambda_+)}\in\mathcal{LIS}(M)$ is a supporting LIS and $Y$ is the corresponding Liouville vector field on $\mathbb{R}\times M$. Then,

(1) $Skel(Y)$ is a section of $\pi:\mathbb{R}\times M \rightarrow M$ given as
$$Skel(Y)=\bigg\{(\Lambda_s(x),x)\in \mathbb{R}\times M \ \ \text{ where } \ \ \Lambda_s:M\rightarrow \mathbb{R} \ \ \text{ is determined by }$$
$$\ker{\big[ \lambda_-(\Lambda_s(x),x)\alpha_-+\lambda_+(\Lambda_s(x),x)\alpha_+\big]}=E \bigg\},$$
 implying that $Skel(Y)$ is a $C^k$ section of $\pi$. More precisely, $Skel(Y)$ is exactly as regular as $E$. When $X$ is Anosov, the skeleton is exactly as regular as the weak stable bundle $E^{ws}$ and in particular, it is always a $C^{1+}$ section of $\pi$;
 
 (2) $\pi_*(Y|_{Skel(Y)})\subset TM$ is a synchronization of $X$;
 
 (3) $Y$ preserves a $C^\infty$ exact Lagrangian foliation $\mathcal{F}^{wn}$ containing the flow lines of $Y$ expands a differentiable transverse measure at $Skel(Y)$. Furthermore, $Y$ is normally hyperbolic at $Skel(Y)$, if and only if, $X$ is Anosov;
 
(4) there exists a 1-dimensional foliation $\mathcal{F}^n$ inside $\mathcal{F}^{wn}$ which is transverse to and invariant under the flow of $Y$. Furthermore, $\mathcal{F}^n$ is $C^l$. In particular, $\mathcal{F}^n$ is $C^{1+}$ when $X$ is Anosov;\\

(5) $Skel(Y)$ $C^k$-varies as one $C^{k}$-deforms $(\alpha_-,\alpha_+)_{(\lambda_-,\lambda_+)}$, while fixing $X$;

(6) 
 $Skel(Y)$ is $C^1$-persistent under $C^2$-deformations of a supported Anosov vector field $X$.
 


\end{theorem}

\begin{proof}
We have finished the proof of (1) in Lemma~\ref{skel} and (2) is established in Lemma~\ref{sync}. For (3), the existence of the weak normal foliation $\mathcal{F}^{wn}$ is discussed in Section~\ref{4.1} and the expansion is shown in Equation~\ref{expansioneq}. The fact that normal hyperbolicity at the Liouville skeleton is characterized by the Anosovity of the supported non-singular partially hyperbolic flow also follows from our computations of the Liouville vector field in this section. However, we will give a more straightforward proof for a generalization of this fact in Section~\ref{7} (see Theorem~\ref{pers}).  

The existence of the strong normal foliation $\mathcal{F}^n$ in (4) is implied by Lemma~\ref{strongnormal}. But we will give a more explicit proof in Section~\ref{6} which allows us to show the claimed regularity.

(5) directly follows from the characterization of the skeleton in (1). (6) for $k>1$ follows from normal hyperbolicity in the Anosov case, i.e. (3), and $C^1$-persistence of normally hyperbolic invariant sets (see Theorem~\ref{normalh} (c)). Note that $C^2$-deformations of $(\alpha_-,\alpha_+)_{(\lambda_-,\lambda_+)}$ implies $C^1$-deformations of the Liouville vector field (in fact $C^{1*}$-deformation of the LIS is sufficient by Lemma~\ref{lioureg}). In Section~\ref{8}, we extend this result to when $k<1$ (i.e. considering the class of non-singular partially hyperbolic flows).
\end{proof}

   \begin{figure}[h]
\centering
\begin{overpic}[width=1\textwidth]{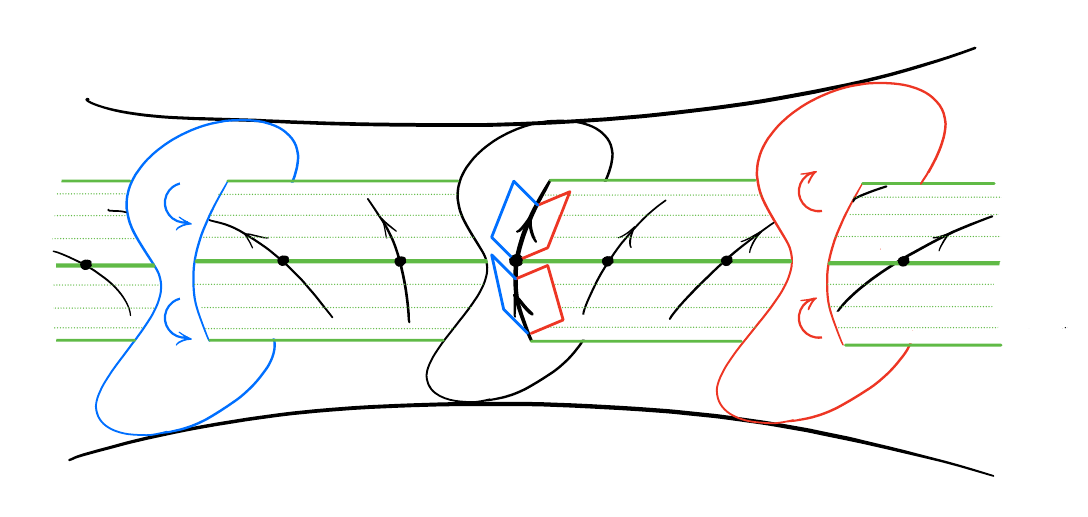}
    \put(55,48){$(M,\alpha_-)$}
      \put(315,48){$(M,\alpha_+)$}
           \put(205,60){$\Lambda_s$}
                \put(420,108){$\mathcal{F}^n$}
                 \put(405,134){$\mathcal{F}^{wn}$}
                    \put(155,36){$\langle Y|_{\Lambda_s}\rangle=P (\alpha_-,\alpha_+)_{(\lambda_-,\lambda_+)}$}
           
              \put(206,136){$\xi_-$}
               \put(247,136){$\xi_+$}
                \put(233,148){$Y$}
                  \put(350,138){$R_+$}
                   \put(75,133){$R_-$}

  \end{overpic}

\caption{Normal expansion at the Liouville skeleton}
\end{figure}
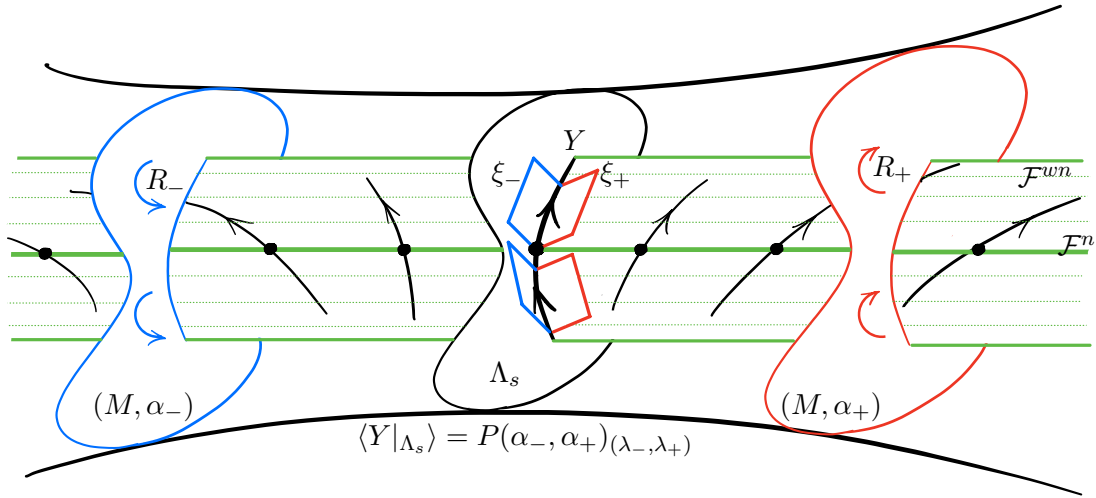

Among other things, the above theorem implies the following. More specifically, this is implied by the explicit understanding of the skeleton dynamics in (2) and the existence of the strong normal bundle in (4). 

\begin{corollary}\label{dynuniq}
Let $X^t$ be a non-singular partially hyperbolic 3-flow with $C^k$ weak dominated bundle. The Liouville vector field induced from a supporting LIS is unique up to $C^k$-conjugacy.
\end{corollary}

Roughly speaking, the standard use of the Moser technique for Liouville manifolds, i.e. Lemma~\ref{moserclassic}, implies that after applying some deffeomorphism the symplectic structure $d\alpha$ in the above setting can be assumed to be fixed under Liouville deformation. Therefore, the rigidity of the Liouville dynamics in Corollary~\ref{dynuniq} should have a geometric counterpart thanks to the fact that $\iota_Y d\alpha=\alpha$. We will explore such geometric rigidity in Section~\ref{5} and show that a novel use of the Moser technique in this contexts proves to refine our understanding of the (even dynamical) Liouville rigidity for Liouville interpolation systems. In fact, while the observations above establish the uniqueness of the Liouville dynamics up to $C^k$-conjugacy (Corollary~\ref{dynuniq}), the application of the Moser technique will improve the regularity of the conjugacy to a $C^\infty$ one.

\section{Geometric rigidity of Liouville interpolations}\label{5}

In the light of the standard application of the Moser technique in Liouville geometry  (see Theorem~\ref{moserclassic}), up to diffeomorphism, we can assume the underlying symplectic structure $d\alpha$ to be fixed during any deformation. Therefore, the rigidity and somewhat uniqueness of the Liouville dynamics explored in Section~\ref{4} should be paralleled with equivalent rigidity phenomena of its dual, the Liouville form $\alpha$. In this section, we explore such rigidity with further exploitation of the Moser technique in the presence of a strict exact Lagrangian foliation, which gives sufficient information to recover the Liouville form strictly (and not just up to homotopy). This viewpoint also helps us go beyond our LIS model and study the regularity of strong normal bundles, which we will discuss in Section~\ref{6}.

\subsection{Moser technique and strict Liouville embeddings}\label{5.1}

In Section~\ref{7}, we will see that the existence of the exact Lagrangian foliation $\mathcal{F}^{wn}$ is not just a coincidence by LIS construction, but in fact, is inherited from the persistent behavior (or normal hyperbolicity) of the underlying Liouville dynamics near its skeleton. 

In the following, a {\em normal diffeomorphism} in the context of LIS s adapted to a fixed (non-singular partially hyperbolic) vector field $X$ is a diffeomorphism of $\mathbb{R}\times M$ which preserves the normal foliation $\mathcal{F}^{wn}$ leaf-wise. {\em Normal isotopy} is defined similarly. We use the Moser technique to show that such isotopies in fact preserve the Liouville form strictly.

The idea is to show that for $k\geq 1$, any 1-parameter family of $\alpha_t$ s induced from a $C^{k*}$ family in $\mathcal{LIS}^{k*}(M)$ is induced from a normal isotopy of $\mathbb{R}_s\times M$, i.e. there is a family of $C^{k*}$ diffeomorphisms $\phi^{t*}\alpha^t=\alpha_0$. 

\begin{remark}\label{lowisotopy}
Moser technique provides a convenient tool, when one wants to study a deformation of Liouville forms $\alpha_t$ induced LIS s. However, it relies on the non-degeneracy of $d\alpha_t$ and therefore, we only apply that to $C^{k*}$-deformation of LIS s for $k\geq 1$, in order to guarantee the vector field generating the desired isotopy is at least $C^1$. This is enough for most purposes. However, we will still find it useful to extend the main rigidity claim to the lower regularity case of $k<1$. For instance, this is helpful in the study of the linearization of the Liouville dynamics of the regularity theory of the strong normal bundle (see Section~\ref{6}).
In this setting, we will exploit explicit isotopies of lower regularity between Liouville forms, i.e. the elementary isotopies introduced in Section~\ref{4.1}. That is, the change of basis, horizontal and scaling isotopies, work in the general $C^{k*}$ setting for any $k>0$. But more care is needed in their use.
\end{remark}

The Liouville condition helps us find a vector field generating the desired isotopy. Here, the underlying manifold $\mathbb{R}_s\times M$ is not compact and hence, such generating vector field is not always complete. Therefore, we need to be more careful trying to do our deformations in a compact sub-domain of $\mathbb{R}_s\times M$ (in contrast to our standard isotopy moves we introduced in Section~\ref{4.1}, which were not compactly supported). This is possible thanks to the fact that any Liouville manifold is uniquely determined, up to strict Liouville equivalence, by a sub-domain containing the skeleton. This is in fact nothing but the LIS version of Lemma~\ref{moserclassic}, where the compactly supported exact error term $\psi^{t*}\alpha_t-\alpha_0$ is shown to vanish. Thanks to the rigidity of Liouville manifolds at their non-compact ends, we can apply our Moser technique beyond the conditions we imposed at the non-compact ends of our manifolds in the definition of LISs (Definition~\ref{lis}) and exploit this to find explicit conjugacies of the dynamics to its linearization at the skeleton (Section~\ref{6}).

We start from the compact embedding version. 

\begin{lemma}\label{emb1}
For any $(\alpha_-,\alpha_+)_{(\lambda_-,\lambda_+)}\in \mathcal{LIS}^{k*}_c(M)$, there exists a strict Liouviile embedding into some exponential LIS $(\bar{\alpha}_-,\bar{\alpha}_+)_e$.
\end{lemma}

\begin{proof}
Consider $\alpha=L(\alpha_-,\alpha_+)_{(\lambda_-,\lambda_+)}$ defined on $[-1,1]_s \times M$. By \cite{massoni}, there is some exponential $(\bar{\alpha}_-,\bar{\alpha}_+)_e$ such that $\ker{\bar{\alpha}_\pm}=\ker{\alpha_\pm}$. After a change of basis isotopy, we write $(\bar{\alpha}_-,\bar{\alpha}_+)_{(\bar{\lambda}_-,\bar{\lambda}_+)}$ and after a horizontal isotopy, we get $(\bar{\alpha}_-,\bar{\alpha}_+)_{(e^{-s+w},s^{s+w})}$ for some distortion function $w:\bar{W} \rightarrow\mathbb{R}$ defined on a compact subset $\bar{W}\subset W=\mathbb{R}\times M$, where $\bar{W}$ is the image of $[-1,1]_s \times M$ under these isotopies. Extend $w$ to an arbitrary compactly supported map $W\rightarrow \mathbb{R}$. Apply a (compactly supported) scaling isotopy, we have an embedding as claimed.
\end{proof}

We then note that this works for a 1-parameter family as well. 

\begin{lemma}\label{newmosercomp}
Let $\alpha_t$ be a $C^{k*}$ family of Liouville forms induced from a family of LIS s in $P^{-1}[X]\subset \mathcal{LIS}_c^{k*}(M)$ on $[-1,1]\times M$ and consider the strict Liouville embedding $i:([-1,1]\times M,\alpha_0)\rightarrow (\mathbb{R}\times M, L(\bar{\alpha}_{0,-},\bar{\alpha}_{0,+})_e))$ as given in Lemma~\ref{emb1}. Then, there is an isotopy $\psi^t:\mathbb{R}\times M \rightarrow \mathbb{R}\times M$ such that $\psi^t \circ i:([-1,1]\times M,\alpha_t)\rightarrow (\mathbb{R}\times M, L(\bar{\alpha}_{0,-},\bar{\alpha}_{0,+})_e)$ is a strict Liouville embedding.
\end{lemma}

\begin{proof}
Let $\bar{\alpha}_t:=i^{-1*}\alpha_t$ be the family of Liouville forms defined on the compact set $\bar{W}:=i([-1,1]_s\times M)\subset \mathbb{R}_s\times M$ and note that by assumption $\bar{\alpha}_0$ is the restriction of $L(\bar{\alpha}_{0,-},\bar{\alpha}_{0,+})_e)$ to such compact set. The Liouville forms $\bar{\alpha}_t$ still annihilates $E^{wn}$, since fixing $[X]$ is equivalent to fixing the exact Lagrangian foliation $\mathcal{F}^{wn}$.

 We use Moser technique to show that such family can be realized as a normal isotopy of the ambient $(\bar{\alpha}_{0,-},\bar{\alpha}_{0,+})_e$. Assume $\psi^{t}$ is such isotopy generated by a vector field $V_t \subset E^{wn}=\langle X,\partial_s \rangle$ on $\mathbb{R}\times M$ and note $\dot{\alpha}_t \in  Annih^1(E^{wn})$.

Let $\bar{\alpha}=L(\bar{\alpha}_{0,-},\bar{\alpha}_{0,+})_e$. On $\bar{W}$, we have
$$\psi^{t*}\bar{\alpha}_0=\bar{\alpha}_t \Longleftrightarrow \psi^{-t*}\bar{\alpha}_t=\bar{\alpha}_0 \Longleftrightarrow \frac{d}{dt}(\psi^{-t*}\bar{\alpha}_t)=0 \Longleftrightarrow \psi^{-t*}(\dot{\bar{\alpha}}_t+\mathcal{L}_{-V_t}\bar{\alpha}_t)=0 \Longleftrightarrow \dot{\bar{\alpha}}_t=\iota_{V_t}d\bar{\alpha}_t,$$
which writing $V_t=p_t Y_t+q_t \partial_s$, is equivalent to
$$\dot{\bar{\alpha}}_t=-p_t\bar{\alpha}_t -q_t \mathcal{L}_{\partial_s} \bar{\alpha}_t.$$
Therefore, the fact that $(\bar{\alpha}_t$,$\mathcal{L}_{\partial_s}\bar{\alpha}_t)$ forms a $C^k$ basis for $ Annih^1(E^{wn})$ means that such $p_t$ and $q_t$ can be determined and would $C^k$ change as  we $C^{k*}$ defrom LIS s. Extend this $V_t$ to a compactly supported normal vector field $V_t:W\rightarrow \mathbb{R}$.This defines a family of complete vector fields $V_t$ (all these vector fields are compactly supported) and the consequently, the desired isotopy of embeddings (note that we are assuming $k\geq 1$).

\end{proof}

We finally notice that that any path of such strict Liouville embeddings $\psi^t\circ i:([-1,1]\times M,\alpha_t)\rightarrow  (\mathbb{R}\times M, L(\bar{\alpha}_{0,-},\bar{\alpha}_{0,+})_e)$ can be uniquely extended to an isotopy of the completions $(\mathbb{R}\times M,\bar{\alpha}_t)$. This basically follows from the discussion in Remark~\ref{isoextension}.

\begin{lemma}\label{moserext}
Any family of embeddings of compact LISs can be extended to a family of strict Liouville equivalences of their completions in a canonical way.
\end{lemma}

Therefore, we have our isotopy claim as follows.

\begin{corollary}
Let $\alpha_t$ be a family of Liouville froms induced on from a family of LISs in $P^{-1}[X]\subset \mathcal{LIS}^{k*}(M)$, there exists a family of diffeomorphisms $\psi^t:\mathbb{R}\times M\rightarrow \mathbb{R}\times M$ satisfying $\psi^{t*}\alpha_t=\alpha_t$.
\end{corollary}

\begin{proof}
Take $N>0$ large enough that $[-N,N]\times M$ contains the skeleton of all $\alpha_t$ s. By Lemma~\ref{newmosercomp}, there is a family of strict Liouville embeddings $([-N,N]\times M,\alpha_t)\rightarrow (\mathbb{R}\times M,\alpha_0)$. By Lemma~\ref{moserext}, this can be extended to a family of strict Liouville equivalences $(\mathbb{R}\times M,\alpha_t)\rightarrow (\mathbb{R}\times M,\alpha_0)$ which is an isotopy of $(\mathbb{R}\times M,\alpha_0)$.
\end{proof}

This essentially means that fixing the flow $X$, the Liouville form induced on $\mathbb{R}\times M$ from a LIS is unique up to strict Liouville equivalence. This helps us promote Mitsumatsu's construction to a 1-to-1 equivalence between appropriate equivalence classes. 
We state this for $k=\infty$ as is mostly used.

\begin{theorem}\label{1to1}
There is a 1-to-1 correspondence between synchronized partially hyperbolic vector fields, up to $C^\infty$-conjugacy, and Liouville forms induced from some LIS on in $\mathcal{LIS}(M)$, up to strict Liouville equivalence, i.e.

$$\bigg\{\substack{\text{Positive reparametrization class of} \\  \text{non-singular partially hyperbolic flows} \\  \text{up to $C^\infty$-conjugacy}} \bigg\} \mbox{\Large$\overset{\text{1-to-1}}{\longleftrightarrow}$}
\bigg\{\substack{\text{Liouville forms induced from some LIS on $\mathbb{R}\times M$} \\  \text{up to strict Liouville equivalence}} \bigg\}.$$
\end{theorem}

\begin{proof}
The correspondence is made via the skeleton dynamics. More precisely, let $\mathcal{PHF}_{sync}(M)$ be the space of synchronized $C^\infty$ non-singular partially hyperbolic flows on $M$, the following map introduced in Section~\ref{4} (see Lemma~\ref{sync})
$$\begin{cases}
\mathcal{LIS}(M)\rightarrow \mathcal{PHF}_{sync} (M)\\
 (\alpha_-,\alpha_+)_{(\lambda_-,\lambda_+)}\mapsto \pi_*(Y|_{\Lambda_s})
\end{cases},$$
recovers the dynamics of the supported flow $P (\alpha_-,\alpha_+)_{(\lambda_-,\lambda_+)}$ from the skeleton of $L(\alpha_-,\alpha_+)_{(\lambda_-,\lambda_+)}$ (up to some positive reparametrization). Furthermore, notice that synchronization is unique in the positive reparametrization class of the supported flow, up to conjugacy (see Proposition~\ref{syncuniq} and Remark~\ref{syncuniqph}).

On the other hand, if two non-singular partially hyperbolic flows $X_1$ and $X_2$ are $C^\infty$-conjugate after a positive reparametrization via the map $\psi:M\rightarrow M$, then $P^{-1}[X_1]=P^{-1}[\psi_*(X_2)]\subset \mathcal{LIS}(M)$, completing the proof.
\end{proof}

An important corollary, which rectifies all the differences between the linear and exponential Liouville pairs, is as follows.

\begin{corollary}\label{linemb}
Fixing a non-singular partially hyperbolic flow (up to positive reparametrization), for any supporting compact LIS $(\alpha_-,\alpha_+)_{(\lambda_-,\lambda_+)}$ and any supporting (non-compact) LIS $(\bar{\alpha}_-,\bar{\alpha}_+)_{(\bar{\lambda}_-,\bar{\lambda}_+)}$, there exists a strict Liouville embedding
$$i:([N_-,N_+]\times M,L(\alpha_-,\alpha_+)_{(\lambda_-,\lambda_+)})\rightarrow (\mathbb{R}\times M, L(\bar{\alpha}_-,\bar{\alpha}_+)_{(\bar{\lambda}_-,\bar{\lambda}_+)}).$$
This is in particular true for any linear and exponential Liouville pairs supporting the same (positive reparametrization class of) flows.
\end{corollary}

\begin{remark}
We can also state above 1-to-1 correspondence in terms of compact LIS s. However, in order to do so, the appropriate equivalence relation between two compact LIS s would be strict Liouville equivalence after completion.
\end{remark}

Moreover, the isotopy and the subsequent strict Liouville equivalence, given in Theorem~\ref{1to1} is of regularity $C^\infty$ (assuming the LIS s are $C^\infty$). Such map naturally provides a $C^\infty$-conjugacy of the underlying Liouville vector fields, improving Corollary~\ref{dynuniq}.

\begin{corollary}\label{dynuniq2}
Let $X^t$ be a non-singular partially hyperbolic 3-flow with $C^k$ weak dominated bundle. The Liouville vector field induced from a $C^\infty$ supporting LIS is unique up to $C^\infty$-conjugacy.
\end{corollary}

   \begin{figure}[h]
\centering
\begin{overpic}[width=1\textwidth]{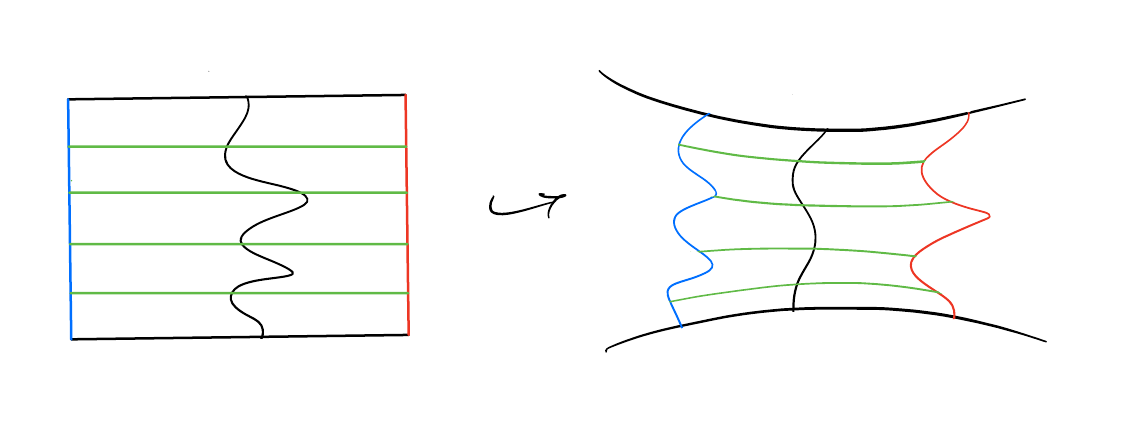}
    \put(80,20){$(\alpha_-,\alpha_+)_l$}
      \put(320,20){$(\bar{\alpha}_-,\bar{\alpha}_+)_e$}
      
         \put(83,80){$\Lambda_s$}
          \put(320,80){$\bar{\Lambda}_s$}

           \put(215,97){$i$}

  \end{overpic}

\caption{Strict Liouville embedding of any compact into non-compact LIS}
\end{figure}
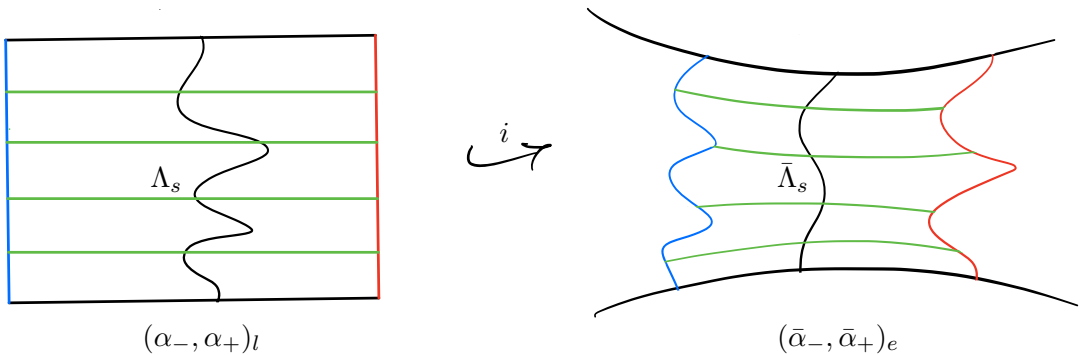

\subsection{A fibration theorem}\label{5.2}

Theorem~\ref{1to1} provide a 1-to-1 correspondence between classes of LIS s and non-singular partially hyperbolic flows, indicating that all the choices in the definition (including the interpolation functions and the chosen contact forms), under than the supported flow, are redundant. The goal of this section is to show that the space of such choices forms a fibration over the space of non-singular partially hyperbolic flows. Similar fibration theorem is proved by Massoni \cite{massoni} for exponential Liouville pairs. Here, we start by revisiting the exponential Liouville pairs. We then show that such fibration theorem can be extended and with more care, information about the skeleton and the adapted norm involved can be kept track of. An important feature of this viewpoint is the fact that the space of LIS s with higher regularity can be deformation retracted into the space of LIS s with lower regularity, a fact which will be used when studying the linearizations of LIS s in Section~\ref{6}.

\subsubsection{Exponential computations}\label{5.2.1}

First, we want to give an explicit fibration of exponential Liouville pairs which we will base our general fibration theorem on next. The essence of this construction is that in the exponential model, the Liouville condition on $\mathbb{R}_s\times M$ can be reduced to its skeleton $\Lambda_s$. Consider an exponential LIS $(\alpha_-,\alpha_+)_e$ with the skeleton $\Lambda_s:M\rightarrow\mathbb{R}$ and $\alpha_u:=\frac{1}{2}\alpha|_{s=\Lambda_s}$ . Let $X$ be a supported partially hyperbolic vector field with the splitting $TM\simeq E\oplus E^u$ (in particular, $\ker{\alpha_u}=E$). Theorem~\ref{generalchar} shows that $r_u$, the expansion rate of $\alpha_u$, is positive (by the Liouville condition at $\Lambda_s$). We can write
$$\begin{cases}
\alpha_+=e^{-\Lambda_s}(h_+\alpha_u-\alpha_s) \\
\alpha_-=e^{\Lambda_s}(h_-\alpha_u+\alpha_s)
\end{cases},$$
for some non-vanishing 1-form $\alpha_s$ (with $\ker{\alpha_s}=E^{wu})$) and positive functions with $h_\pm:M\rightarrow \mathbb{R}_{>0}$ with $h_-+h_+=2$. Note that the positive/negative contact condition for $\alpha_\pm$ implies
$$\begin{cases}
 X\cdot \ln{h_+}+r_u-r_s>0 \\
 X\cdot \ln{h_-}+r_u-r_s>0
\end{cases}
\Leftrightarrow
\begin{cases}
X\cdot \ln{h_+}+r_u-r_s>0 \\
\frac{-X\cdot h_+}{2-h_+}+r_u-r_s>0
\end{cases}.
$$

The two equations together imply $r_u-r_s>0$ (since the signs of $X\cdot \ln{h_+}$ and $\frac{-X\cdot h_+}{2-h_+}$ differ). On the other hand, if $\alpha_u$ and $\alpha_s$ are given such that $r_u>0$ and $r_s<r_u$, there is an open set of allowable functions for $h_+$. The Liouville condition is preserved under a horizontal isotopy $s\mapsto s-\Lambda_s$ to which gives an exponential LIS of the form $(\bar{\alpha}_-:=h_-\alpha_u+\alpha_s,\bar{\alpha}_+:=h_+\alpha_u-\alpha_s)_e$, whose skeleton is $\Lambda_{s}\equiv 0$. Letting $\bar{\alpha}:=L(\bar{\alpha}_-,\bar{\alpha}_+)_e$, the Liouville condition reads
$$0<\frac{1}{2}\iota_X\iota_{\partial_s}(d\bar{\alpha} \wedge d\bar{\alpha})=(\mathcal{L}_{X}\bar{\alpha})\wedge(\mathcal{L}_{\partial_s}\bar{\alpha})=(e^{-s}\mathcal{L}_X\bar{\alpha}_-+e^s\mathcal{L}_X\bar{\alpha}_+)\wedge (-e^{-s}\bar{\alpha}_-+e^s\bar{\alpha}_+)$$
$$=e^{-2s}[-(\mathcal{L}_X\bar{\alpha}_-) \wedge \bar{\alpha}_-]+e^{2s}[(\mathcal{L}_X\bar{\alpha}_+) \wedge \bar{\alpha}_+]+\mathcal{L}_X{(\bar{\alpha}_- \wedge \bar{\alpha}_+)}=:A(s,x)$$

The above shows that the poistive and negative contactness of $\alpha_+$ and $\alpha_-$ is a necessary condition for the Liouville condition to hold everywhere. Noting that
$$\mathcal{L}_X{(\bar{\alpha}_- \wedge \bar{\alpha}_+)}=\mathcal{L}_X{(h_-+h_+)(\alpha_s \wedge \alpha_u)}=2\mathcal{L}_X(\alpha_s \wedge \alpha_u)=2(r_s+r_u)(\alpha_s \wedge \alpha_u),$$
we continue the above computation
$$A(s,x)=[e^{-2s}(X\cdot h_- +h_-(r_u-r_s))+e^{2s}(X\cdot h_+ +h_+(r_u-r_s))+2(r_s+r_u)](\alpha_s\wedge\alpha_u).$$
With $\lim_{s\rightarrow \pm \infty}A(s,x)=+\infty$ for any $x\in M$, and the minimum (with respect to $s$) of the above expression happens at $$e^{2s}=\sqrt{\frac{X\cdot h_- +h_-(r_u-r_s)}{X\cdot h_+ +h_+(r_u-r_s)}},$$
yielding
$$\min_{s}\frac{1}{2}\iota_X\iota_{\partial_s }(d\bar{\alpha} \wedge d\bar{\alpha})/ (\alpha_s \wedge \alpha_u)=2\sqrt{(X\cdot h_- +h_-(r_u-r_s))(X\cdot h_+ +h_+(r_u-r_s))}+2(r_s+r_u)$$
$$=2\sqrt{(-X\cdot h_+ +(2-h_+)(r_u-r_s))(X\cdot h_+ +h_+(r_u-r_s))}+2(r_s+r_u)=2\sqrt{(2(r_u-r_s)-B)B}+2(r_s+r_u),$$
where $B=X\cdot h_+ +h_+(r_u-r_s)$. On the other hand,
$$\sqrt{(2(r_u-r_s)-B)B}+(r_s+r_u)\geq (r_u-r_s)+(r_u+r_s)=2r_u.$$

Therefore, as long as we have the necessary conditions of $r_u>0$ (the Liouville condition at the skeleton) and $\pm$-contactness of $\alpha_{\pm}$, the Liouville condition is automatically satisfied elsewhere.

This means that the following data determines an exponential LP uniquely: (1) the skeleton function $\Lambda_s:M\rightarrow \mathbb{R}$; (2) the 1-form $\alpha_u$ with $\ker{\alpha_u}=E$ and expansion rate $r_u>0$ (to be interpreted as $\frac{1}{2}\alpha|_{\Lambda_s}$); (3) the 1-form $\alpha_s$ with $\ker{\alpha_s}=E^{wu}$ and expansion rate $r_s<r_u$; (4) a function $h_+:M\rightarrow (0,2)$ satisfying 
$$\begin{cases}
r_u-r_s+X\cdot \ln{h_+}>0 \\
r_u-r_s-\frac{X\cdot \ln{h_+}}{2-h_+}>0
\end{cases}.$$

Given this information, we have can define
$$\begin{cases}
\alpha_+:=e^{-\Lambda_s}(h_+\alpha_u-\alpha_s) \\
\alpha_-:=e^{\Lambda_s}((2-h_+)\alpha_u+\alpha_s)
\end{cases},$$
and the previous discussion implies that $(\alpha_-,\alpha_+)_e$ is an exponential Liouville pair. If the dominated bundle $E$ is $C^k$, this argument in fact gives an explicit fibration of the space of $C^{k*}$ exponential Liouville pairs over supported partially hyperbolic flow up to positive reparametrization.

\subsubsection{Extension to the LIS setting}\label{5.2.2}

The goal of this section is to formalize the observations and the constructions of Section~\ref{5.1} in terms of a Serre fibration of the space of Liouville interpolating systems over the space of non-singular partially hyperbolic flows up to positive reparamterization $P:\mathcal{LIS}(M)\rightarrow S\mathcal{PHF}(M)$ . This generalizes and refines the fibration theorem of Massoni \cite{massoni} by showing that such fibration can carry information about the synchronization of the flow (or equivalently, the adapted norm induced on the stable bundle) as well as the skeleton.

Now, suppose a general LIS $(\alpha_-,\alpha_+)_{(\lambda_-,\lambda_+)}$ is given. We can write 
$$\bar{s}=[\frac{1}{2}\ln{\frac{\lambda_+(s,.)}{\lambda_-(s,.)}}-\frac{1}{2}\ln{\frac{\lambda_+(\Lambda_s,.)}{\lambda_-(\Lambda_s,.)}}]+\Lambda_s,$$
which yields
$$\alpha:=L(\alpha_-,\alpha_+)_{(\lambda_-,\alpha_+)}=\lambda_-\alpha_-+\lambda_+\alpha_+=e^{w(\bar{s},x)-w(\Lambda_s,x)}[e^{-\bar{s}}\alpha_-+e^{\bar{s}}\alpha_+],$$
where $w(\bar{s},x)=\frac{1}{2}\ln{(\lambda_-(s,x)\lambda_+(s,x))}$. Therefore, if we apply a change of basis operation $$(\alpha_-,\alpha_+)_{(\lambda_-,\lambda_+)}\mapsto (e^{w(\Lambda_s,.)}\alpha_-,e^{w(\Lambda_s,.)}\alpha_+)_{(e^{-w(\Lambda_s,.)}\lambda_-,e^{-w(\Lambda_s,.)}\lambda_+)},$$
which now has no distortion at $\Lambda_s$. Therefore, a horizontal map preserving $\Lambda_s$ as above gives us a LIS of the form $(\alpha_-,\alpha_+)_{(e^{-s+w},e^{s+w})}$, where $w:\mathbb{R}_s\times M\rightarrow$ is the distortion function satisfying $w(\Lambda_s,.)\equiv 0$ and
$$\begin{cases}
\alpha_+=e^{-\Lambda_s}[h_+\alpha_u-\alpha_s] \\
\alpha_-=e^{\Lambda_s}[h_-\alpha_u+\alpha_s]
\end{cases},$$
where $\alpha_u=\frac{1}{2}\alpha|_{\Lambda_s}$ (i.e. $h_-+h_+=2$), $\alpha_s$ is determined as the "$\alpha_s$
 part" of $\alpha_-|_{\Lambda_s}$. Now, the $\pm$-contactness of $\alpha_\pm$ implies 
 $$\begin{cases} r_u-r_s+X\cdot \ln{h_+}>0 \\
r_u-r_s-\frac{X\cdot \ln{h_+}}{2-h_+}>0
\end{cases},$$
yielding $r_u-r_s>0$. Therefore, by the exponential computation of the previous section $(\alpha_-,\alpha_+)_e$ is also an exponential Liouville pair. Therefore, an scaling isotopy between $(\alpha_-,\alpha_+)_{(e^{-s+w},e^{s+w})}$ and $(\alpha_-,\alpha_+)_e$ exists which preserves $\Lambda_s$ since $w(\Lambda_s,.)\equiv 0$.

This means that the only other (compared to the exponential case) piece of information needed to determine a general LIS is the choice of any distortion function $w:\mathbb{R}_s\times M \rightarrow \mathbb{R}$ such that $w(\Lambda_s,.)\equiv0$.
So, assuming the dominated bundle $E$ being $C^k$, this essentially provides a Serre fibration of $C^{k*}$ LIS s over the space of exponential LIS s, which is then fibered over the space of partially hyperbolic flows by the construction of the previous section and/or Massoni \cite{massoni}. Also, note how much we can fiber through the skeleton and/or synchronization space, i.e. we have the Serre fibration of the space of $C^{k*}$ LIS s $\mathcal{LIS}^{k*}(M)\rightarrow \mathcal{PHF}_{sync}(M)$,
over the space of synchronized partially hyperbolic flows with $C^k$ dominated bundle $E$.

The last piece we need in the argument to extend it to the smooth Serre fibration of the form $\mathcal{LIS}(M)\rightarrow \mathcal{PHF}_{sync}(M)$ via an argument of Massoni. More precisely, Massoni  proves a similar statement in a slightly different setting (\cite{massoni}, Lemma~4.6). His proof in fact works the same in our setting.

\begin{lemma}
The natural map $i:P^{-1}[X]\subset\mathcal{LIS}(M)\rightarrow P^{-1}[X]\subset \mathcal{LIS}^{k*}(M)$, i.e. fixing $X$, the inclusion of the space of $C^{k*}$ LIS s into the space of $C^{k*}$ LIS s, is a homotopy equivalence.
\end{lemma}

Massoni  proves a very similar statement in a slightly different setting (\cite{massoni}, Lemma~4.6) and using standard algebraic topology theory. More specifically, in his case, $k=1$ and he restricts to {\em Anosov} Liouville pairs. But his proof in fact works exactly the same in our slightly more general setting. Here, we bring the outline for completion.

\begin{proof}
As argued in \cite{massoni}, we first recall that homotopy equivalences are local \cite{partition}, in the sense that $i:\mathcal{LIS}(M,X)\rightarrow \mathcal{LIS}^{k*}(M,X)$ is a homotopy equivalence, if there exists a numerable open cover $\mathcal{U}$ of $\mathcal{LIS}^{k*}(M,X)$ such that (1) $\mathcal{U}$ is stable under intersection and (2) for every $U\in\mathcal{U}$, we have $i:i^{-1}(U)\rightarrow U$ is a homotopy equivalence. Let $U$ be a cover of $\mathcal{LIS}^{k*}(M,X)$ by small $C^{k*}$-balls (which is numerable since $\mathcal{LIS}^{k*}(M,X)$ is metrizable), which after a refinement includes all finite self-intersections, and its elements are convex in $||.||_{C^k}$-norm. Note that even when $k<1$, we can use the Hölder norm $||.||_{C^k}$. Since $C^\infty$ LIS s supporting $X$ is dense in $\mathcal{LIS}^{k*}(M,X)$ as a subset of $$\Omega^1(M)\times\Omega^1(M)\times C(\mathbb{R}_s\times M)\times C(\mathbb{R}_s\times M),$$
every $C^{k*}$-open ball in $\mathcal{LIS}^{k*}(M,X)$ intersects $\mathcal{LIS}(M,X)$. The induced cover on $\mathcal{LIS}(M,X)$ is still convex with the $||.||_{C^k}$-norm. Therefore, the conditions of homotopy equivalence are satisfied.
\end{proof}

\begin{theorem}\label{fib}
The map
$$\begin{cases}
\mathcal{LIS}(M) \rightarrow S \mathcal{PHF}(M)  \\
(\alpha_-,\alpha_+)_{(\lambda_-,\lambda_+)}  \mapsto [\pi(F(\alpha_-,\alpha_+)_{(\lambda_-,\lambda_+)} |_{\Lambda_u})]
\end{cases}$$
is a Serre fibration of the space of Liouville interpolating systems over the space of non-singular partially hyperbolic flows up to positive reparametrization. This
\end{theorem}

\begin{proof}
We have shown that $\mathcal{LIS}^{k*}(M)$ can fibered over the space of exponential Liouville pairs (if the dominated bundle $E$ is $C^k$). Then, use Massoni's fibration theorem to complete the proof. See Theorem.4.3 and Lemma 4.6 of Massoni \cite{massoni}.
\end{proof}

\begin{remark}\label{fibfilter}
Our construction shows that the fibration in Theorem~\ref{fib} can be filtered as
$$\mathcal{LIS}(M) \overset{\pi_1}{\rightarrow} \mathcal{ELP}(M) \overset{\pi_2}{\rightarrow}\mathcal{ELP}_{\Lambda_s}(M)  \overset{\pi_3}{\rightarrow}\mathcal{PHF}_{sync}(M) \overset{\pi_4}{\rightarrow} S \mathcal{PHF}(M),$$
where $\mathcal{ELP}_{\Lambda_s}(M)$ is the space of exponential Liouville pairs with skeleton $\Lambda_s$ and $\mathcal{PHF}_{sync}(M)$ is the space of synchronized non-singular partially hyperbolic flows on $M$. Therefore, in our fibration theorem of Theorem~\ref{fib}, we can equip the positive reparametrization class of non-singular partially hyperbolic flows with auxilliary information such as the synchronization (or equivalently, the information of an expanding norm on $E^u$, up to constant scaling) or the skeleton graph.
\end{remark}

\begin{corollary}
All invariants of Anosov flows through LP construction gives invariants of partially hyperbolic homotopies.
\end{corollary}

\begin{remark}
Note that the Liouville homotopic invariants one can derive from an LIS (see \cite{clmm}) are invariants of homotopy through non-singular partially hyperbolic flows and the construction of the DA bi-contact deformation discussed in Section~\ref{3.5} implies that during a Liouville homotopy through LIS s, the orbit structure of the supported flow can in fact change and therefore, the (isotopy class of the) exact Lagrangian foliation $\mathcal{F}^{wn}$ is not an invariant of Liouville homotopy in general.
\end{remark}

\subsection{Important example:
From linear to exponential model (and back)}\label{5.3}

 The goal of this section is to take a closer look at the embedding of the linear model into the exponential one. Even though Corollary~\ref{linemb} provides this in more generality, it is still useful to construct explicitly construct such embeddings using the elementary maps of Section~\ref{4.1}. Firstly, both linear and exponential Liouville pairs have appeared in the literature most often and while the constructions are similar in nature (other than compactness), it remained curious whether there is any {\em real} difference between the two Liouville geometric models. The issue is discussed more explicitly by Massoni \cite{massoni}. Secondly, the constructions will be useful to study the linearization of a Liouville dynamics for a general LIS in Section~\ref{6}, as the use of elementary maps allows use to formulate the result in regularities less than $1$ as needed in the most general setting. Recall that the use of Moser technique is limited to when we have at least $C^1$-regularity (which is enough if we restrict to the Anosov case).
 
First, write a linear Liouville pair given as $\alpha=L(\alpha_-,\alpha_+)_l$. If we want to strictly Liouville embed $([-1,1]\times M,\alpha)$ into $(\mathbb{R}\times M,L(\bar{\alpha}_-,\bar{\alpha}_+)_e)$, we need $$\alpha=e^{-s}\bar{\alpha}_-+e^s\bar{\alpha}_+\big|_{s\in\{ -1,1\}} \Longleftrightarrow \begin{cases}
2\alpha_+=e^{-1}\bar{\alpha}_-+e\bar{\alpha}_+ \\
2\alpha_-=e\bar{\alpha}_-+e^{-1}\bar{\alpha}_+
\end{cases},$$
which gives
$$\alpha=\frac{1-s}{2}(e\bar{\alpha}_-+e^{-1}\bar{\alpha}_+)+\frac{1+s}{2}(e^{-1}\bar{\alpha}_-+e\bar{\alpha}_+)=\lambda_- \bar{\alpha}_- +\lambda_+\bar{\alpha}_+,$$
where
$$\begin{cases}
\lambda_-=\frac{e+e^{-1}}{2}-\frac{e-e^{-1}}{2}s \\
\lambda_+=\frac{e+e^{-1}}{2}+\frac{e-e^{-1}}{2}s
\end{cases}.$$

Letting $$\bar{s}:=\begin{cases}
\bar{s}:=\frac{1}{2}\ln{\frac{\lambda_+}{\lambda_-}} \\
w(\bar{s})=\frac{1}{2}\ln{(\lambda_-\lambda_+)}=\frac{1}{2}\ln\big( (\frac{e+e^{-1}}{2})^2-(-\frac{e-e^{-1}}{2})^2s^2 \big)
\end{cases},$$ and we have
$$\alpha=e^{-\bar{s}+w(\bar{s})}\bar{\alpha}_- + e^{-\bar{s}+w(\bar{s})}\bar{\alpha}_+,$$
defined for $-1\leq \bar{s}\leq 1$ where $w(-1)=w(1)=0$. Extend $\alpha$ arbitrarily to $\mathbb{R}\times M$ by extending $w(\bar{s})$, while respecting the Liouville condition, i.e. $w$ is admissible in the sense of Lemma~\ref{genscale}, or, equivalently, $Y_0\cdot w >-1$ where $Y_0$ is the Liouville vector field of. Notice that $\alpha$ is not necessarily a Liouville form outside $s\in [-1,1]$.

Now, by Theorem~4.13 of \cite{massoni}, there exists an exponential LIS $(\tilde{\alpha}_-,\tilde{\alpha}_+)_e$ such that $\ker{\tilde{\alpha}_\pm}=\ker{\bar{\alpha}_\pm}$. The computations of the previous sections for scaling maps still holds wherever $\alpha$ is Liouville. Therefore, we can get a map, via scaling along the Liouville flow of $L(\tilde{\alpha}_-,\tilde{\alpha}_+)_e$, which matches the Liouville flow of $\alpha$ wherever the latter is defined. The restriction of this map to $[-1,1]\times M$ defines the desired embedding into the exponential model of $(\tilde{\alpha}_-,\tilde{\alpha}_+)_e$. 

Since we have done the deformation using the elementary maps, everything can be done in the category of $\mathcal{LIS}^{k*}(M)$ for any $k>0$, i.e. including the low regularity.

\begin{theorem}\label{linemb2}
 For any linear Liouville pair $(\alpha_-,\alpha_+)_l\in\mathcal{LLP}^{k*}(M)$, there exists an exponential Liouville pair $(\bar{\alpha}_-,\bar{\alpha}_+)_e \in \mathcal{ELP}^{k*}(M)$ and a $C^{k*}$ strict Liouville embedding $i:([-1,1]\times M, L(\alpha_-,\alpha_+)_l)\rightarrow (\mathbb{R}\times M,L(\bar{\alpha}_-,\bar{\alpha}_+)_e)$.
\end{theorem}

\section{Linearization and regularity of the strong normal foliations}\label{6}

The goal of this section is to study the linearization of Liouville dynamics at the skeleton of a Liouville manifold induced from an LIS. The linearization of a dynamical systems at an invariant submanifold is well studied, especially in the case of normally hyperbolic $C^1$ invariant submanifolds, they play an important role in the rigidity in dynamics \cite{ps,hps}. We want to apply such analysis to our LIS construction, and in fact, extend to the lower regularity case of $C^k$ for $k<1$. We will see that in those cases (which is possible for non-Anosov non-singular partially hyperbolic flows), conjugacy to a linearized dynamics is still possible.

The linearization of the Liouville form can be written as $\alpha=\alpha_u+s\beta$ with $s=0$ is the skeleton, which can be written as $\alpha=(1-s)\frac{\alpha_u-\beta}{2}+(1+s)\frac{\alpha_u}{2}$, where $\alpha_u$ is foliation 1-form for $E^{ws}$ and $\frac{\alpha_u\mp\beta}{2}$ is a $\pm$ contact structure. However, this is not necessarily Liouville on all of $\mathbb{R}_s\times M$.

One interesting feature of linear Liouville pairs is the fact that it provides a model of the linearization of the Liouville dynamics at the skeleton of a LIS. Note that in the symmetric case (which by construction, can be assumed to be $C^{k*}$ whenever the weak invariant bundles are $C^k$), we have
$$\begin{cases}
\alpha+:=\alpha_u-\alpha_s \\
\alpha_-:=\alpha_u+\alpha_s
\end{cases} \Rightarrow \alpha:=(1-s)\alpha_-+(1+s)\alpha_+,$$
where a simple computation (previously given in \cite{hoz4}) gives the Liouville vector field as
$$Y=\frac{1}{r_u}X+2\frac{r_u-r_s}{r_u}s\partial_s,$$
where $X$ is any partially hyperbolic vector field supported by such pair, which can be written as
$$Y=X+2(1-r_s) s\partial_s,$$ if we take $X$ to be the synchronization with respect to $\alpha_u$. Note that in the above, $r_s$ is constant, if and only if, $r_s\equiv -1$, if and only if, $X$ is volume preserving.

We note that $Y$ produces aflow
$$\begin{cases}
Y^t:\mathbb{R}_s\times M \rightarrow \mathbb{R}_s\times M \\
Y^t(s,x)=(se^{2(1-r_s)},(X)^t(x))
\end{cases},$$
which is complete when defined on $(\mathbb{R}_s\times M)$. However, the definition of the Liouville form $\alpha=(1-s)\alpha_-+(1+s)\alpha_+$ does not fit our criteria for the model at $s\rightarrow \pm \infty$ of a LIS. The Liouville condition in this case can be written as
$$0<(\mathcal{L}_X \alpha)\wedge (\mathcal{L}_{\partial_s}\alpha)=(2r_u\alpha_u-2r_ss\alpha_s )\wedge (-2\alpha_s)=4r_u\alpha_s \wedge \alpha_u,$$
which is satisfied as long as $\alpha_u$ is expanding ($r_u>0$) and the above argument shows that it is complete. Therefore, we have a Liouville manifold structure on $(\mathbb{R}_s\times M,\alpha)$. By Theorem~\ref{linemb2}, we can embed the subset $([-1,1]_s\times M,\alpha)$ as a Liouville domain inside an exponential model $(\mathbb{R}_s\times M,L(\bar{\alpha}_-,\bar{\alpha}_+)_e$), which can then be extended to a struct Liouville equivalence between $(\mathbb{R}_s\times M,\alpha)$ and $(\mathbb{R}_s\times M,L(\bar{\alpha}_-,\bar{\alpha}_+)_e)$ as discussed in Section~\ref{5.1}. This gives a $C^{k*}$ normal strict Liouville equivalence between a LIS and its linearization, an equivalence which is at least $C^1$ in the Anosov case. The strong bundles of the linearization flow above is given by $E^n=\langle \partial_s \rangle$ (i.e. the strong normal foliation $\mathcal{F}^n$ is foliation by $s$-coordinate curves in $\mathbb{R}\times M$). Therefore, the $C^{k*}$ Liouville equivalence of a LIS to its linearization at the skeleton (which by Lemma~\ref{lioumaps}, $C^k$-maps the Liouville vector field to the one for the linearization flow) maps the strong normal foliation $\mathcal{F}^n$ to the $s$-coordinate curves. In particular, this means that the strong normal foliation of any LIS is $C^k$, whenever the weak invariant bundles of the supported flow are $C^k$. This gives $C^1$-regularity in the Anosov case. We summarize this in the following.

\begin{theorem}
Let $\alpha=L(\alpha_-,\alpha_+)_{(\lambda_-,\lambda_+)}$ supporting a non-singular partially hyperbolic $X$ with $C^k$ weak invariant bundle ($k\geq 1$ when $X$ Anosov). There is a $C^k$ strict Liouville equivalence $\psi:\mathbb{R}\times M \rightarrow \mathbb{R}\times M$ between $(\mathbb{R}\times M,\alpha)$ and its linearization. In particular, $\phi_*(Y)$ is the linearization of $Y$, the Liouville vector field of $\alpha$, at its skeleton.
\end{theorem}

\begin{corollary}
The strong normal Lagrangian bundle $E^n$ of any LIS at its skeleton is tangent to a 1-dimensional $C^k$ foliation $\mathcal{F}^n$ of $\mathbb{R}\times M$, contained in the weak normal Lagrangian foliation $\mathcal{F}^{wn}$, whenever the weak bundle of the supported non-singular partially hyperbolic flow are $C^k$. In particular, $\mathcal{F}^n$ is $C^1$, whenever $X$ is Anosov.
\end{corollary}

We have so far established the strict equivalence of a non-compact linear (which does not satisfy the conditions of a LIS as we defined in Definition~\ref{lis}) model with the LIS model, some care is needed to include them in the definition as we see below, which is one of the reasons we enforced certain conditions at infinity when defining the LIS model in Section~\ref{3.4}. In particular, consider the 1-form $\alpha=(1-s)\alpha_- +(1+s)\alpha_+$ defined on $\mathbb{R}_s\times M$ for the bi-contact form $(\alpha_-,\alpha_+)$. Considering that at $s\rightarrow \pm \infty$, we have $\alpha \simeq s(\alpha_+-\alpha_-)$.

\begin{claim}
The necessary condition for $\alpha$ to be Liouville on $\mathbb{R}_s\times M$ is $\ker{(\alpha_+-\alpha_-)}=E^{wu}$.
\end{claim}

\begin{proof}
The Liouville condition reads $$0\neq \mathcal{L}_X \alpha \wedge \mathcal{L}_{\partial_s} \alpha =[\mathcal{L}_X (\alpha_-+\alpha_+)+s(\mathcal{L}_X(\alpha_+-\alpha_-))]\wedge(\alpha_+-\alpha_-).$$
Therefore, the necessary condition to have this non-vanishing condition for all $s$ is $$(\mathcal{L}_X(\alpha_+-\alpha_-))\wedge(\alpha_+-\alpha_-)=0,$$
i.e. $\ker{(\alpha_+-\alpha_-)}$ is invariant under the flow of $X$. However, the dominated splitting is unique in the sense that the invariant plane field $\ker{(\alpha_+-\alpha_-)}$ needs to be equal to $E$ or $E^{wu}$, which with our orientation convention, it should be $E^{wu}$.
\end{proof}

We complete our observations in the following lemma.

\begin{lemma}
Let $(\alpha_-,\alpha_+)$ be a bi-contact form. The followings are equivalent:

(1) $\ker{(\alpha_+-\alpha_-)}=E^{wu}$;

(2) $(1-s)\alpha_-+(1+s)\alpha_+$ is a complete Liouville form on $(\mathbb{R}_s\times M)$;

(3) we can write $$\begin{cases}\alpha_+=\alpha_u-\alpha_s \\ \alpha_-=\alpha_u+h\alpha_s \end{cases},$$
with $r_u>0$, $r_s<r_u$ and $r_s+X\cdot{\ln{h}}<r_u$.
\end{lemma}

\begin{proof}
So far, we have shown $(2)\rightarrow (1)$ and $(1)\rightarrow (3)$ is obvious. It is enough to show $(3)\rightarrow (2)$. As above, the Liouville condition is 
$$0\neq \mathcal{L}_X(\alpha_-+\alpha_+)\wedge (\alpha_+-\alpha_-)=(2r_u\alpha_u+A\alpha_s)\wedge (-1-h)\alpha_s=2r_u(1+h)\alpha_s \wedge \alpha_u.$$ Therefore, we get the Liouville condition everywhere, regardless of what $h$ is (as long as $h>-1$). To see that we get an actual Liouville manifold, we need to check that the resulting Liouville vector field is complete.

To see this, recall that $Y=fX+g\partial_s$ is equivalent to $\alpha=f\mathcal{L}_X\alpha+g\mathcal{L}_{\partial_s}\alpha$. Note that in our case, we have
$$f(2r_u\alpha_u+(Ax+B)\alpha_s)=f\mathcal{L}_X\alpha=\alpha-g(\alpha_+-\alpha_-)=2\alpha_u+(g-s)(1+h)\alpha_s,$$
for some $A,B:M\rightarrow \mathbb{R}$, implying $f\equiv \frac{1}{r_u}$ and $g=\bar{A}s+\bar{B}$ is a linear function of $s$. Such vector field is again integrable, since the $g$ part is dominated by a linear flow. More specifically, the Liouville flow is given by
$$\begin{cases}
Y^t:\mathbb{R}\times M \rightarrow \mathbb{R}\times M \\
Y^t(s,x)=(\bar{s}(t),(\frac{1}{r_u}X)^t)\bar{A}
\end{cases},$$
where $\bar{s}:\mathbb{R}\times M\rightarrow \mathbb{R}$ is the unique solution to the ODE
$$\begin{cases}
\bar{s}(0)=s \\
\partial_t\cdot \bar{s}(x,t)= \bar{A}\circ (\frac{1}{r_u}X)^t (x)s+\bar{B}\circ (\frac{1}{r_u}X)^t (x)
\end{cases},$$
for which the solution exists for all $t\in\mathbb{R}$. 
\end{proof}

\section{Geometry of persistence near codimension 1 skeletons}\label{7}

The goal of this section is to study persistence in the dynamics and geometry of Liouville manifolds with 3 dimensional skeletons. The upshot of this section is that $C^1$-persistence, i.e. the persistence of the skeleton as a $C^1$ 3-dimensional submanifold, is characterized by the examples of Mitsumatsu's construction based on Anosov 3-flows, if we further assume the transversality of the skeleton and the Liouvile form. We still give structure results for the skeleton dynamics, when we drop the transversality assumption. We finally study the situation with weaker assumptions in the partially hyperbolic case. That corresponds to the absence of the rate condition of normal hyperbolicity in the repellence at the skeleton. We will see that the Liouville dynamics in the neighborhood of the skeleton can a priori be different than one induced from a LIS. 

In the following, let $(W,\alpha)$ be a Liouville 4-manifold with the Liouville vector field $Y$ with a $C^1$ embedded 3 dimensional Liouville skeleton $i:\Lambda\rightarrow W$.

\subsection{Normal hyperbolicity at the skeleton and the LIS model}\label{7.1}

Normal hyperbolicity at an invariant submanifold is equivalent to $C^1$-persistence, by the celebrated results of Hirsch-Pugh-Shub \cite{hps} and Mañe \cite{mane}. One can try to have similar characterization results in the Liouville category. Normal hyperbolicty gives a geometric interpretation of the more dynamical notion of persistence. Normal hyperbolicity for an oriented $C^1$ 3-dimensional Liouville skeleton $\Lambda \subset (W,\alpha)$ essentially means one invariant unstable bundle in the sense of Definition~ (while no stable bundle, i.e. we have $TW\simeq T\Lambda\oplus E^n$ for a 1-dimensional repelling invariant bundle $E^n$) and in particular, $W\simeq \mathbb{R}\times M$.

We start our discussion with an elementary observation.

\begin{lemma}\label{trans}
Point wise, we have $i^*\alpha\neq 0$, if and only if, $i^{-1}_* Y\pitchfork \ker{i^*d\alpha}$. If $\ker{\alpha} \pitchfork T\Lambda$ everywhere, then $\alpha \neq 0$.
\end{lemma}

\begin{proof}
Let $i:\Lambda\rightarrow W$ be such $C^1$ embedding. Non-degeneracy of $d\alpha$ implies $i^* \alpha\neq 0$ and has maximal rank on $T\Lambda$, i.e. has a 1-dimensional kernel $l$, where $l\subset \ker{i^*\alpha}$, since $Y\subset  T\Lambda|_\Lambda$ implies $di^*\alpha(l,i^{-1}_* Y)=i^*\alpha(l)=0$.

At $p$, we have $$i^*\alpha \neq 0 \Longleftrightarrow \text{for some}\ \ v\in T_p\Lambda, \ \ i^*d\alpha(i^{-1}_* Y,v)=i^*\alpha(v)\neq 0$$  $$\Longleftrightarrow Y\not\subset \ker di^*\alpha=l.$$

For the rest, notice that $\ker{\alpha}\pitchfork T\Lambda$ implies $\alpha$ is non-vanishing in a neighborhood of $\Lambda$, which is extended to non-vanishing on $W$ via scaling along $Y$.
\end{proof}

Now, suppose we have $\ker{\alpha}\pitchfork T\Lambda$ everywhere. 

Note that in particular, we are assuming $\alpha$ to be non-vanishing. Since $\mathcal{L}_Y \alpha=\alpha$, we know that $i^{-1}_* Y$ preserves $\ker{i^*\alpha}$ and $Y|_\Lambda$ defines an action on $T\Lambda / \ker{\alpha}$ with the constant expansion rate equal to $r=1$. We also note that as discussed in the proof of the above lemma, $\ker{i^*d\alpha}\subset \ker{i^*\alpha}$ and transverse to $i^{-1}_* Y$. In fact $i^{-1}_* Y$ preserves $\ker{i^*d\alpha}$ since $\mathcal{L}_Y d\alpha=d\alpha$. In order, to find an invariant sub-bundle of $T\Lambda$ transverse to $\ker{i^*\alpha}$, we need the domination of $\ker{i^*\alpha}$ by $T\Sigma / \ker{\alpha}$ (in order to use Lemma~\ref{strong}). That is, we need the expansion rate of $\ker{\alpha}$ under the action of $Y$ to be strictly less than 1. It turns out this condition is equivalent to the action of $Y$ at on the normal bundle of $\Lambda$, i.e. the action of $Y$ on $TW/T\Lambda |_{\Lambda}$, to be locally expanding. While such expansion is generic in the class vector fields $C^0$-repelling at the invariant set $\Lambda$, it is not clear if that is the case in the class of Liouville vector fields (see Question~\ref{qgenrep}).

More precisely, let $s$ be a $C^1$ normal coordinate near $\Lambda$ with $\Lambda=\{ s=0 \}$. Notice that we can do so, thanks to the fact that $\Lambda$ and $W$ are oriented. If we take the non-vanishing vector field $e\subset \ker{di^*\alpha}$ (notice that we can do this since $\ker{di^*\alpha}$ is co-oriented with $Y$ and $i^*\alpha$) such that on $TW$ at $\Lambda$ we have
$$d\alpha \wedge d\alpha=ds\wedge \hat{e}\wedge  d\alpha.$$

We have
$$\mathcal{L}_Y(d\alpha \wedge d\alpha)=\mathcal{L}_Y ds \wedge \hat{e} \wedge i^* d\alpha + ds \wedge \mathcal{L}_Y \hat{e} \wedge i^* d\alpha+ds \wedge \hat{e} \wedge i^* d\alpha,$$
which considering the fact that $div_Y(d\alpha \wedge d\alpha)=2$ implies
$$r_{ds}+r_e=1,$$
where $r_{ds}$ and $r_e$ are the expansion rates of $ds$ and $e$ respectively.

Now, since $Y$ is Liouville with skeleton $\Lambda$, $Y$ is $C^0$-repelling at $\Lambda$. But if we assume $Y$ to be $C^1$-repelling at $\Lambda$, we have $r_{ds}>0$ which then implies $r_e<1$. Hence, we achieve the domination of $\ker{i^*\alpha}$ by $T\Lambda /\ker{i^*\alpha}$ as desired. We are in the situation of Lemma~\ref{strong} now. To see this, let $E_1=\ker{i^*\alpha}$, $E_2=T\Lambda$ and $E_3=T\Lambda / \ker{i^*\alpha}$ with $T_1,T_2,T_3$ induced by the flow as before and we have
$$m(T_3 | E_{1x})=r_u=1>r_e=||T_1 | E_{3x} ||$$
for $x\in\Lambda$. The Lemma implies that there exists an invariant bundle $E^{u}$ transverse to $\ker{\alpha}$ in $\Lambda$, which has the expansion rate $r_u=1$. This means that $Y|_{\Lambda}$ admits a partially hyperbolic splitting $E^u \oplus \ker{\alpha}$, where $E^{wu}=E^{u}\oplus \langle Y|_{\Lambda} \rangle$.

The case $Y$ is normally hyperbolic at $\Sigma$ corresponds to when $1-r_e=||Y_{ds} || > ||Y|_{T\Sigma} ||=1$, i.e. when $r_e<0$. This is exactly when $Y$ is contracting on $E^s:=\ker{i^*\alpha}$ and $Y|_{\Lambda}$ is in fact Anosov.

Finally, $Y|_\Lambda$ is a volume preserving and Anosov, if and only if, $r_e=-r_u=-1$, or equivalently, $r_{ds}=2$. This is equivalent to the eigenvalue of the $\partial_s$ direction to be the squared of the eigenvalue of the $E^{u}$-direction, both bigger than 1 in magnitude, at any periodic orbit of $Y|_\Lambda$, a condition called {\em 2:1 resonance} (see \cite{reg}). Therefore, having 2:1 resonance at all periodic orbits of the $\Lambda$ is equivalent to $Y|_\Lambda$ (which is Anosov), being volume preserving at each orbit, which is equivalent to be volume preserving everywhere. This is a straighforward consequence of the Livsic theorem if we assume transitivity of $Y|_\Lambda$, but Massoni shows in Appendix B of \cite{massoni} that the transitivity condition can be dropped.

We have proven the following.

\begin{lemma}\label{anosovnh}
Let $\Lambda \subset (W,\alpha)$ be the 3 dimensional $C^1$ embedded skeleton of $Y$ such that $\ker{\alpha}\pitchfork T\Lambda$. Then, $Y|_\Lambda$ is Anosov, if and only if, $Y$ is normally hyperbolic at $\Lambda$ 
\end{lemma}

   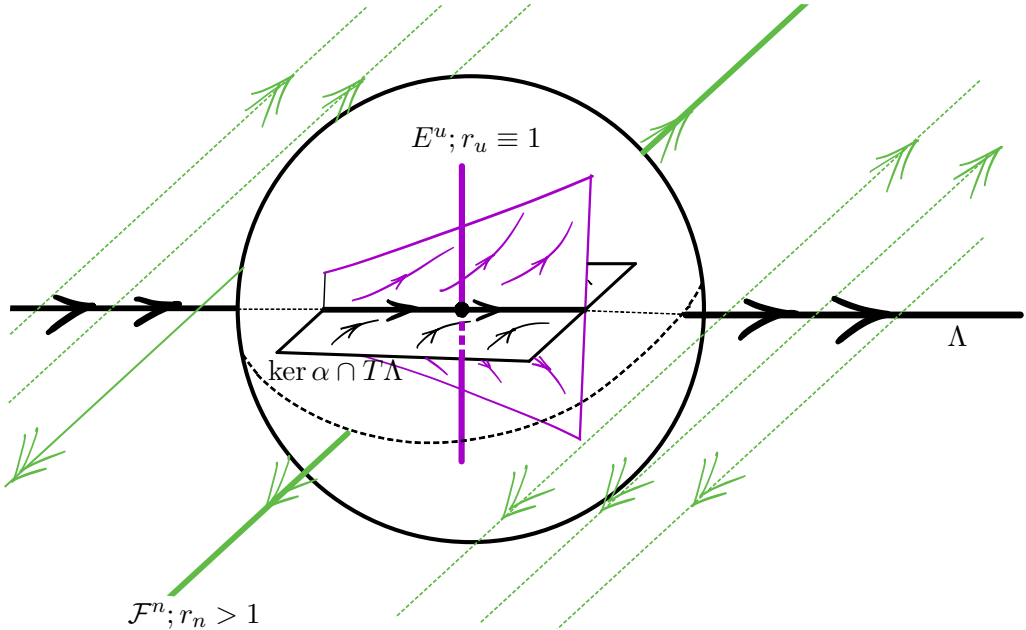
\begin{figure}[h]
\centering
\begin{overpic}[width=1\textwidth]{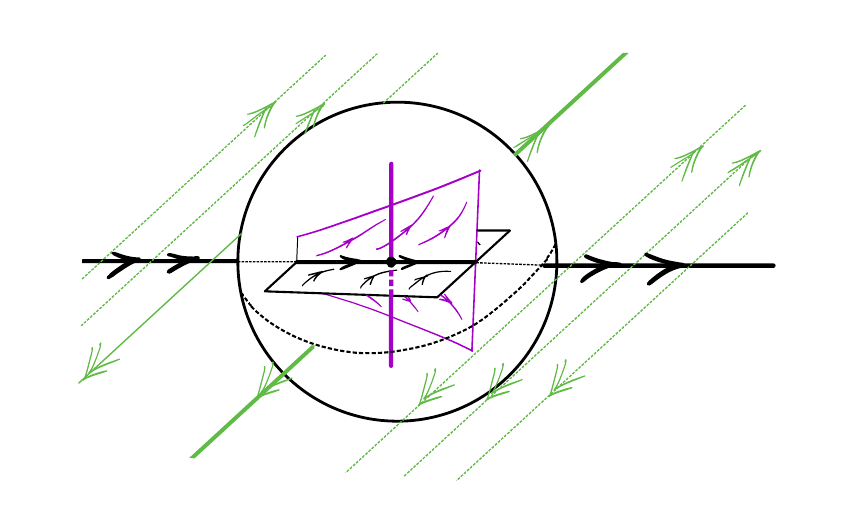}
    \put(90,20){$\mathcal{F}^n;r_n>1$}
     \put(143,112){$\ker{\alpha}\cap T\Lambda$}
     \put(197,200){$E^u;r_u\equiv 1$}
         \put(400,125){$\Lambda$}

  \end{overpic}

\caption{Geometry of normal hyperbolicity under the transversality assumption}
\end{figure}

In the Anosov case, we finally would like to establish $C^1$-equivalence of $(W,\alpha)$ and Liouville interpolations systems. The situation for partially hyperbolic flows is more flexible and leaves the possibility of exotic dynamics near the skeleton (see Question~\ref{qexotic}). To do so, we first prove the following lemma which gives a more general description of the interaction between normal hyperbolicty at $\Lambda$ and transversality of $\ker{\alpha}$, which will be useful later on.

\begin{lemma}
Let $\Lambda$ be the normally hyperbolic $C^1$ 3D skeleton of $(W,\alpha)$ with the 1 dimensional normal invariant bundle $E^n$ and assume $\alpha$ to be non-vanishing. Then,

(a) $\Lambda_L:=\{ E^n\subset \ker{\alpha} \}$ and $\Lambda_T:=\{ T\Lambda=\ker{\alpha}\}$ are compact invariant subsets of $\Lambda$;

(b) for any $p\in \Lambda \backslash (\Lambda_L \cup \Lambda_T)$, the forward and backward limit sets are contained in $\Lambda_L$ and $\Lambda_T$, respectively;

(c) $\Lambda_L\neq \emptyset$;

(d) when $\Lambda_T= \emptyset$, we have $\Lambda_L=\Lambda$;

(e) $\Lambda_T$ is an attractor for $Y|_\Lambda$;

(f) $\Lambda_L$ is a hyperbolic invariant set for $Y|_\Lambda$.
\end{lemma}

\begin{proof}
(a) follows directly from the fact that $T\Lambda$, $\ker{\alpha}$ and $\langle Y,E^n\rangle$ are invariant under the flow of $Y$.

To see (b), recall the folklore fact, also put in details in \cite{folk}, that assuming transitivite partially hyperbolic diffeomorphism or flow, any invariant bundle can be split into invariant sub-bundles included in the componenets of the dominted splitting. This means that if $Y|_\Lambda$ is transitive, we have $\Lambda=\Lambda_L$ or $\Lambda=\Lambda_T$. But $\Lambda=\Lambda_T$ is impossible, since it would imply $d\alpha|_{T\Lambda}=0$ which contradicts $d\alpha$ being symplectic. Therefore, we have $\Lambda=\Lambda_L$.

But in the non-transitive case, a bit more care is needed. Suppose $p\notin \Lambda_T \cup \Lambda_L$, i.e. at $p$, $\ker{\alpha}$ makes some non-zero angle with $T\Lambda$. Without loss of generality, assume $L$ orthogonal to $T\Lambda$ and $0<\theta<\frac{\pi}{2}$ (see Figure 10). The domination of $T\Lambda$ by $L$, implies that $\theta_t$, the angle $Y^t_*(\ker{\alpha})$ makes with $T\Lambda$,  is increasing with $t$ and therefore, as $t\rightarrow \infty$, i.e. moving in the forward direction of the flow, $\ker{\alpha}|_\Lambda$ gets arbitrarily close to $E^n$ and consequently, $Y^t(p)$ gets arbitrarily close to $\Lambda_L=\{ \cos{\theta}=0 \}$. Similarly, as $t\rightarrow -\infty$,  i.e. moving in the backward direction of the flow $Y^t(p)$ gets arbitrary close to $\Lambda_T=\{ \sin{\theta}=0 \}$, proving the claim.

   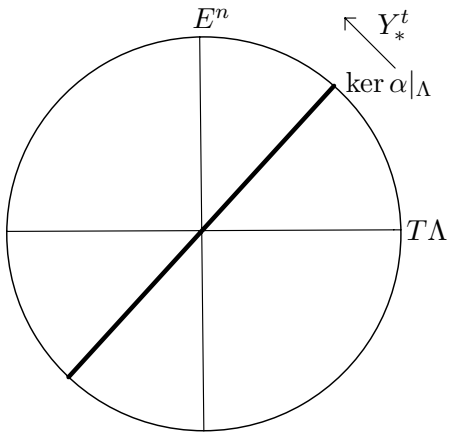
\begin{figure}[h]
\centering
\begin{overpic}[width=0.38\textwidth]{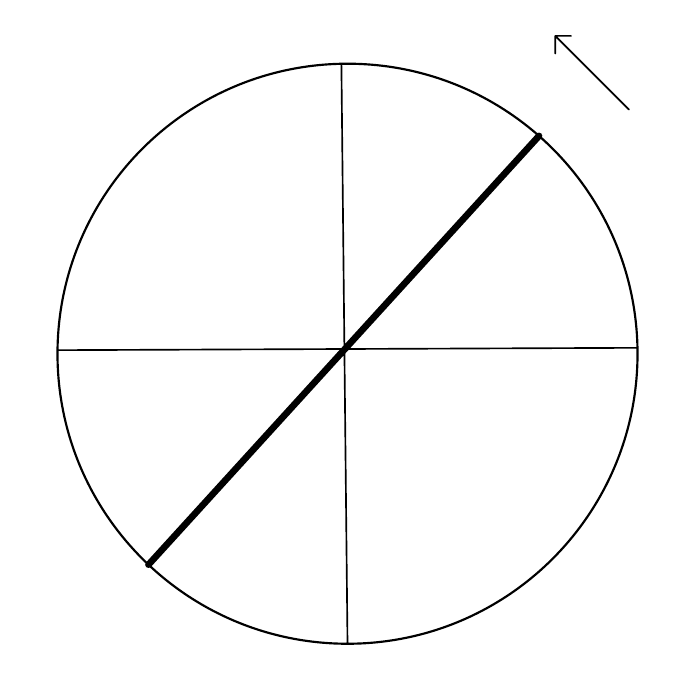}
    \put(85,162){$E^n$}
    \put(143,138){$\ker{\alpha}|_\Lambda $}
        \put(155,160){$Y_*^t$}
        \put(166,82){$T\Lambda$}
  \end{overpic}

\caption{From normal hyperbolicity to rotation of $\ker{\alpha}$ under the skeleton dynamics}
\end{figure}

(c) follows from (b) by noticing that if $\Lambda_L=\emptyset$, then we would have $\Lambda=\Lambda_T$ which yields $d\alpha|_{T\Lambda}=0$, a contradiction with $d\alpha$ being symplectic.

(d) follows from (b).

To see (e), note that on $\Lambda_T$, $\alpha$ does not vanish on $L$ and the expansion rate of $L$ is $r_L=1$ by the equation $\mathcal{L}_Y\alpha=\alpha$. Normal hyperbolicity implies that with respect to some norm on $T\Lambda$, we have $||Y^t_*(e)||<e^t ||e||$ for any $e\in T\Lambda|_{\Lambda_T}$ and $t\in\mathbb{R}$. On the other hand, $d\alpha|_{T\Lambda}$ is a transverse 2-form for $Y|_\Lambda$ at $\Lambda_T$ (see Lemma~\ref{trans}) and $\mathcal{L}_Yd\alpha=d\alpha$ implies the sum of the expansion rates of vectors in $T\Lambda/\langle Y \rangle |_{\Lambda_T}$ equals 1.
 Normal hyperbolicity implies that each of such expansions should be strictly less than 1, which yields both such expansions should be positive, hence, the claim about $\Lambda_T$ being an attractor.

(f) Since at $\Lambda_L$, $i^*\alpha$ is non-vanishing, our argument in Lemma~\ref{anosovnh} on the hyperbolicity of the entire $\Lambda$ can be extended to $\Lambda_L$ in the general setting.
\end{proof}

The above lemma provides a description for the skeleton dynamics of a 3 dimensional Liouville skeleton which is non-vanishing and normally repelling and as we saw in Lemma~\ref{anosovnh}, adding the condition of transversality between such skeleton and $\ker{\alpha}$ implies the skeleton dynamics is Anosov, similar to the LIS construction. We next want to show Liouville geometry is in fact $C^1$-strictly Liouville equivalent to the Mitsumatsu's examples. 

We have so far shown that assuming normal hyperbolicity of the Liouville vector field $Y$ at the 3 dimensional skeleton $\Lambda$ with $\ker{\alpha}\pitchfork T\Lambda$ implies Anosovity of $Y|_{\Lambda}$. Let $E^u$ and $E^s$ be the unstable and stable bundles of $Y|_{\Lambda}$ and $E^n$ the 1-dimensional repelling invariant bundle normal to $T\Lambda$. Therefore, $\Lambda$ is in fact a hyperbolic invariant set for $Y$ with the a 2 dimensional unstable bundle $\bar{E^u}=E^u\oplus E^n$ and 1 dimensional stable bundle $\bar{E^s}=E^s$. Since the corresponding weak bundle $\bar{E^{wu}}$ is codimension 1 and hence, Hasselblat's Theorem~\ref{hassbunch} implies that $\bar{E^{wu}}$ is in fact $C^{1+}$. On the other hand, we are assuming $\alpha$ (and hence, $\ker{\alpha}$) to be $C^\infty$. Therefore, $\langle E^n,X \rangle=\bar{E^{wu}} \cap \ker{\alpha}$ to be $C^1$. 

\begin{remark}
Note that in the above argument, the Liouville condition is resulting in the regularity of the invariant plane field $E^{wn}$ outperforming the expectations from the standard regularity theory of normally hyperbolic flows. This is in fact important for what follows further, as we will use this extra degree of regularity to establish $C^1$-conjugacy with the construction of Mitsumatsu. More specifically, we have the dominated splitting  $TM/\langle Y\rangle \simeq T\Lambda \oplus E^{wn}$. It is easy to compute the unstable bunching constant for the regularity of $E^{wn}$ and observe that $B^u(Y|_{\Lambda})<1$. This means that our expectation of the regularity of $E^{wn}$ based on the rate conditions only guarantees Hölder continuity of such invariant bundle  (see Remark~\ref{bunchex}). Therefore, using the Liouville condition in the above argument (when assuming that $\ker{\alpha}$ is preserved under the flow of $Y$) is indispensible to establish the $C^1$-regularity of $E^{wn}$, a fact which will be used in the reconstruction of the LIS model from the $C^1$-persistence of the skeleton model.
\end{remark}

We then need to observe the weak invariant bundles are strict exact Lagrangians. Choose a $C^1$ normal coordinate $s$ such that $E^{wn}:=E^n\oplus \langle Y|_\Lambda \rangle= \langle \partial_s, Y|_\Lambda \rangle$, $\partial_s \pitchfork Y$ (this can be arranged by taking $\partial_s$ to be a $C^0$-approximation of $E^n\subset E^{wn}$) and $\Lambda=\{ s=0 \}$. Solving the ODE
$$\begin{cases}
\partial_s \cdot [\alpha(\partial_s)]=(\mathcal{L}_{\partial_s}\alpha)(\partial_s)=d\alpha(\partial_s,\partial_s)=0 |_{s=0}\\
\alpha(\partial_s)=0|_{s=0}
\end{cases},$$
yields $\alpha|_{E^{wn}}=0$. 
Therefore, $Y$ is tangent to and preserves a strict exact Lagrangian foliation tangent to $E^{wn}$, which we denote by $\mathcal{F}^{wn}$. Note that the Normal Hyperbolicity Theorem~\ref{normalh} only gives the existence of weak invariant manifold tangent to $E^n\oplus T\Lambda$ which in our case is the entire $W=\mathbb{R}\times M$ and hence, useless. To note the invariant Lagrangian manifold $\mathcal{F}^{wn}$, view $E^{wn}:=\bar{E}^{wu}\cap \ker{\alpha}$ is a $C^1$ plane field invariant under the flow of $Y$ and containing $Y$. Therefore, $E^{wn}$ is an integrable and tangent to a strict Lagrangian foliation $\mathcal{F}^{wn}$. 

We then need the following lemma next.

\begin{lemma}
$\Lambda$ has a tubular neighborhood $N(\Lambda)$ strictly Liouville equivalent to a compact LIS.
\end{lemma}

\begin{proof}
In a neighborhood of $\Lambda$, choose a $C^1$ vector field $\partial_s\subset E^{wn}$ transverse to $\Lambda$ and use this to define a $C^1$ product structure $N(\Lambda)\simeq [-\epsilon,\epsilon]_s\times \Lambda$ such that $\Lambda=\{ s=0\}$ and note that if we let $X$ be the natural extension of $Y|_\Lambda$ to $[-\epsilon,\epsilon]_s\times \Lambda$, we have $\langle \partial_s ,X\rangle$ is a trivial Lagrangian plane field. Also, define $\alpha_u:=\frac{1}{2}\alpha|_\Lambda$ (note that $Y|_\Lambda$ is synchonized with respect to such $\alpha_u$) and some other 1-form $\alpha_s$ with $\ker{\alpha_s}=E^{wu}$ such that $\alpha_\pm:=\alpha_u\mp\alpha_s$ is a $\pm$ contact structure. Finally, use the product structure to define these 1-forms on $[-\epsilon,\epsilon]_s\times \Lambda$. Since $\alpha$ vanishes on $E^{wn}=\langle \partial_s ,X \rangle$, we can write
$$\alpha=\lambda_-\alpha_-+\lambda_+\alpha_+,$$
for some functions $\lambda_\pm:N(\Lambda)\rightarrow \mathbb{R}$, since $(\alpha_-,\alpha_+)$ provides a basis for $Annih^1(E^{wn})$. At $\Lambda=\{ s=0\}$, we have $\alpha|_\Lambda=2\alpha_u$, $\lambda_-+\lambda_+=2$ and $\lambda_--\lambda_+=0$, i.e. $\lambda_-(0,.)=\lambda_+(0,.)=1$. Considering this, we can write the Liouville condition at $\Lambda$ as 
$$0<\frac{1}{2}\iota_X\iota_{\partial_s}( d\alpha \wedge d\alpha)=(\mathcal{L}_X\alpha)\wedge(\mathcal{L}_{\partial_s} \alpha) |_{s=0}=[\mathcal{L}_X(\alpha_-+\alpha_+)]\wedge [(\partial_s \cdot \lambda_-) \alpha_- +(\partial_s \cdot \lambda_+) \alpha_+]$$
$$=2\alpha_u \wedge[(\partial_s \cdot \lambda_-) \alpha_- +(\partial_s \cdot \lambda_+) \alpha_+]=2\alpha_u \wedge [\partial_s \cdot \lambda_- -\partial_s \cdot \lambda_+]\alpha_s.$$
Therefore, the Liouville condition at $\Lambda$ yields $$\partial_s \cdot \lambda_+ -\partial_s \cdot \lambda_- |_{\Lambda}>0 \Rightarrow \partial_s\cdot \frac{\lambda_+}{\lambda_-}|_{\Lambda}>0,$$
where the conclusion is thanks to the fact that $\lambda_-=\lambda_+=1|_{\Lambda}$. Therefore, in a neighborhood of $\Lambda$ (possibly smaller than $N(\Lambda)\simeq [-\epsilon,\epsilon]_s\times \Lambda$), we have the conditions $\lambda_\pm>0$ and $\partial_s \cdot \ln{\frac{\lambda_+}{\lambda_-}}>0$. Moreover, the restriction of $\alpha$ to the boundary of such tubular neighborhood is a contact from. Hence, we have all the conditions of a compact LIS.
\end{proof}

Given the above lemma, we can then use Corollary~\ref{linemb} to embed a neighborhood of the skeleton $\Lambda$ into any non-compact LIS. Therefore, there is a strict Liouville equivalence between $(W,\alpha)$, as the unique completion of such neighborhood, and any non-compact LIS. Hence, we have established the following in this section.

\begin{theorem}\label{pers}
Suppose $(W^4,\alpha)$ is Liouville manifold with an oriented $C^1$-persistent 3-dimensional skeleton $\Lambda$ and $\alpha$ is nowhere vanishing Let $E^n$ be the invariant unstable normal bundle at $\Lambda$. Then,

(a) $Y|_\Lambda$ is an Axiom A flow, where $\Lambda_T=\{ \ker{\alpha}=T\Lambda \}$ is collection of a finite number of repelling periodic orbits of $Y$, and for some tubular neighborhood of $\Lambda_T$, $\Lambda/N(\Lambda_T)$ is a hyperbolic plug whose core is given by $\Lambda_L=\{ E^n\subset \ker{\alpha}\}$.

(b)  If $T\Lambda \pitchfork \ker{\alpha}$, $Y|_\Lambda$ is a synchronized Anosov vector field and $(W^4,\alpha)$ is $C^1$-strictly Liouville equivalent to a alaiouville form induced from a LIS supporting $Y|_\Lambda$. 
\end{theorem}

As a result, under the skeleton transversality assumption, i.e. (b) in the above theorem, the 1-to-1 correspondence of Theorem~\ref{1to1} is reduced to one between (positive reparametrization classes of) Anosov 3-flows and $C^1$-persistent Liouville skeletons of codimension 1. In other words, we can drop the somewhat {\em extrinsic} condition of {\em being induced from a LIS} in the statement of Theorem~\ref{1to1} and rely on the intrinsic condition of $C^1$-persistence. We in fact construct an appropriate $C^1$ coordinate systems in which the given Liouville form is induced from a LIS.

\begin{corollary}
There exists the following 1-to-1 correspondence
$$\bigg\{\substack{\text{Positive reparametrization classes of} \\  \text{Anosov flows} \\  \text{up to $C^\infty$-conjugacy}} \bigg\} \mbox{\Large$\overset{\text{1-to-1}}{\longleftrightarrow}$}
\bigg\{\substack{\text{Liouville forms on $\mathbb{R}\times M$ with $C^1$-persistent} \\  \text{3-dimensional skeleton $\Lambda$ with $\ker{\alpha}\pitchfork T\Lambda$} \\  \text{up to strict Liouville equivalence}} \bigg\}.$$
\end{corollary}

\begin{remark}
Here, we see a property of Liouville vector fields which is not true for a general vector field with positive divergence. $C^1$-persistence at the Liouville skeleton $\Lambda\subset (W,\alpha)$ results in $C^0$ rigidity of $Y|_\Lambda$. This is not true in general. Since, any vector field $X$ on a 3-manifold $M$ can be extended to the thickening $(-\epsilon,\epsilon)\times M$, by adding a sufficiently repelling normal bundle and therefore realizing $(M,X)$ as a $C^1$-persistent 3-dimensional {\em skeleton dynamics}. The above result shows that this is not possible in general in the category of Liouville flows. In particular, Theorem~\ref{pers} implies the existence of a non empty hyperbolic invariant set $\Lambda_L\subseteq \Lambda$.
\end{remark}


\subsection{Partially hyperbolic dynamics on the skeleton}\label{7.2}

In this section, we want to take a look at $C^1$-embedded Liouville skeletons in the absence of normal hyperbolicity. Without normal hyperbolicity, we are unable to derive structure results near the invariant set $\{ \ker{\alpha}=T\Lambda \}$ like in Theorem~\ref{pers}. Therefore, we restrict our attention to when $\ker{\alpha}\pitchfork T\Lambda$.

First, recall this lemma from the discussion in the previous section.
\begin{lemma}\label{almosth}
Let $\Lambda \subset (W,\alpha)$ be the 3 dimensional $C^1$ embedded skeleton of $Y$ such that $\ker{\alpha}\pitchfork T\Lambda$. Also assume that at $\Lambda$, $Y$ expands $TW/T\Lambda$. Then, $Y|_\Lambda$ is partially hyperbolic.
\end{lemma}

We don't know if this condition is generic for Liouville vector fields, while it is generic among general vector fields (see Question~\ref{qgenrep}). This means that under such condition, transversality is enough to establish a weaker form of hyperbolicity for the skeleton dynamics.

Of course, we have already seen that any non-singular partially hyperbolic flows can in fact be realized as the Liouville dynamics on the skeletons in the LIS examples. However, the skeleton in those examples fail to be a priori $C^1$ (while it is still Hölder continuous). We observe that the Liouville dynamics admits a strong (repelling) invariant line bundle transverse to the skeleton (i.e. the Liouville flow is conjugate to a linearization). This at least gives $C^r$-persistence of the skeleton as a topological (or Hölder continuous) manifold via the classical graph transformation ideas. This argument for $C^0$-persistence works without using its linearization and only using the {\em $C^k$-section theorem} (see Section~\ref{8}). However, in general understanding topological persistence can be more complicated in the absence of the rate condition of normal hyperbolicity \cite{floer1,c0}.

However, we can argue that when not too far from the rate condition in normal hyperbolicity, the conjugacy to the linearization can be established. In particular, suppose $Y$ expands $TW/T\Lambda$ with $r_n>\frac{1}{2}$. The discussion of the previous section implies that $r_s=r_{\ker{\alpha}\cap T\Lambda}=1-r_n<\frac{1}{2}<r_n$. Therefore, $\ker{\alpha}$ is dominated by the transverse direction, which is enough to find an invariant line bundle now. More specifically, use Lemma~\ref{strong} by letting $E_1=\ker{d\alpha|_{T\Lambda}}$, $E_2=\ker{\alpha}$ and $E_3=E_2 / E_1$, we have
$$e^{\int_0^1 r_s\circ Y^\tau d\tau}=||Y^*|_{E_1}||<m(Y^*|_{E_3} )=e^{\int_0^1 r_n\circ Y^\tau d\tau},$$
giving the existence of an invariant complement for $E^1$ inside $\ker{\alpha}$, which we denote by $E^n$. Since $Y$ preserves $\ker{\alpha}$, this means that $E^n$ is in fact an invariant bundle under $Y$.

Notice that now $\Lambda$ is an invariant set with a dominated splitting
$$\begin{cases}
TW|_\Lambda \simeq E\oplus F \\
E=\langle Y \rangle \oplus \ker{(d\alpha|_{T\Lambda})} \text{ and } F=E^n \oplus E^u
\end{cases},$$
where $E^u$ is the unstable bundle of the partially hyperbolic flow $Y|_{\Lambda}$. Note that in the above, we have $m(Y|_F)>e^{\frac{1}{2}}$ and $||Y|_E|| <e^{\frac{1}{2}}$.

\begin{theorem}\label{nearanosov}
In the conditions of the Lemma~\ref{almosth}, suppose further that $Y$ expands $TW/T\Lambda$ with a rate $r_n>\frac{1}{2}$. Then, there exists an invariant bundle $E^n\subset \ker{\alpha}$ such that $E^n\pitchfork T\Lambda$. 
\end{theorem}

In particular, under the rate condition of Theorem~\ref{nearanosov}, the Liouville dynamics is $C^0$-conjugate to its linearization at the skeleton. It is not clear if other kinds of {\em exotic} Liouville dynamics are allowed near a Liouville skeleton, if we further drop such rate condition (see Question~\ref{qexotic}). One would like to extend the classification results of this section for $C^1$-persistent Liouville skeletons to such lower regularity settings related to non-Anosov partially hyperbolic flows (see Question~\ref{qc0}).


\section{Regularity theoretic interpretation}\label{8}

Part (a) of Theorem~\ref{dynrig} establishes a Liouville geometric approach toward the regularity theory of the weak invariant bundles of Anosov flows (or more generally, the dominated bundle of non-singular partially hyperbolic flows). In particular, the Liouville skeleton of a LIS supportin a non-singular partially hyperbolic $X$ is a graph of a $C^k$ function $\Lambda_s:M\rightarrow \mathbb{R}$, if and only if, the dominated bundle $E$ is $C^k$. Therefore, we can study the regularity theory of such dominated bundle in terms of this Liouville skeleton. We have seen so far, that the normal expansion at any point on the skeleton is $r=1-r_s$, where $r_s$ is the expansion rate of the dominated bundle. We can use the $C^r$-section theorem then in order to determine the regularity of such graph $\Lambda_s$. More specifically, we can demonstrate a fiberwise contraction of a tubular neighborhood $N(\Lambda_s)\rightarrow \Lambda_s$, via the strong normal bundle, or equivalently, the linearization of the flow at the skeleton. Point-wise at any $p\in \Lambda_s$, we have ($Y^*$ being the time-1 action) $$||(Y|_{\Lambda_s})^*|_{E^{wn}}||=e^{\int_0^1(1-r_s)\circ Y^t(p) dt} \text{ and } ||(Y|_{\Lambda_s})^*|_{T\Lambda_s}||=e,$$ since $r_u\equiv 1$ gives the maximum expansion on $T\Lambda_s$. Now, if
$$B_s(Y|_{\Lambda_s})=\inf_{p\in M}[1-\sup_{t>0}\frac{1}{t}\int_0^t r_s\circ X^\tau (p)d\tau]$$
is the bunching constant of the dominated weak bundle (see Remark~\ref{bunchex}), then we have
$$||(Y|_{\Lambda_s})^*|_{E^{wn}}||>||(Y|_{\Lambda_s})^*|_{T\Lambda_s}||^{B_s(Y|_{\Lambda_s})}.$$
Therefore, the $C^r$-section theorem (see Theorem~3.5 of \cite{hps}) implies that, the invariant manifold $\Lambda_s$ has the regularity $C^{B_s(Y)-\epsilon}$ for arbitrary $\epsilon>0$. This gives a new proof for the result of Hasselblatt on $C^{1+}$-regularity of weak invariant bundles for Anosov 3-flows \cite{reg} and in fact, generalizes such result to the weak dominated bundles of non-singular partially hyperbolic flows. We can furthermore use the persistence part of the $C^r$-section theorem to show that such regularity is well behaved under homotopy.

Note that the Anosovity case corresponds to $B_s>1$ and in that case, one could also use the normal hyperbolicity theorem (see Theorem~\ref{normalh}) to make the same conclusion about the regularity of $\Lambda_s$.

\begin{theorem}\label{regex}
Let ${\Pi}^k(M)$ be the space of $C^k$ plane fields and $\mathcal{PHF}(M;B_s>k)$ be the space of non-singular partially hyperbolic vector fields like $X$ with the splitting  $TM/\langle X\rangle \simeq E\oplus E^{u}$, and the bunching constant satisfying $B_s>k$. Then, the map defined by $$\begin{cases}\mathcal{D}:\mathcal{PHF}(M;B_s>k)\rightarrow \Pi^k(M) \\
\mathcal{D}(X):=E \end{cases},$$
sending a non-singular partially hyperbolic flow to its weak dominated bundle as a plane field,
is well-defined and continuous. In particular, the stable weak bundle $E^s$ $C^1$-varies as one deforms $X$ through Anosov vector fields. 
\end{theorem}

Here, the map $\mathcal{D}$ being well-defined refers to the fact that the bunching condition guarantees that the weak dominated bundle is $C^k$ and the continuity is equivalent to $C^k$-dependence of such plane field to the variations of the generating vector field.

\begin{remark}
We would like to point out that our argument proves a refinement of Theorem~\ref{regex} in terms of the regularity of the vector field $X$. More specifically, the map $\mathcal{D}$ can be extended to a continuous map on the space of $C^2$ non-singular partially hyperbolic flows, i.e. a map of the form $$\mathcal{D}:\mathcal{PHF}^2(M;B_s>k)\rightarrow \Pi^k(M).$$
\end{remark}

\begin{proof}
As discussed above, the map being well-defined is a non-trivial consequence of $C^r$-section theorem and to see to the continuity, 
suppose a $C^{k+1}$-family of generating vector fields $X$ is given with $B_s>k$. There exists a $C^{k+1}$-family of supporting LIS s $(\alpha_-,\alpha_+)_{(\lambda_-,\lambda_+)}$ (which without loss of generality by Theorem~\ref{1to1}, we can assume to be a family of exponential Liouville pairs) inducing a $C^k$-family of Liouville vector fields and the persistent part of $C^r$-section theorem implies that the skeleton is $C^k$-persistent under such $C^k$-deformation Liouville vector fields (see Theorem~3.5 of \cite{hps}).

Note that when the deformation is through Anosov flows, $k>1$ and any $C^2$-deformation of $X$ results in the $C^1$-deformation of the weak stable bundle. In this case, one can conclude $C^1$-regularity and persistence, thanks to the standard theory of normal hyperbolicity (see Theorem~\ref{normalh}). The same can be said for the weak unstable bundle by considering $-X$.

\end{proof}

It is natural to ask about the optimality of the regularity results for the invariant bundles, as is investigated in the classical literature \cite{reg,reg2}. Our construction implies that such optimality questions can be formulated in terms of the {\em Liouville persistence} of skeleton, i.e. the persistence of the skeleton as we deform the Liouville flow through Liouville flows (see Question~\ref{qlioupers}).

\section{Remarks on the related geometric objects}\label{9}

The goal of this section is to make some remarks and observations about other related symplectic geometric objects inside the Liouville manifolds obtained from the LIS construction. That is, the Lagrangians and Hamiltonian vector fields. We gather our main observations in the following theorem and the proof follows the discussions in the remainder of this section.

\begin{theorem}
Suppose $(\alpha_-,\alpha_+){(\lambda_-,\lambda_+)}\in \mathcal{LIS}(M)$ supports a non-singular partially hyperbolic flow $X$.

(1) the weak normal foliation $\mathcal{F}^{wn}$ $C^\infty$-depends on $X$;

(2) the Liouville skeleton $\Lambda_s$ is foliated by a $C^1$ strict exact Lagrangian foliation. when $X$ is Anosov;

(3) the Reeb flows for any supporting $(\alpha_-,\alpha_+)$ can be realized as the Hamiltonian flows on a pair of energy hypersurfaces inside $(\mathbb{R}\times M,L(\alpha_-,\alpha_+)_{(\lambda_-,\lambda_+)})$.
\end{theorem}

\subsection{Underlying Lagrangian foliations}\label{9.1}

There are Lagrangian foliations naturally defined in the LIS construction. We would like to show that such foliations behave nicely under the deformation of the flow described in Section~\ref{4} and \ref{5}. 

First, we always have the normal Lagrangian bundle $E^{wn}=\langle X,\partial_s \rangle $ which is strictly exact, i.e. $\alpha|_{E^{wn}}\equiv 0$. These Lagrangian plane fields (and therefore the Lagrangian foliations $\mathcal{F}^{wn}$ induced) are $C^k$ deformed under $C^k$-deformations of $X$. Furthermore, if a $C^k$-deformation $X_\tau$ is through non-singular partially hyperbolic vector fields, we have a $C^k$ family of Lagrangian integrable plane fields $E^{wn}_\tau$ tangent to the Lagrangian foliations $\mathcal{F}^{wn}_\tau$. Note that the notion of being an exact Lagrangian foliation is invariant for a sub-manifold under Liouville homotopy. That is since by Lemma~\ref{moserclassic}, we have $\alpha_0+dh_\tau=\psi^{\tau*}\alpha_\tau$ after applying an isotopy $\psi^\tau$. Therefore, $\mathcal{F}_0^{wn}$ can be assumed, after an isotopy, to stay an exact Lagrangian foliation during a Liouville homotopy. Also, note that for all $\tau$, the leaf space of $\mathcal{F}^{wn}_\tau$ and the orbit space of $X_\tau$ is the same. This means that all invariants of Liouville pairs which are invariants up to Liouville homotopy stay invariant. As motivated in \cite{clmm}, one can try to relate the Fukaya sub-categories of $\mathcal{F}_\tau^{wn}$ generated by periodic orbits of $X_\tau$ (corresponding to exact Lagrangian annuli). Note that the bi-contact DA deformation gives an example where the two $\mathcal{F}^{wn}_\tau$ in the same family can be different as foliations, since their leaf spaces, which are equivalent to the orbit spaces of the underlying flows, are different, before and after the DA deformation). While it is natural ask whether such Fukaya sub-category split-generates the wrapped Fukaya category of $(\mathbb{R}\times M,\alpha)$ \cite{clmm}, we can also ask to what degree such sub-category remembers the orbit structure of the flow.

Another family of Lagrangians, which naturally shows up in the LIS setting, is the foliation of the Liouville skeleton by a 2-dimensional $C^k$ strictly exact Lagrangian foliation. While one might still be able to make sense of this in the low regularity setting of non-Anosov partially hyperbolic flow, we restrict our attention to the Anosov case, where the skeleton is the graph of a $C^{1+}$ function $\Lambda_s:M\rightarrow \mathbb{R}$. In this case, the flow $Y|_{\Lambda_s}$ is Anosov with the strong stable bundle $E^{s}=\ker{(d\alpha|_{T\Lambda_s} )}\subset \ker{(\alpha|_{T\Lambda})}$. Therefore, the weak stable foliation of the skeleton $\mathcal{F}^{ws}|_{\Lambda_s}$ is a foliation by strict Lagrangians (note $\alpha|_{T\mathcal{F}^{ws}|_{\Lambda_s}}=0$). With some more care, one might be able to make sense of this observation as the Anosovity condition is relaxed to partial hyperbolicity and the skeleton is only Hölder continuous. Notice that even in the case of non-singular partially hyperbolic flows, there is a $C^k$-conjugacy, for some $k>0$, with a linearized flow (see Section~\ref{6}) and therefore, the weak stable foliation of the skeleton in the low regularity linearized model, which is truly an exact Lagrangian foliation, is mapped, via a low regularity conjugacy, to the stable foliation of the low regularity skeleton of such LIS s.

\subsection{Hamiltonians}\label{9.2}

Consider the 3-dimensional $C^1$ Liouville skeleton $\Lambda_s$ induced from a LIS supporting an Anosov vector field, which is foliated by the weak stable foliation (the second exact Lagrangian foliation discussed above). Note that a Hamiltonian whose energy level is this $C^1$ skeleton $\Lambda_s$ satisfies $X_H\subset \ker{\alpha|_{T\Lambda}}=E^s$, where by $E^s$, we refer to the strong stable bundle of $Y|_{\Lambda_s}$. 

Naturally, the Reeb flows of the supporting bi-contact pairs can be manifested as Hamiltonian of a level surface on the appropriate side of the skeleton $\Lambda_s$, i.e. the Reeb flows of the positive contact form in the positive side the skeleton, i.e. $\{ s>\Lambda_s\}$, and the Reeb flows of the negative contact form in the negative side of the skeleton, i.e. $\{ s<\Lambda_s\}$. This is simply thanks to the standard fact that the two sides of the Liouvillr skeleton are determined as Liouville cylinder diffeomorphic to a symplectization (see Section~\ref{3.1}). Therefore, any contact form can be realized as a Hamiltonian on some Hamiltonian energy level. The realization of supporting contact forms as Hamiltonian flows restricted to sections of $\pi:\mathbb{R}\times M \rightarrow M$ can be generalized to non-singular partially hyperbolic flows, as it does not require $C^1$-regularity of the weak dominated bundle.

\section{Questions}\label{10}
In this section, we gather some of the questions which were motivated throughout this paper. Many of these questions revolve around establishing a general dynamical framework for Liouville dynamics.

A general line of questions is motivated by distinguishing a Liouville flow from an arbitrary flow of positive divergence, up to positive reparametrization. This is paralleled with the historical motivations of the subject as discussed in the introduction. Recall that Mitsumatsu's construction \cite{mitsumatsu} was introduced to show that the world of Liouville geometry is wider than the study of {\em Stein structures}. More specifically, in Liouville geometry, disconnected boundary can happen while that is not possible in the Stein case. Along the same lines, we can ask how different a general flow of positive divergence is from a Liouville one. Results in Section~\ref{7} for instance should be interpreted along these lines. Any 3-dimensional flow can be realized on the normally repelling 3-dimensional invariant submanifold $\Lambda_s \subset \mathbb{R}\times M$. Start with any vector field $X$ on $M$ and by adding enough expansion in the $\partial_s$-direction, extend it to a vector field on $\mathbb{R}\times M$, normally hyperbolic with restriction to $X$ on $\{ 0\}\times M$. After a positive reparametrization, we can set the divergence to be any prescribed function. However, Theorem~\ref{pers} shows that doing this as Liouville dynamics has implications on the skeleton dynamics. In particular, the skeleton dynamics is axiom A and has a non-empty hyperbolic invariant subset. One can ask about other phenomena distinguishing the Liouville condition for flow. In particular, one can ask if the generic properties of Liouville vector fields is the same as general vector fields.

\begin{question}\label{qgeneral}
How are Liouville flows qualititatively different from a general flow with positive divergence? In particular, what are the generic properties of Liouville flows?
\end{question}

To motivate the question about the generic properties above, recall that in Theorem~\ref{almosth}, we showed that the expansion of the normal bundle of a $C^1$ 3-dimensional Liouville skeleton $\Lambda$ and assuming $\ker{\alpha}\pitchfork T\Lambda$ implies that the skeleton dynamics to be a non-singular partially hyperbolic flow. $\Lambda$ being the Liouville skeleton implies topological repelling nearby. But in Theorem~\ref{almosth}, $C^1$-repellence is assumed. Now, in the category of general vector fields (even allowing restriction to flow with positive divergenc,e if desired), $C^1$-repellence can be arranged generically, since arbitrarily small normal repellence at $\Lambda$ can be added to guarantee $C^1$-repellence. However, such trivial perturbation does not work in the category of Liouville flows, since by Lemma~\ref{moserclassic} any deformation of the Liouville dynamics is done, after applying an isotopy, by adding a Hamiltonian, and hence volume preserving, vector field. Our modification of adding normal repellence at $\Lambda$ is clearly not volume preserving. Therefore, it remains unclear if such condition is generic in the category of Liouville flows.

\begin{question}\label{qgenrep}
Is $C^1$-repellence at $C^1$ 3-dimensional Liouville skeletons generic among Liouville flows?
\end{question}

As discussed further in Section~\ref{7}, thanks to the classic work of Hirsch-Pugh-Shub \cite{hps} and Mañe \cite{mane}, normal hyperbolicity and $C^1$-persistence of invariant submanifolds are known to be equivalent in the category of general flows. It remains unclear if $C^1$-persistence in the categories of Liouville flows is any different than the category of general (volume expanding) flows. Note that $C^1$-persistence implies $C^1$-Liouville persistence. So, the question is whether in converse is also true. More specifically, is it possible to have a $C^1$ skeleton, which $C^1$-persists under the ($C^1$- or $C^\infty$-) Liouville deformation of the generating Liouville vector field (i.e. $C^1$-deformation through Liouville vector fields), which is not normally hyperbolic in the sense of Theorem~\ref{normalh}?

\begin{question}\label{qlioupers}
Is there any difference between $C^1$-persistence and $C^1$-Liouville persistence?
\end{question}

Theorem~\ref{pers} gives a description of the skeleton dynamics in the case of normally hyperbolic 3-dimensional Liouville skeleton $\Lambda$. However, we currently only know of examples with $\ker{\alpha}\pitchfork T\Lambda$, i.e. the Mitsumatsu examples.

\begin{question}\label{qtrans}
Are there other examples of non-Weinstein Liouville geometry with 3-dimensional $C^1$-embedded oriented normally hyperbolic skeleton $\Lambda$, where $\ker{\alpha}$ and $T\Lambda_s$ are not transverse? 
\end{question}

The construction of Mitsumatsu implies that the deformation of flows through non-singular partially hyperbolic flows results in the homotopy of the underlying Liouville domain (or manifold). Considering that the symplectic geometric invariants of Anosov flows introduced in \cite{clmm} are invariants of Liouville homotopy, any two Anosov flows which can be homotoped through non-singular partially hyperbolic flows, have the same Liouville geometric invariants. Ideally, one would hope for these Liouville geometric invariants to determine the supported flow uniquely, up to orbit equivalence. The example of bi-contact DA deformation discussed in Section~\ref{5} shows that this does not hold in the category of non-singular partially hyperbolic flows, as one can construct two non-orbit equivalent flows, whose resulting LIS s are Liouville homotopic. But it is natural to ask about the situation in the category of Anosov flows. Therefore, natural to ask the following. 

\begin{question}\label{qinv}
Show that if two Anosov flows are connected through non-singular partially hyperbolic flows, then they are orbit equivalent.
\end{question}

Moreover, we only study non-vanishing Liouville forms with $C^1$-peristent Liouville skeletons. One can try to extend our results by allowing the Liouville form vanish at some points (necessarily on the skeleton).

\begin{question}\label{qsing}
Provide structure theorems for singular Liouville flows with $C^1$-persistent 3-dimensional Liouville skeleton.
\end{question}

As in the theory of general flows, studying $C^0$-persistence is typically much more complicated \cite{floer1,floer2}. Our LIS examples for non-Anosov partially hyperbolic 3-flows give examples of $C^0$-but-not-$C^1$-persistent skeletons (note that this is persistence in the category of general vector fields). We can ask whether any characterization similar to Theorem~\ref{pers} is possible for the skeleton dynamics in the case of $C^0$ persistent Liouville skeletons.

\begin{question}\label{qc0}
Can we characterize Liouville domains (or manifolds) with $C^0$-persistent 3-dimensional skeletons (possibly in terms of partial hyperbolicity)?
\end{question}

Finally, in Section~\ref{7.2} we argue that, for non-singular partially hyperbolic dynamics on an oriented 3-dimensional Liouville skeleton, which is sufficiently close to being Anosov in the sense of Theorem~\ref{nearanosov}, the Liouville flow is topologically conjugate to its linearization, similar to any Liouville flow induced from an LIS. However, dropping the hyperbolicity condition, other types of dynamics near the partially hyperbolic skeleton are a priori possible. We ask whether examples of such {\em exotic} Liouville dynamics in fact exist.

\begin{question}\label{qexotic}
Are there examples of exotic Liouville dynamics near 3-dimensional $C^1$-embedded Liouville skeleton?
\end{question}


\Addresses

\section{References}

\begin{thebibliography}{00}

\bibitem{anosov0} D. V. Anosov, {\em Ergodic properties of geodesic flows on closed Riemannian manifolds of negative curvature}, Dokl. Akad. Nauk SSSR, 151:6 (1963), 1250–1252

\bibitem{anosov} Anosov, Dmitry Victorovich. {\em Geodesic flows on closed Riemannian manifolds of negative curvature.} Trudy Mat. Inst. Steklov 90.5 (1967).

\bibitem{three} Araújo, Vítor, Maria José Pacifico, and Marcelo Viana. {\em Three-dimensional flows.} Berlin: Springer, 2010.

\bibitem{asaoka} Asaoka, Masayuki. {\em Regular projectively Anosov flows on three-dimensional manifolds.} Annales de l'Institut Fourier. Vol. 60. No. 5. 2010.

\bibitem{adn} Asaoka, Masayuki, Emmanuel Dufraine, and Takeo Noda. {\em Homotopy classes of total foliations.} Commentarii Mathematici Helvetici 87.2 (2012): 271-302.

\bibitem{bartintro} Barthelmé, Thomas. {\em Anosov flows in dimension 3 preliminary version.} Preprint (2017).

\bibitem{bby} Béguin, François, Christian Bonatti, and Bin Yu. {\em Building Anosov flows on 3–manifolds.} Geometry \& Topology 21.3 (2017): 1837-1930.

\bibitem{bowden} Bowden, Jonathan.{\em Exactly fillable contact structures without Stein fillings.} Algebraic \& Geometric Topology 12.3 (2012): 1803-1810.

\bibitem{bbp} Bowden, Jonathan, Christian Bonatti, and Rafael Potrie. {\em Some remarks on projective Anosov flows in hyperbolic 3-manifolds.} (2020): 359-369.

\bibitem{sing} Braddell, Roisin, et al. {\em An invitation to singular symplectic geometry.} International Journal of Geometric Methods in Modern Physics 16.supp01 (2019): 1940008.

\bibitem{folded} Breen, Joseph. {\em Folded symplectic forms in contact topology.} Journal of Geometry and Physics 201 (2024): 105213.

\bibitem{breen} Breen, Joseph. {\em Morse-Smale characteristic foliations and convexity in contact manifolds.} Proceedings of the American Mathematical Society 149.9 (2021): 3977-3989.

\bibitem{bc} Breen, Joseph, and Austin Christian. {\em Torus bundle Liouville domains are stably Weinstein.} arXiv preprint arXiv:2109.07615 (2021).

\bibitem{bhh} Breen, Joseph, Ko Honda, and Yang Huang. {\em The Giroux correspondence in arbitrary dimensions.} arXiv preprint arXiv:2307.02317 (2023).

\bibitem{c0} Capiński, Maciej J., and Hieronim Kubica. {\em Persistence of normally hyperbolic invariant manifolds in the absence of rate conditions.} Nonlinearity 33.9 (2020): 4967.

\bibitem{jul} Chaidez, Julian. {\em Robustly non-convex hypersurfaces in contact manifolds.} arXiv preprint arXiv:2406.05979 (2024).

\bibitem{gradient} Cieliebak, Kai. {\em A note on gradient-like vector fields.} arXiv preprint arXiv:2406.02985 (2024).

\bibitem{ce} Cieliebak, Kai, and Yakov Eliashberg. {\em From Stein to Weinstein and back: symplectic geometry of affine complex manifolds.} Vol. 59. American Mathematical Soc., 2012.

\bibitem{clmm} Cieliebak, K., Lazarev, O., Massoni, T., \& Moreno, A. (2022). {\em Floer theory of Anosov flows in dimension three.} arXiv preprint arXiv:2211.07453.

\bibitem{colin} Colin, Vincent, and Sebastiao Firmo. {\em Paires de structures de contact sur les variétés de dimension trois.} Algebraic \& Geometric Topology 11.5 (2011): 2627-2653.

\bibitem{pot} Crovisier, Sylvain, and Rafael Potrie. "Introduction to partially hyperbolic dynamics." School on Dynamical Systems, ICTP, Trieste 3.1 (2015).

\bibitem{partition} Dieck, Tammo Tom. {\em Partitions of unity in homotopy theory.} Compositio Mathematica 23.2 (1971): 159-167.

\bibitem{doering} Doering, Clauss. {\em persistently transitive vector fields on three-dimensional manifolds, Dynamical Systems and Bifurcation Theory.} Pitman Res. Notes Math. Ser. 160 (1987): 59-89.

\bibitem{eoy} Eliashberg, Yakov, Noboru Ogawa, and Toru Yoshiyasu. {\em Stabilized convex symplectic manifolds are Weinstein.} Kyoto Journal of Mathematics 61.2 (2021): 323-337.

\bibitem{ep} Eliashberg, Yakov, and Dishant Pancholi. {\em Honda-Huang's work on contact convexity revisited.} arXiv preprint arXiv:2207.07185 (2022).

\bibitem{et} Eliashberg, Yakov, and William P. Thurston. {\em Confoliations.} Vol. 13. American Mathematical Soc., 1998.

\bibitem{hyp} Fisher, Todd, and Boris Hasselblatt. {\em Hyperbolic flows.} 2019.

\bibitem{floer1} Floer, Andreas. {\em A topological persistence theorem for normally hyperbolic manifolds via the Conley index.} Transactions of the American Mathematical Society 321.2 (1990): 647-657.

\bibitem{floer2} Floer, Andreas. {\em A refinement of the Conley index and an application to the stability of hyperbolic invariant sets.} Ergodic Theory and Dynamical Systems 7.1 (1987): 93-103.

\bibitem{da} Franks, John, and Bob Williams. {\em Anomalous anosov flows.} Global Theory of Dynamical Systems: Proceedings of an International Conference Held at Northwestern University, Evanston, Illinois, June 18–22, 1979. Berlin, Heidelberg: Springer Berlin Heidelberg, 2006.

\bibitem{contop} Geiges, Hansjörg. {\em An introduction to contact topology.} Vol. 109. Cambridge University Press, 2008.

\bibitem{geiges} Geiges, Hansjörg. {\em Examples of Symplectic 4‐Manifolds with Disconnected Boundary of Contact Type.} Bulletin of the London Mathematical Society 27.3 (1995): 278-280.

\bibitem{geiges2} Geiges, Hansjörg. {\em Symplectic manifolds with disconnected boundary of contact type.} International Mathematics Research Notices 1994.1 (1994): 23-30.

\bibitem{ghys} Ghys, Étienne. {\em Flots d'Anosov dont les feuilletages stables sont différentiables.} Annales scientifiques de l'Ecole normale supérieure. Vol. 20. No. 2. 1987.

\bibitem{convex} Giroux, Emmanuel. {\em Convexité en topologie de contact.} Diss. Lyon 1, 1991.

\bibitem{giroux} Giroux,Emmanuel. {\em Geometrie de contact: de la dimension trois vers les dimensions superieures.} Proceedings of the international congress of mathematicians, ICM 2002, 2002, Beijing, Chine. pp.405-414.

\bibitem{ideal} Giroux, Emmanuel. {\em Ideal Liouville Domains-a cool gadget.} arXiv preprint arXiv:1708.08855 (2017).

\bibitem{gour} Gourmelon, Nikolaz. {\em Adapted metrics for dominated splittings.} Ergodic Theory and Dynamical Systems 27.6 (2007): 1839-1849.

\bibitem{gromov} Gromov, Mikhael. {\em Pseudo holomorphic curves in symplectic manifolds.} Inventiones mathematicae 82.2 (1985): 307-347.

\bibitem{poiss} Guillemin, Victor, Eva Miranda, and Ana Rita Pires. {\em Symplectic and Poisson geometry on b-manifolds.} Advances in mathematics 264 (2014): 864-896.

\bibitem{regp} Hasselblatt, Boris. {\em Periodic bunching and invariant foliations.} Mathematical Research Letters 1.5 (1994): 597-600.

\bibitem{reg} Hasselblatt, Boris. {\em Regularity of the Anosov splitting and of horospheric foliations.} Ergodic Theory and Dynamical Systems 14.4 (1994): 645-666.

\bibitem{reg2} Hasselblatt, Boris. {\em Regularity of the Anosov splitting II.} Ergodic Theory and Dynamical Systems 17.1 (1997): 169-172.

\bibitem{hps} Hirsch, Morris W., Charles Chapman Pugh, and Michael Shub. {\em Invariant manifolds.} Bulletin of the American Mathematical Society 76.5 (1970): 1015-1019.

\bibitem{holmes} Holmes, R. B. {\em A formula for the spectral radius of an operator.} The American Mathematical Monthly 75.2 (1968): 163-166.

\bibitem{honda} Honda, Ko, and Yang Huang. {\em Convex hypersurface theory in contact topology.} arXiv preprint arXiv:1907.06025 (2019).

\bibitem{hoz5} Hozoori, Surena. {\em Anosov contact metrics, Dirichlet optimization and entropy.} arXiv preprint arXiv:2311.15397 (2023).

\bibitem{hoz4} Hozoori, Surena. {\em On Anosovity, divergence and bi-contact surgery.} Ergodic Theory and Dynamical Systems (2022): 1-23.

\bibitem{hoz3} Hozoori, Surena. {\em Symplectic geometry of Anosov flows in dimension 3 and bi-contact topology.} Advances in Mathematics 450 (2024): 109764.

\bibitem{huang} Huang, Yang. {\em A dynamical construction of Liouville domains.} Proceedings of the American Mathematical Society 148.12 (2020): 5323-5330.

\bibitem{katok} Hurder, Steve, and Anatoly Katok. {\em Differentiability, rigidity and Godbillon-Vey classes for Anosov flows.} Publications Mathématiques de l'IHÉS 72 (1990): 5-61.

\bibitem{mane} Mañé, Ricardo. {\em Persistent Manifolds Are Normally Hyperbolic.} Transactions of the American Mathematical Society, vol. 246, 1978, pp. 261–83. JSTOR, https://doi.org/10.2307/1997974. Accessed 24 Apr. 2024.

\bibitem{massoni} Massoni, Thomas. {\em Anosov flows and Liouville pairs in dimension three.} arXiv preprint arXiv:2211.11036 (2022).

\bibitem{massoni2} Massoni, Thomas. {\em Taut foliations and contact pairs in dimension three.} arXiv preprint arXiv:2405.15635 (2024).

\bibitem{massot} Massot, Patrick. {\em Topological methods in 3-dimensional contact geometry.} Contact and symplectic topology 26 (2014): 27-83.

\bibitem{mnw} Massot, Patrick, Klaus Niederkrüger, and Chris Wendl. {\em Weak and strong fillability of higher dimensional contact manifolds.} Inventiones mathematicae 192.2 (2013): 287-373.

\bibitem{mcduff} McDuff, Dusa. {\em Symplectic manifolds with contact type boundaries.} Inventiones mathematicae 103.1 (1991): 651-671.

\bibitem{symptop} McDuff, Dusa, and Dietmar Salamon. {\em Introduction to symplectic topology.} Vol. 27. Oxford University Press, 2017.

\bibitem{mitsumatsu} Mitsumatsu, Yoshihiko. {\em Anosov flows and non-Stein symplectic manifolds.} Annales de l'institut Fourier. Vol. 45. No. 5. 1995.

\bibitem{mit2} Mitsumatsu, Yoshihiko, Daniel Peralta-Salas, and Radu Slobodeanu. {\em On the existence of critical compatible metrics on contact $3 $-manifolds.} arXiv preprint arXiv:2311.15833 (2023).

\bibitem{pa} Morales, Carlos A., Maria José Pacifico, and Enrique R. Pujals. {\em Robust transitive singular sets for 3-flows are partially hyperbolic attractors or repellers.} Annals of mathematics (2004): 375-432.

\bibitem{mori2} Mori, Atsuhide. {\em On the violation of Thurston-Bennequin inequality for a certain non-convex hypersurface.} arXiv preprint arXiv:1111.0383 (2011).

\bibitem{mori1} Mori, Atsuhide. {\em Reeb foliations on $S^5$ and contact 5-manifolds violating the Thurston-Bennequin inequality.} arXiv preprint arXiv:0906.3237 (2009).

\bibitem{folk} MION-MOUTON, M. A. R. T. I. N. {\em DISTRIBUTIONS INVARIANT BY PARTIALLY HYPERBOLIC DIFFEOMORPHISMS AND ANOSOV FLOWS.} (2024).

\bibitem{noda} Noda, Takeo. {\em Projectively Anosov flows with differentiable (un) stable foliations.} Annales de l'institut Fourier. Vol. 50. No. 5. 2000.

\bibitem{noda2} Noda, Takeo. {\em Regular projectively Anosov flows with compact leaves.} Annales de l'Institut Fourier, Volume 54 (2004) no. 2, pp. 481-497.

\bibitem{patreg} Paternain, Gabriel P. {\em Regularity of weak foliations for thermostats.} Nonlinearity 20.1 (2006): 87.



 
\bibitem{ps} Pugh, Charles, and Michael Shub. {\em Linearization of normally hyperbolic diffeomorphisms and flows.} Inventiones mathematicae 10.3 (1970): 187-198.

\bibitem{pa2} Pujals, Enrique. {\em From hyperbolicity to dominated splitting.} Inst. de Matemática Pura e Aplicada, 2006.

\bibitem{sal2} Salmoiraghi, Federico. {\em Goodman surgery and projectively Anosov flows.} arXiv preprint arXiv:2202.01328 (2022).

\bibitem{sal1} Salmoiraghi, Federico. {\em Surgery on Anosov flows using bi-contact geometry.} arXiv preprint arXiv:2104.07109 (2021).

\bibitem{simic} Simić, Slobodan. {\em Codimension one Anosov flows and a conjecture of Verjovsky.} Ergodic Theory and Dynamical Systems 17.5 (1997): 1211-1231.

\bibitem{weinstein} Weinstein, Alan. {\em Contact surgery and symplectic handlebodies.} Hokkaido Mathematical Journal 20.2 (1991): 241-251.

\end{thebibliography}
\end{document}